\pdfmapfile{=<name>.map}
\documentclass[12pt,oneside,reqno,fleqn]{amsart}
\usepackage{amssymb}
\usepackage{bbm}
\usepackage{cases}
\usepackage{amsmath,mathtools}
\usepackage{graphicx}
\usepackage{mathrsfs}
\usepackage{stmaryrd}
\usepackage{color}
\usepackage{soul}
\usepackage[dvipsnames]{xcolor}
\usepackage{amsfonts}
\usepackage{amsmath,amssymb,amsthm}
\usepackage{libertine}
\usepackage[T1]{fontenc}
\usepackage[libertine,varbb]{newtxmath}
\usepackage[colorlinks,linkcolor=Red,citecolor=Red,urlcolor=Red]{hyperref}
\usepackage[shortlabels] {enumitem}
\usepackage[nameinlink,capitalize]{cleveref}
\usepackage[spacing=true,kerning=true,tracking=true,letterspace=50]{microtype}
\usepackage{orcidlink}
\usepackage{ulem}
\usepackage[disable]{todonotes}
\hypersetup{pdftitle={Quantitative approximation of stochastic kinetic equations: from discrete to continuum}, pdfauthor={Zimo Hao, Khoa L\^e and Chengcheng Ling}}
\numberwithin{equation}{section}
\newcommand{\be}{\begin{eqnarray}}
\newcommand{\ee}{\end{eqnarray}}
\newcommand{\ce}{\begin{eqnarray*}}
\newcommand{\de}{\end{eqnarray*}}
\newtheorem{theorem}{Theorem}[section]
\newtheorem{lemma}[theorem]{Lemma}
\newtheorem{proposition}[theorem]{Proposition}
\newtheorem{corollary}[theorem]{Corollary}
\theoremstyle{remark}
\newtheorem{assumption}[theorem]{Assumption}

\newtheorem{example}[theorem]{Example}
\newtheorem{remark}[theorem]{Remark}
\newtheorem{definition}[theorem]{Definition}

\crefname{eqn}{Equation}{Equations}
\crefname{assumption}{Assumption}{Assumptions}
\crefname{innercustomthm}{Condition}{Conditions}
\crefrangelabelformat{innercustomthm}{#3#1#4-#5#2#6}

\def\bbk{{\boldsymbol{k}}}
\def\bbp{{\boldsymbol{p}}}
\def\bbr{{\boldsymbol{r}}}
\def\bbq{{\boldsymbol{q}}}
\def\bba{{\boldsymbol{a}}}

\def\bb2{{\boldsymbol{2}}}
\def\no{\nonumber}
\def\={&\!\!=\!\!&}

\def\e{{\mathrm{e}}}
\def\eps{\varepsilon}

\def\p{\partial}

\def\<{{\langle}}
\def\>{{\rangle}}
\def\({{\Big(}}
\def\){{\Big)}}

\def\bx{{\mathbf{x}}}

\def\dif{d}

\def\min{{\mathord{{\rm min}}}}

\def\no{\nonumber}
\def\={&\!\!=\!\!&}
\def\bt{\begin{theorem}}
\def\et{\end{theorem}}
\def\bl{\begin{lemma}}
\def\el{\end{lemma}}
\def\br{\begin{remark}}
\def\er{\end{remark}}
\def\bd{\begin{definition}}
\def\ed{\end{definition}}
\def\bp{\begin{proposition}}
\def\ep{\end{proposition}}
\def\bc{\begin{corollary}}
\def\ec{\end{corollary}}
\def\bx{\begin{example}}
\def\ex{\end{example}}

\def\cB{{\mathcal B}}

\def\cD{{\mathcal D}}

\def\cM{{\mathcal M}}
\newcommand\cmm{\cM}

\def\cR{{\mathcal R}}
\def\cS{{\mathcal S}}

\def\mB{{\mathbb B}}

\def\mE{{\mathbb E}}
\def\E{\mE}

\def\mI{{\mathbb I}}

\def\mL{{\mathbb L}}

\def\mN{{\mathbb N}}

\def\mP{{\mathbb P}}
\def\mQ{{\mathbb Q}}
\def\mR{{\mathbb R}}

\def\sF{{\mathscr F}}

\def\sI{{\mathscr I}}

\def\sM{{\mathscr M}}

\def\sV{{\mathscr V}}
\def\bC{{\mathbb C}}
\def\bB{{\mathbb B}}

\def\geq{\geqslant}
\def\leq{\leqslant}

\def\DD{{b}}
\def\wt{\widetilde}
\def\bu{{\mathbf{u}}}

\newcommand{\R}{{\mathbb R}}
\newcommand{\Rd}{{\R^d}}
\newcommand{\tand}{\quad\text{and}\quad}
\newcommand{\LL}{{\mathbb{L}}}

\newcommand{\1}{{\mathbf 1}}

\newcommand{\norm}[1]{{\left\vert\kern-0.25ex\left\vert\kern-0.25ex\left\vert #1
    \right\vert\kern-0.25ex\right\vert\kern-0.25ex\right\vert}}

\renewcommand{\le}{\leq}
\renewcommand{\ge}{\geq}
\allowdisplaybreaks
\begin{document}
	\title{Quantitative approximation of stochastic kinetic equations: from discrete to continuum}
	\date{\today}
	
\author{Zimo Hao, Khoa L{\^e} 
\and Chengcheng Ling 
}

\address{Universit\"at Bielefeld, Fakult\"at f\"ur Mathematik, 33615 Bielefeld, Germany}
\email{zhao@math.uni-bielefeld.de}
\address{
University of Leeds, School of Maths,
Leeds, LS2 9JT, UK.}
\email{ K.Le@leeds.ac.uk}
\address{ University of Augsburg, Institut f\"ur Mathematik,
86159 Augsburg,  Germany.}
\email{ chengcheng.ling@uni-a.de}

	\begin{abstract}
We study the convergence of a generic tamed Euler-Maruyama (EM) scheme for the kinetic type stochastic differential equations (SDEs) (also known as second order SDEs) with singular coefficients in both weak and strong probabilistic senses. We show that when the drift exhibits a relatively low regularity compared to the state of the art, the singular system is well-defined both in the weak and strong  probabilistic senses. Meanwhile, the corresponding tamed EM scheme is shown to converge at the rate of $1/2$ in both the weak and the strong senses.
		
		\bigskip
		

	\end{abstract}
	
	\maketitle
\section{Introduction}

We consider the following SDE:
\begin{align}\label{eq:SDE}
\begin{cases}
\dif X_t=V_t\dif t\\
\dif V_t=b(t,Z_t)\dif t+\dif W_t,\quad Z_0=\zeta=(\xi,\eta)\in\sF_0
\end{cases}
\end{align}
where $Z_t=(X_t,V_t)\in\Rd\times\Rd$, $t\geq0$, $b :\mR_+\times\mR^{2d}\to\mR^d$ is measurable, and $W$ is a standard $d$-dimensional Brownian motion on the probability space $(\Omega,\sF,(\sF_s)_{s\ge0},\mP)$. \eqref{eq:SDE} is one of the typical models that describes the Hamiltonian mechanics in the form of Langevin equation (\cite{S,T}). In such context,  $Z_t=(X_t,V_t)$ usually denotes the position and velocity of a moving particle at a time $t$. In statistical mechanics,  \eqref{eq:SDE}
  leads to the following diffusion (or ''Fokker–Planck'') equation  which describes the evolution of  the probability density $u(t,x,v)$ of $Z_t$  (\cite{RE, V}):
 \begin{align}\label{eq:PDE-intro}
    \partial_tu=\Delta_vu +v\cdot\nabla_x u+b\cdot\nabla_vu.
 \end{align}
In the literature, it has been rather well-understood  that the
kinetic structure has good “propagation” properties from the $v$-component to the $x$-component (e.g. \cite{B,CZ, IS}). Moreover, notice that even though in \eqref{eq:SDE} the noise is degenerate because it acts directly only on the momentum coordinates and
not positions, it can also be shortly wrote as for $z=(x,v)\in\mathbb{R}^{2d}$
\begin{align}
    \label{eq:SDE-H}
    d Z_t=F(t, Z_t) dt+\Sigma(t,Z_t) dW_t,  F(t,z):= \left(\begin{array}{c} v\\ b(t,z) \end{array}\right)\in \mathbb{R}^{2d}, \Sigma(t,z):=\left(\begin{smallmatrix}0&0\\0&\mathbb{I}_{d\times d}
    \end{smallmatrix}\right)\in \mathbb{R}^{d\times d},
\end{align}
which falls into the setting that known H\"ormander condition (\cite{H}) holds. In this spirit, the study on the  well-posedness of \eqref{eq:SDE} and \eqref{eq:PDE-intro} with singular $b$  in parallel gains quite development: on the regularity  results for \eqref{eq:PDE-intro}, \cite{B, RE} gain the optimal Sobolev regularity when $b\equiv1$; later on for non-constant particularly singular $b$, various of regularity  estimates also are shown within series of work \cite{IS, CZ, Zhang2018, HWZ20, RM} (also the references inside); meanwhile due to the strong connection between \eqref{eq:SDE} and \eqref{eq:PDE-intro}, the well-posedness for \eqref{eq:SDE} is established for various types of singular $b$ correspondingly within  \cite{R,RM, CMPZ, FFPV,WZ, HZZZ22}. Along this process, we can see that the system \eqref{eq:SDE} concerning its well-posedness is gradually understood well for singular $b$ comparing with classical first order SDEs (\cite{Zv, Ver, KR}).\\

Due to the needs from simulation and application side, also from the observation that stochastic numerical analysis  for the first order singular SDEs obtains delightful results (\cite{BDG, DGL, LL,  MGY20, JM21, GLL, Holland}),  it is demanding to implement numerical methods to the second order singular SDEs as well, which however exist fewer. Let us first mention some of results on the one of simplest numerical approximation-- the Euler-Maruyama scheme (EM scheme for short) for the singular first order SDEs:  on the strong convergence rate (i.e. $\mE\sup_{t\in[0,T]}\|Z_t-Z_t^n\|$) \cite{BDG, DGL, LL,  MGY20, NSz} show the known $n^{-\frac{1}{2}}$-optimal convergence rate, and \cite{JM21, Holland} show the weak convergence rate (i.e. $|\mE f(Z_T)-\mE f(Z_T^n)|$ for suitable test function $f$). According to our best knowledge, in the early stage \cite{MSH}  and \cite{T} study the time-discrete approximations (for instance, explicit and implicit Euler-Maruyama scheme) for \eqref{eq:SDE} under Lyapunov function type conditions which usually require certain regularity  on $b$. Later  \cite{LM}  obtains the non asymptotic bounds for the Monte Carlo algorithm associated to the Euler discretization for bounded $b$.  Recently \cite{LS18} that seems to be the first one obtains the quantitative strong convergence bound on the EM scheme with order $n^{-\frac{1}{4}}$ for $b$ being piecewise Lipschitz function. Evidently, comparing with  the study on EM scheme for singular first order SDEs, the quantitative analysis on singular second order SDEs is largely open.

In the current paper we are interested in giving the upper bound of the strong and weak EM scheme for \eqref{eq:SDE} under somehow the minimal
assumptions (see \cref{ass:main} and \cref{ass:main1}) that guarantee the SDE is correspondingly strongly and weakly well-posed.
Therefore we propose the following \emph{tamed Euler scheme} to approximate \eqref{eq:SDE}.
\subsection*{Proposed numerical scheme}
 For any $n\in\mN$, define $k_n(t):=\frac{\lfloor nt \rfloor}{n}, t\in \mR_+$.
Let $\{Z^n_t\}_{t\ge0}:=\{(X^n_t,V^n_t)\}_{t\ge0}$ be the solution to the following Euler scheme
\begin{align}\label{eq:SDE-EM}
\begin{cases}
X^n_t=\xi+\int_0^t V^n_s\dif s,\\
V^n_t=\eta+\int_0^t \Gamma_{s-k_n(s)}b_n(s,Z^n_{k_n(s)})\dif s+W_t,
\end{cases}
\end{align}
where for some parameter $\vartheta>0$ and a probability density function $\phi$, let
\begin{align}\label{def:bn}
    b_n:=b*\phi_n,\quad \phi_n(x,v):=n^{4d\vartheta}\phi(n^{3\vartheta}x, n^{\vartheta}v),
\end{align}
and for $t\in\mR$, define
\begin{align*}
\Gamma_tf(z):=f(\Gamma_t z):=f(x+tv,v),\quad z=(x,v).
\end{align*}
As we may observe, comparing with standard Euler-Maruyama scheme for the first order SDEs, we replace the classical Euler scheme $\int^t_0 b(s,Z^n_{k_n(s)})\dif s$ with the term $\int^t_0 \Gamma_{s-k_n(s)}b_n(s,Z^n_{k_n(s)})\dif s$.  There are mainly the following reasons that  $\int^t_0 \Gamma_{s-k_n(s)}b_n(s,Z^n_{k_n(s)})\dif s$ should be the considered one:

    Firstly,   we use tamed term $b_n$ instead of $b$ itself to guarantee that the solution sequence $Z_t^n$ to the discretized equation \eqref{eq:SDE-EM} is well-defined, which is similar to  the idea from \cite{HJK, LL} (known as ``taming technique"). Moreover,  the scheme with $\int^t_0 \Gamma_{s-k_n(s)}b_n(s,Z^n_{k_n(s)})\dif s$ provides a faster convergence rate which could be better for applications. We can illustrate this point easily: for any smooth function $f = f(x, v) = f(x)$ mapping $\mathbb{R}^{2d}$ to $\mathbb{R}$, 
    we analyze the following error: denote
\begin{align*}
I_n &:= \mathbb{E}\Big|f\Big(\int_0^t W_s\dif s, W_t\Big) - f\Big(\int_0^{k_n(t)} W_s\dif s, W_{k_n(s)}\Big)\Big|\\&= \mathbb{E}\Big|f\Big(\int_0^t W_s\dif s\Big) - f\Big(\int_0^{k_n(t)} W_s\dif s\Big)\Big|,
\end{align*}
clearly
\begin{align*}
I_n \lesssim \int_{k_n(t)}^t \mathbb{E}|W_s|\dif s \lesssim \int_{k_n(t)}^t s^{\frac{1}{2}}\dif s \lesssim n^{-1}.
\end{align*}
In contrast, if we replace $I_n$ with $J_n$ defined as
\begin{align*}
J_n &:= \mathbb{E}\Big|f\Big(\int_0^t W_s\dif s, W_t\Big) - \Gamma_{t-k_n(t)}f\Big(\int_0^{k_n(t)} W_s\dif s, W_{k_n(s)}\Big)\Big|\\
&=\mathbb{E}\Big|f\Big(\int_0^t W_s\dif s\Big) - f\Big(\int_0^{k_n(t)} W_s\dif s+t-k_n(t)W_{k_n(s)}\Big)\Big|,
\end{align*}
it yields
\begin{align*}
J_n \lesssim \mathbb{E}\Big|\int_{k_n(t)}^t (W_s - W_{k_n(s)})\dif s\Big| \lesssim \int_{k_n(t)}^t (s - k_n(s))^{\frac{1}{2}}\dif s \lesssim n^{-\frac{3}{2}},
\end{align*}
which shows that the tamed term could provide a convergence rate superior to $n^{-1}$.

      Secondly,  from one technical view point,
        given that in this paper we employ the semi-group to enhance the regularity of some singular functions like drift term $b$, which may not even be continuous, for the kinetic SDEs, in general we can not avoid to consider the semi-group with the following form (see also \eqref{def:g} and \eqref{CC01})
    \begin{align*}
        P_tf(x,v):=\mE f\left(x+tv+\int_0^t W_s\dif s,v+W_t\right)=\mE\Big(f\big(G_t(\Gamma_t(x,v))\big)\Big), G_t:=(\int_0^t W_sds, W_t).
    \end{align*}
   Comparing with the standard heat semi-group which relates to the heat equation,  the above one is the ``analog semi-group'' corresponding to  the kinetic equation \eqref{eq:PDE-intro} (\cite{HZZZ22, HRZ, DF, LM}).  It implies that we need to consider regularizing effect from $P_t$ for the singular functional with the form
    \begin{align}\label{0425:00}
        P_tf - \Gamma_{t-s}P_sf,\quad 0 < s < t,
    \end{align}
    rather than the classical form $P_tf-P_sf$ used for the singular first order SDEs (e.g. \cite{BDG,DGL,LL}). For further details, we refer to \cite{ZZ21, HZZZ22}. In this paper, we derive an estimate for
\begin{align*}
P_tf - P_s\Gamma_{t-s}f,
\end{align*}
which exhibits the same level of regularity as \eqref{0425:00} (see  \cref{rmk314} below).


\subsection*{Obtained convergence rate}
In the end, we are able to show both the weak  and strong convergence of the EM scheme with order $n^{-\frac{1}{2}}$ (see details in \cref{thm:main-rough}) under \cref{ass:main} and \cref{ass:main1}. It turns out to be the first new  quantitative bounds on the EM scheme for singular kinetic SDEs; the rate also is faster than limited known results from \cite{LM} and \cite{LS18}. We also want to mention the conjecture from \cite{LS18} which says the rate $n^{-1/4}$ obtained in \cite{LS18} is not optimal and expects better convergence rate. Our result can also be evident for this point.
{
Even for the non-degenerate case, it should be noted that in \cite{JM21}, the weak convergence rate for taming $L^qL^p$ drift is $(1-\frac{d}{p}-\frac{2}{q})/2$, which is called the "gap to singularity". Our weak convergence rate $1/2$ is independent of $p$ and gives a strong evidence that this gap to singularity can be removed.}
Meanwhile the assumptions are weaker than  the known results for the weak well-posedness (\cite{RM, LM}) and strong well-posedness (\cite{R, Zhang10}).\\

Here we roughly comment on the key ideas inside the proof.
\subsection*{Idea of the proof} \begin{enumerate}
    \item For quantitative weak error estimates, i.e. for the bounds of $|\mE f(Z_T)-\mE f(Z_T^n)|$, where $f$ is just being bounded (again this is singular), by using It\^o's formula to $r\to P_{t-r}\varphi(Z_r)$ and $P_{t-r}\varphi(Z^n_r)$, it reads (see also details from \eqref{0419:00})
\begin{align*}
    |\<\varphi,\rho_t-\rho^n_t\>|\le& \Big|\mE\int_0^t (b-b_n)(r,Z_r)\cdot \nabla_vP_{t-r}\varphi(Z_r)\dif r\Big|\no\\
    &+\Big|\mE\int_0^t b_n(r)\cdot \nabla_vP_{t-r}\varphi(Z_r)-b_n(r)\cdot \nabla_vP_{t-r}\varphi(Z_r^n)\dif r\Big|\no\\
    &+\Big|\mE\int_0^t \big(b_n(r,Z^n_r)-\Gamma_{r-k_n(r)}b_n(r,Z_{k_n(r)}^n)\big)\cdot \nabla_vP_{t-r}\varphi(Z_r^n)\dif r\Big|\no\\
    =:&I_1^n(t)+I_2^n(t)+I_3^n(t).
\end{align*}
 Since in this case $b$ and $\phi$ both have low regularity, known estimates can not be applied.  We need to fine estimate $I_1$, $I_2$ and $I_3$ separately.  Inside the proof, we get the desired bound of $I_2$ and $I_3$ via applying the regularization properties of kinetic semigroup $(P_t)_{t\geq0}$ collected in \cref{sec:Tools}, estimates on $I_1$ further heavily relies on the paraproduct analysis.
\item For the strong convergence rate, we first write
\begin{align*}
    Z_t-Z_t^n=\begin{pmatrix}\int_0^tV_s-V_s^n\dif s\\\int_0^t b(s,Z_s)-\Gamma_{s-k_n(s)}b_n(s,Z^n_{k_n(s)})\dif s
    \end{pmatrix}=:\begin{pmatrix}I_t^n\\S_t^n
    \end{pmatrix}.
\end{align*}
For the bound of $\mE\sup_{t\in[0,T]}|Z_t-Z_t^n|^m$, $m\geq1$,  the main task  is to estimate $\mE\sup_{t\in[0,T]}|S_t^n|^m$ since the linear $I_t^n$ term is  ready for the form of applying Gr\"onwall's inequality. Via John-Nirenberg inequality (\cite{Le2022}) together with Girsanov transformation and Zvonkin transformation,  it is then reduced to the estimate  on
\begin{align*}
\Big\|\sup_{t\in(0,1)}\big|\int_0^t\nabla_v U(r,M_r(z))\(\Gamma_{r-k_n(r)}b(r,M_{k_n(r)}(z))&-b(r,M_r(z))\)\dif r\big|\Big\|_{L^m(\Omega)}
\end{align*}
where  $M_t(z):=(x+tv+\int_0^tW_s\dif s,v+W_t)$, $\forall z=(x,v)\in\mR^{2d}$, $U$ is the solution to \eqref{eq:PDE-intro} with extra forcing field $b^n$ (so call \emph{Zvonkin transformation} method). The bound $n^{-\frac{1}{2}}$ is achieved by applying \emph{stochastic sewing lemma} (\cite{Le}) together with regularity estimates involved with \eqref{0425:00}, which also yields the desired strong convergence rate $n^{-\frac{1}{2}}$ after applying the Gr\"onwall's inequality in the end.

\end{enumerate}

\subsection*{Organization of the paper}
We present the frequently used notations and main results in \cref{sec:Notation-result}. We collect the tools  in \cref{sec:Tools}.  \cref{sec:Strong-Weak-Stable} discusses the weak and strong well-posedness, stability  of \eqref{eq:SDE}, which is new and has its own interests.  We give the auxiliary quantitative estimates on singular functional in \cref{sec:quantitative-est} and the final proofs for the main results are presented in \cref{sec:weak-strong-con}.  \cref{app} collects several technical lemmas for calculations appeared in the main proofs.
\section{Notations and main results}\label{sec:Notation-result}
\subsection{Notations}
Here we collect all of the notations that we quite often use in the whole text. We start with introducing the anisotropic Besov spaces.
\subsection*{Anisotropic Besov spaces} \label{sec.AS}

For an $L^1$-integrable function $f$ in $\mR^{2d}$, let $\hat f$ be the Fourier transform of $f$ defined by
$$
\hat f(\xi):=(2\pi)^{-d}\int_{\mR^{2d}} \e^{-{\rm i}\xi\cdot z}f(z)\dif z, \quad\xi\in\mR^{2d},
$$
and $\check f$ the inverse Fourier   transform of $f$ defined by
$$
\check f(z):=(2\pi)^{-d}\int_{\mR^{2d}} \e^{{\rm i}\xi\cdot z}f(\xi)\dif\xi, \quad z\in\mR^{2d}.
$$
Let $\bba=(3,1)$. For $z=(x,v)$, $z'=(x',v')\in\mR^{2d}$, we introduce the anisotropic distance
$$
|z-z'|_\bba:=|x- x'|^{1/3}+|v-v'|.
$$
Note that $z\mapsto|z|_{\bba}$ is not smooth.
For $r>0$ and $z\in\mR^{2d}$, we also introduce the ball centered at $z$ and with radius $r$ with respect to the above distance
as follows:
$$
B^\bba_r(z):=\{z'\in\mR^{2d}:|z'-z|_\bba\leq r\},\ \ B^\bba_r:=B^\bba_r(0).
$$
Let $\chi^\bba_0$ be  a symmetric $C^{\infty}$-function  on $\mR^{2d}$ with
$$
\chi^\bba_0(\xi)=1\ \mathrm{for}\ \xi\in B^\bba_1\ \mathrm{and}\ \chi^\bba_0(\xi)=0\ \mathrm{for}\ \xi\notin B^\bba_2.
$$
We define
$$
\phi^\bba_j(\xi):=
\left\{
\begin{aligned}
&\chi^\bba_0(2^{-j\bba}\xi)-\chi^\bba_0(2^{-(j-1)\bba}\xi),\ \ &j\geq 1,\\
&\chi^\bba_0(\xi),\ \ &j=0,
\end{aligned}
\right.
$$
where for $s\in\mR$ and $\xi=(\xi_1,\xi_2)$,
$$
2^{s\bba }\xi=(2^{3s}\xi_1, 2^{s}\xi_2).
$$
Note that
\begin{align}\label{Cx8}
{\rm supp}(\phi^\bba_j)\subset\big\{\xi: 2^{j-1}\leq|\xi|_\bba\leq 2^{j+1}\big\},\ j\geq 1,\ {\rm supp}(\phi^\bba_0)\subset B^\bba_2,
\end{align}
and
\begin{align}\label{AA13}
\sum_{j\geq 0}\phi^\bba_j(\xi)=1,\ \ \forall\xi\in\mR^{2d}.
\end{align}

Let $\cS$ be the space of all Schwartz functions on $\mR^{2d}$ and $\cS'$ the dual space of $\cS$, called the tempered distribution space.
For given $j\geq 0$, the  dyadic anisotropic block operator  $\mathcal{R}^\bba_j$ is defined on $\cS'$ by
\begin{align}\label{Ph0}
\mathcal{R}^\bba_jf(z):=(\phi^\bba_j\hat{f})\check{\ }(z)=\check{\phi}^\bba_j*f(z),
\end{align}
where the convolution is understood in the distributional sense and by scaling,
\begin{align}\label{SX4}
\check{\phi}^\bba_j(z)=2^{(j-1)4d}\check{\phi}^\bba_1(2^{(j-1)\bba}z),\ \ j\geq 1.
\end{align}
For $j\in\mN$, by definition it is easy to see that
\begin{align}\label{KJ2}
\cR^\bba_j=\cR^\bba_j\widetilde\cR^\bba_j,\ \mbox{ where }\ \widetilde\cR^\bba_j:=\sum_{|i-j|\leq 2}\cR^\bba_i,
\end{align}
where we used the convention that $\cR^\bba_i:=0$ for $i<0$. Moreover,
by the symmetry of $\phi^\bba_j$,
$$
\<\cR^\bba_j f,g\>=\< f,\cR^\bba_jg\>,\ \ f\in\cS', \ g\in\cS.
$$
Similarly, we can define the isotropic block operator $\cR_jf=\check\phi_j*f$ in $\mR^d$, where
\begin{align}\label{Cx9}
{\rm supp}(\phi_j)\subset\big\{\xi: 2^{j-1}\leq|\xi|\leq 2^{j+1}\big\},\ j\geq 1,\ {\rm supp}(\phi_0)\subset B_2.
\end{align}
Now we introduce the following anisotropic Besov spaces (cf. \cite[Chapter 5]{Tri06}).
\begin{definition}\label{bs}
Let $s\in\mR$, $q\in[1,\infty]$ and $\bbp\in[1,\infty]^2$. The  anisotropic Besov space is defined by
$$
\mathbf{B}^{s,q}_{\bbp;\bba}:=\left\{f\in \cS': \|f\|_{\mathbf{B}^{s,q}_{\bbp;\bba}}
:= \left(\sum_{j\geq0}\big(2^{ js}\|\cR^\bba_{j}f\|_{\bbp}\big)^q\right)^{1/q}<\infty\right\},
$$
where $\|\cdot\|_\bbp$ is defined by
\begin{align}\label{LP1}
\|f\|_{\mL^\bbp}:=\|f\|_{\mL_z^\bbp}:=\|f\|_{\bbp}:=\left(\int_{\mR^d}\|f(\cdot,v)\|_{p_x}^{p_v}\dif v\right)^{1/p_v}.
\end{align}
In the following we also use the notation $\|f\|_{L_T^q(\mL_z^\bbp)}:=\left(\int_0^T\|f(t,\cdot)\|_{\mL^\bbp_z}^{q}\dif t\right)^{1/q}$ and $\|f\|_{\mL_T^\infty}:=\sup_{t\in[0,T]}\|f(t)\|_\infty$.
\end{definition}
Note that for any $f\in \bB^{s,q}_{\bbp;\bba}$, by \eqref{AA13} we have
\begin{align}\label{SE1}
f=\sum_{j\geq 0}\cR^\bba_j f\  \mbox{ in $\bB^{s,q}_{\bbp;\bba}$.}
\end{align}
The mixed Besov space defined as follows
 shall be used in the study of strong solutions of kinetic SDEs.
\begin{definition}\label{def:bs-mix}Let $(s_0,s_1)\in\mR^2$.  Define
$$
\bB^{s_0,s_1}_{\bbp;x,\bba}:=\left\{f\in \cS': \|f\|_{\mathbf{B}^{s_0,s_1}_{\bbp;x,\bba}}
:= \sup_{k,j\geq 0}2^{\frac{ks_0}{3}}2^{ js_1}\|\cR^x_k\cR^\bba_{j}f\|_\bbp<\infty\right\},
$$
where for a function $f:\mR^{2d}\to\mR$,
$$
\cR^x_kf(x,v):=\cR_k f(\cdot,v)(x).
$$
Moreover, we also define
$$
\bB^{s}_{\bbp;x}:=\bB^{s,\infty}_{\bbp;x}:=\left\{f\in \cS': \|f\|_{\mathbf{B}^{s}_{\bbp;x}}
:= \sup_{k\geq 0}2^{k s}\|\cR^x_kf\|_\bbp<\infty\right\}.
$$
Similarly, one defines the usual  isotropic Besov spaces  $\bB^{s,q}_{p}$
in $\mR^d$ in terms of isotropic block operators $\cR_j$. If there is no confusion, we shall write
\begin{align}
    \label{def:Bes-pq}\bB^s_{\bbp;\bba}:=\bB^{s,\infty}_{\bbp;\bba},\ \ \bB^s_{p}:=\bB^{s,\infty}_{p}.
\end{align}
\end{definition}
\begin{definition}\label{def:bs-Ho}
For a function $f:\mR^{2d}\to\mR$,  the first-order difference operator is defined by
$$
\delta^{(1)}_hf(z):=\delta_hf(z):=f(z+h)-f(z),\ \ z, h\in\mR^{2d}.
$$
For $M\in\mN$, the $M$-order difference operator  is defined recursively by
$$
\delta^{(M+1)}_hf(z)=\delta_h\circ\delta^{(M)}_hf(z).
$$
\end{definition}
The relation among the introduced Besov spaces and differential operator can be illustrated via the following  known results (cf. \cite{ZZ21} and \cite[Theorem 2.7]{HZZZ22}).
\bp
For $s>0$, $q\in[1,\infty]$ and $\bbp\in[1,\infty]^2$, an equivalent norm of $\bB^{s,q}_{\bbp;\bba}$ is given by
\begin{align}\label{CH1}
\|f\|_{\bB^{s,q}_{\bbp;\bba}}\asymp \left(\int_{\mR^{2d}}\left(\frac{\|\delta_h^{([s]+1)}f\|_{\bbp}}{|h|^s_\bba}\right)^q\frac{\dif h}{|h|^{4
d}_\bba}\right)^{1/q}+\|f\|_{\bbp}=:[f]_{\bB^{s,q}_{\bbp;\bba}}+\|f\|_{\bbp},
\end{align}
where $[s]$ is the integer part of $s$. In particular, $\bC^s_{\bba}:=\mathbf{B}^{s,\infty}_{\infty;\bba}$ is the anisotropic
H\"older-Zygmund space, and for $s\in(0,1)$, there is a constant $C=C(\alpha,d,s)>0$ such that
\begin{align*}
  \|f\|_{\bC^s_{\bba}} \asymp_C\|f\|_\infty+\sup_{z\not= z'}|f(z)-f(z')|/|z-z'|^{s}_\bba,
\end{align*}
Similarly,  the mixed and $x$-direction H\"older-Zygmund spaces are defined by
$$
\bC^{s_0,s_1}_{x,\bba}:=\bB^{s_0,s_1}_{\infty;  x,\bba},\ \ \bC^{s}_{x}:=\bB^{s}_{\infty;  x}.
$$
\ep

Let us also introduce the following total variation distance of distributions.
\begin{definition}
    \label{def:TV}
    Let $\mu_1$ and $\mu_2$  be probability distributions on $\mR^d$, then the total variation distance of $\mu_1$ and $\mu_2$ is defined as
    \begin{align*}
        \|\mu_1-\mu_2 \|_{\rm var}:=\sup_{\|f\|_\infty}|\mu_1(f)-\mu_2(f)|
    \end{align*}
\end{definition}
\subsection*{Some conventions}
On finite dimensional vector spaces we always use the Euclidean norm.

Without confusion we sometimes also use $<,=, \leq$ and  $>,\geq$ between vectors if their each component shares the same order, e.g. we say $\bbq:=(q_i)_{i\geq 1}\leq \bbp:=(p_i)_{i\geq 1}$ if $q_i\leq p_i$ for each $i$.

In proofs, the notation $a\lesssim b$ abbreviates the existence of $C>0$ such that $a\leq C b$, such that moreover $C$ depends only on the parameters claimed in the corresponding statement. If the constant depends on any further parameter $c$, we incorporate it in the notation by writing $a\lesssim_c b$.
\subsection{Main results}

We study the convergence scheme mainly based on the following two types of conditions.
\begin{assumption}[Weak] Let $\beta\in(0,1/3)$, $\bbp=(p_x,p_v)\in(2,\infty)^2$. We assume
    \label{ass:main}
    $b\in L^\infty_T(\bB^\beta_{\bbp;\bba})$ with $\bba\cdot d/\bbp<1$.
\end{assumption}
\begin{assumption}[Strong]
    \label{ass:main1}
  Let $\beta\in(0,1/3)$, $\bbp=(p_x,p_v)\in(2,\infty)^2$. We assume $b\in L^\infty_T(\bB^{\frac{2}{3},\beta}_{\bbp;x,\bba})$  with $\bba\cdot d/\bbp<1$.
\end{assumption}
In summary, our main results can be stated as follows. The detailed results are presented in \cref{thmW} (well-posedness result) and \cref{thm-weak} (weak convergence), \cref{thm-strong} (strong convergence) accordingly.
\begin{theorem}
    \label{thm:main-rough}
        Let $(Z_t)_{t\in[0,1]}$ and $(Z^n_t)_{t\in[0,1]}$  be the solutions corresponding to \eqref{eq:SDE} and \eqref{eq:SDE-EM}. Assume $0<\vartheta<(2\bba\cdot\frac{d}{\bbp})^{-1}$. 
    \begin{enumerate}[(i)]
        \item ({\bf Weak convergence})  If \cref{ass:main} holds and $\vartheta=\frac{1}{2}$,  there is a constant $C>0$ such that for $n\in\mN$
     \begin{align}
             \label{est:thm-EM-w}
          \int_0^1 \Big\|\mP\circ(Z^n_t)^{-1}-\mP\circ(Z_t)^{-1}\Big\|_{\rm var}^2\dif t&\lesssim_C  n^{-1};
         \end{align}
         \item  ({\bf Strong convergence}) If  \cref{ass:main1} holds, 
         then for $p:=\min(p_x,p_v)$, we have  for any $m\in[1,p)$
         \begin{align}
             \label{est:thm-EM-s}
             \big\|  \sup_{t\in[0,1]}|Z_t-Z_t^n|\big\|_{L^m(\Omega)}\lesssim_C n^{-\frac{1+\beta/3}{2}+\epsilon}+n^{-\vartheta(\beta+1-\bba\cdot\frac{d}{\bbp})+\epsilon}. 
         \end{align}
         where $\epsilon>0$ is some sufficiently small constant.
    \end{enumerate}
\end{theorem}
\subsection*{Discussion on the main results.}
Notice that \cref{thm:main-rough} not only contains the results on numerical approximations which are precisely given in \cref{thm-weak} and \cref{thm-strong}, but also new well-posedness results concerning the conditions on $b$ which are presented precisely in \cref{thmW}. In the following we list the known results summarily and compare these with ours.
\begin{enumerate}
    \item {\bf Well-posedness}
    \begin{enumerate}
        \item Weak well-posedness:  There is very limited work on the study of the weak well-posedness of  \eqref{eq:SDE}. Sharp conditions given from  \cite[Theorem 1]{RM}  are either $b\in L^\infty_T(C_{x,v}^\beta)$ with $ \beta\in(0,1)$ or $b\in L^\infty_T(\mL^\infty_z)$ or $L_T^q(\mL^p_z)$ with $\frac{2}{q}+\frac{4d}{p}<1$ and $p\geq 2,q>2$. \\Comparing with our result, instead of working in H\"older space, \cref{thmW} (i) provides a more general framework in anisotropic Besov spaces; moreover it also covers one of the conditions $b\in L^\infty_T(C_{x,v}^\beta)$ with $ \beta\in(0,1)$  and $L_T^q(\mL^p_z)$ from   \cite[Theorem]{RM}.
        \item Strong well-posedness: To our best knowledge,  the work \cite{R} seems to be the first one that studies the strong well-posedness of \eqref{eq:SDE} in the singular setting. \cite[Theorem 1.1]{R} shows that there exists a unique strong solution to \eqref{eq:SDE} when $b\in C_T(C_x^{\beta_x}C_v^{\beta_v})$ with  $\beta_x\in(\frac{2}{3},1)$ and $\beta_v\in(0,1)$. Later on \cite{WZ} reduces the H\"older condition  on $b$ from \cite[Theorem 1.1]{R} to H\"older-Dini’s continuity assumption.     \cite[Theorem 1.1]{Zhang2018} obtains the strong well-posedness  of  \eqref{eq:SDE} for the case with
critical differentiability indices $\beta_x=\frac{2}{3}$
and $\beta_v=0$, i.e. $\int_0^\infty \|(\mathbbm{I}_d-\Delta_x)^\frac{1}{3}b(s)\|_{\bbp}^p\dif s<\infty$ with $\bbp=(p,p)$ and $p>2(2d+1)$.
Notice that our assumption on the strong well-posedness \cref{thmW} (ii) is different from the setting of \cite{WZ} and \cite{Zhang2018} and weaker than \cite[Theorem 1.1]{R}.
    \end{enumerate}
    \item  {\bf Convergence }
    \begin{enumerate}
        \item  Weak convergence: Considering the fact that there is few result on weak well-posedness, it seems that our result \cref{thm-weak} is the first work for quantitative weak convergence rate $n^{-\frac{1}{2}}$ for such singular $b$. The work \cite{LM} also studies the law of the discrete EM equation but does not provide any quantitative bounds.
        \item Strong convergence: \cite{LS18} obtains the strong convergence rate $n^{-\frac{1}{4}}$ of EM scheme under the piecewise Lipschitz  condition on $b$. We can see that the rate  $n^{-\frac{1}{4}}$ is sub-optimal comparing with \cref{thm-strong}.
         \cite[Conclusion]{LS18} also conjectures the faster rate. Here \cref{thm-strong} gives better rate $n^{-\frac{1}{2}}$ by
 choosing proper parameter $\theta$ (i.e. fine tuned $(b_n)$). Notice that the piecewise Lipschitz  condition are not the same as the singularity we are considering.
     \end{enumerate}
     \item {\bf Future work}
     \begin{enumerate}
         \item  Recently there are many fruitful results on particle systems and mean filed equations appearing, see for instance \cite{DF, Ha23, HRZ}. Our one of the on-going works is to apply   \cref{thm:main-rough} to study the numerics of kinetic particle systems.
         \item
     We may see that the $\frac{2}{3}$-regularity of $b$ on $x$-variable assumed in \cref{ass:main1} does not contribute too much on the convergence scheme \eqref{est:thm-EM-s}, which may actually be possible to improve the rate. This point will be studied into more details in the future.
     \end{enumerate}
\end{enumerate}

\section{Tools}\label{sec:Tools}
\subsection{Stochastic sewing}
We use 
L\^e's stochastic sewing lemma (SSL) \cite{Le}
with a time shift from \cite[Lemma 2,9]{DGL}.\\
For $S<T$ and $i\in\mN$ we denote $[S,T]_{\leq}^i=\left\{(t_0,\ldots,t_{i-1})\in[S,T]^i:\,t_0\leq \cdots\leq t_{i-1}\right\}$
and $\widehat{[S,T]}_{\leq}^2=\left\{(s,t)\in[S,T]_{\leq}^2:\,|s-t|\leq S\right\}$.
Let
$(A_{s,t})_{(s,t)\in[S,T]_\leq^2}$ be a family of $\R^d$-valued random variables satisfying $A_{s,s}=0$, for all $s\in[S,T]$ and such that $A_{s,t}$ is $\mathcal{F}_t$-measurable for all $(s,t)\in[S,T]_{\leq}^2$. Let $p\in[2,\infty)$.   Let $(S,T)\in[0,1]_\leq^2$.

\begin{lemma}\label{lem:SSL-shift}
 Suppose that there exist constants $\Gamma_1,\Gamma_2\in[0,\infty)$, $\delta_1, \delta_2>0$,  $\beta_1>1/2$, and $\beta_2>1$  with $\beta_1-\delta_1>0$ and $\beta_2-\delta_2>0$ such that the following hold:
\begin{enumerate}[(1)]
    \item $\|A_{s,t}\|_{L^p(\Omega)}\leq \Gamma_1S^{-\delta_1}|t-s|^{\beta_1}$, $(s,t)\in\widehat{[S,T]}_{\leq}^2$,
    \item $\|\E_s\delta A_{s,u,t}\|_{L^p(\Omega)}\leq \Gamma_2S^{-\delta_2}|t-s|^{\beta_2}$, $(s,t)\in\widehat{[S,T]}_{\leq}^2, u\in[s,t]$.
\end{enumerate}
  Then there exists a unique (up to modification) $(\mathcal{F}_t)_{t\geq S}$-adapted process $\mathcal{A}:[S,T]\mapsto L^p(\Omega;\R^d)$  such that $\mathcal{A}_S=0$ and the following bounds hold for some constants $K_1,K_2, K>0$: $(s,t)\in\widehat{[S,T]}_{\leq}^2$,
  \begin{enumerate}[(i)]
      \item $\|\mathcal{A}_t-\mathcal{A}_s- A_{s,t}\|_{L^p(\Omega)}\leq K_1|t-s|^{\beta_1}+ K_2|t-s|^{\beta_2}$,
      \item $\|\E_s(\mathcal{A}_t-\mathcal{A}_s- A_{s,t})\|_{L^p(\Omega)}\leq K_2|t-s|^{\beta_2}$,
      \item  $\|\mathcal{A}_t-\mathcal{A}_s\|_{L^p(\Omega)}\leq Kp\Gamma_1S^{-\delta_1}|t-s|^{\beta_1}+KpS^{-\delta_2}\Gamma_2|t-s|^{\beta_2}$. Here $K$ depends only on $\beta_1,\beta_2,d$.
  \end{enumerate}
\end{lemma}

\subsection{Calculus on Anisotropic Besov spaces}

The following Bernstein inequality is standard and proven in \cite{ZZ21}. We omit the details.
\bl[Bernstein's inequality]\label{BI00}
For any $\bbk:=(k_1,k_2)\in\mN^2_0$, $\bbp\le\bbp'\in[1,\infty]^2$, there is a constant $C=C(\bbk,\bbp,\bbp',d)>0$ such that for all $j\geq 0$,
\begin{align}\label{Ber}
\|\nabla^{k_1}_x\nabla^{k_2}_v\cR^\bba_j f\|_{\bbp'}\lesssim_C 2^{j \bba\cdot(\bbk+\frac{d}{\bbp}-\frac{d}{\bbp'})}\|\cR^\bba_j  f\|_{\bbp},\qquad  \|\cR^\bba_jf   \|_{\bbp}\lesssim_C \|f\|_{\bbp}.
\end{align}
\el

The following lemma provides important properties about Besov spaces. 
\bl\label{lemB1}
\begin{enumerate}[(i)]
 \item For any $\bbp\in[1,\infty]^2$,  $s'>s$ and $q\in[1,\infty]$, it holds that
\begin{align}\label{AB2}
\bB^{0,1}_{\bbp;\bba}\hookrightarrow\mL^\bbp\hookrightarrow\bB^{0,\infty}_{\bbp;\bba},\ \ \bB^{s',\infty}_{\bbp;\bba}\hookrightarrow \bB^{s,1}_{\bbp;\bba}\hookrightarrow \bB^{s,q}_{\bbp;\bba}.
\end{align}
\item For any  $\beta,\beta_1,\beta_2\in\mR$, $q,q_1,q_2\in[1,\infty]$ and $\bbp,\bbp_1,\bbp_2\in[1,\infty]^2$ with
$$
\beta=\beta_1+\beta_2,\ \ 1+\tfrac{1}{\bbp}=\tfrac1{\bbp_1}+\tfrac1{\bbp_2},\ \ \tfrac{1}{q}=\tfrac1{q_1}+\tfrac1{q_2},
$$
it holds that
\begin{align}\label{Con}
\|f*g\|_{\bB^{\beta,q}_{\bbp;\bba}}\leq 5\|f\|_{\bB^{\beta_1,q_1}_{\bbp_1;\bba}}\|g\|_{\bB^{\beta_2,q_2}_{\bbp_2;\bba}}.
\end{align}
\item For any $s>0$, $q\in[1,\infty]$  and $\bbp, \bbp_1,\bbp_2\in [1,\infty]^2$ with $\tfrac1\bbp=\tfrac1{\bbp_1}+\tfrac1{\bbp_2}$,
there is a constant $C=C(\alpha, d,s,q,\bbp_1,\bbp_2)>0$ such that for any $(f,g)\in \bB^{s,q}_{\bbp_1;\bba}\times \bB^{s, q}_{\bbp_2;\bba}$,
\begin{align}\label{EB07}
\|fg\|_{\bB^{s,q}_{\bbp;\bba}}\lesssim_C
\|f\|_{\bB^{s,q}_{\bbp_1;\bba}}\|g\|_{\bbp_2}+\|f\|_{\bbp_1}\|g\|_{\bB^{s,q}_{\bbp_2;\bba}}.
\end{align}
\item Let $\bbp,\bbp_1,\bbp_2\in[1,\infty]^2$ and $s,s_0,s_1\in\mR$, $\theta\in[0,1]$, $q,q_1,q_2\in[1,\infty]$.
Suppose that
\begin{align}\label{DD01}
\tfrac1{q}=\tfrac{1-\theta}{q_1}+\tfrac{\theta}{q_2},\
\tfrac1{\bbp}\leq\tfrac{1-\theta}{\bbp_1}+\tfrac{\theta}{\bbp_2},\ \ s-\bba\cdot\tfrac{d}{\bbp}=(1-\theta)\big(s_1-\bba\cdot\tfrac{d}{\bbp_1}\big)+\theta\big(s_2-\bba\cdot\tfrac{d}{\bbp_2}\big).
\end{align}
Then there is a constant $C=C(\bbp,\bbp_1,\bbp_2,s,s_1,s_2,\theta)>0$ such that
\begin{align}\label{Sob}
\|f\|_{\bB^{s,q}_{\bbp;\bba}}\lesssim_C\|f\|^{1-\theta}_{\bB^{s_1,q_1}_{\bbp_1;\bba}}\|f\|^\theta_{\bB^{s_2,q_2}_{\bbp_2;\bba}}.
\end{align}
In particular, for $1\leq\bbp_1\leq\bbp\leq\infty$, $q\in[1,\infty]$ and $s=s_1+\bba\cdot\big(\tfrac{d}{\bbp}-\tfrac{d}{\bbp_1}\big)$,
\begin{align}\label{Sob1}
\|f\|_{\bB^{s,q}_{\bbp;\bba}}\lesssim_C\|f\|_{\bB^{s_1,q}_{\bbp_1;\bba}}.
\end{align}
\item For any $\eps>0$ and $\bbp\in[1,\infty]^2$, $(s_0,s_1)\in\mR^2$, we have the following embeddings:
\begin{align}\label{IA2}
\bB^{{s_0-\eps,s_1+\eps}}_{\bbp; x,\bba}\hookrightarrow\bB^{{s_0,s_1}}_{\bbp; x,\bba}
\hookrightarrow\bB^{{s_0+\eps,s_1-\eps}}_{\bbp; x,\bba},
\end{align}
and for $s_1>0$,
\begin{align}\label{IA3}
\bB^{{s_0,s_1}}_{\bbp; x,\bba}\hookrightarrow
\bB^{{s_0+s_1}}_{\bbp; x}.
\end{align}
\item For $\theta\in[0,1]$, $\gamma,\beta\in\mR$ and $\bbp\in[1,\infty]^2$, there is a constant $C=C(d,\theta,\bbp)>0$
such that
\begin{align}\label{IA1}
\|f\|_{\bB^{\theta \gamma, (1-\theta)\beta}_{\bbp; x,\bba}}\lesssim\|f\|^\theta_{\bB^{ \gamma/(1+\alpha)}_{\bbp;x}}
\|f\|^{1-\theta}_{\bB^\beta_{\bbp;\bba}}.
\end{align}
\end{enumerate}
\el
\br\rm
Inequality \eqref{Sob} is a version of the classical Gagliardo-Nirenberg inequality in Besov spaces.
Inequality \eqref{Sob1} is the classical Sobolev inequality.
\er

\subsection{Density estimates}
First, we 
define for $t>0$
\begin{align}
    \label{def:g}g_t(x,v):=(\frac{\pi t^4}{3})^{-d/2}\exp\left(-\frac{3|x|^2+|3x-2tv|^2}{2t^3}\right)
\end{align}
which is the density of $G_t:=(\int_0^tW_s\dif s,W_t)$. Also define
\begin{align}\label{CC01}
P_tf(z):=\Gamma_t g_t*\Gamma_tf(z)=\Gamma_t(g_t*f)(z),\quad z=(x,v)\in\mR^{2d}.
\end{align}
Notice that
\begin{align*}
  P_tf(z)=\mE f(x+tv+\int_0^tW_r\dif r,v+W_t)=\mE f(G_t(\Gamma_tz)).
\end{align*}
It is easy to see that
 for any $\varphi\in C^\infty_b(\mR^{2d})$,
\begin{align*}
\p_tP_t\varphi=(\frac{1}{2}\Delta_v+v\cdot\nabla_x)P_t\varphi.
\end{align*}
One also sees  the  scaling
\begin{align}\label{828:00}
g_t(x,v)=t^{-2d}g_1(t^{-\frac{3}{2}}x,t^{-\frac{1}{2}}v),\nonumber\quad \\ \Gamma_tg_t(x,v)=t^{-2d}g_1(t^{-\frac{3}{2}}x+t^{-\frac{1}{2}}v,t^{-\frac{1}{2}}v)=t^{-2d}\Gamma_1g_1(t^{-\frac{3}{2}}x,t^{-\frac{1}{2}}v),
\end{align}
which implies  the following estimates for $(g_t)_{t\in(0,1]}$.
\bl
\begin{enumerate}
    \item For any $\bbp\in[1,\infty]^2$, there is a constant
 $C=C(d,\alpha,\bbp)$ such that
\begin{align}
    \label{CC03}
    \|g_t\|_{\bbp}\lesssim t^{-\bba\cdot\frac{d}{2}(1-\frac{1}{\bbp})},\quad t\in(0,1].
\end{align}
\item For any  $\alpha,\beta\ge0$ and $m,n\in\mN_0$, for any $\bbp\in[1,\infty]^2$  there exist constants $C=C(d,\alpha,\beta,m,n,\bbp)$ such that for any $t\in(0,1]$,
\begin{align}\label{CC02+}
&\big\||x|^\alpha|v|^\beta|\nabla_x^m\nabla_v^ng_t(x,v)|\big\|_\bbp\le C t^{\frac{3(\alpha- m)}{2}+\frac{\beta- n}{2}-\bba\cdot\frac{d}{2}(1-\frac{1}{\bbp})}.
\end{align}
\item
For any  $\alpha,\beta\ge0$, $\bbp\in[1,\infty]^2$ and $m,n\in\mN_0$, there exist constants $C=C(d,\alpha,\beta,\bbp,m,n)$ such that for any $t\in(0,1]$,
\begin{align}\label{CC03+}
\||x|^\alpha|v|^\beta|\nabla_x^m\nabla_v^n\partial_t(\Gamma_tg_t)(x,v)|\big\|_\bbp\le C t^{\frac{3(\alpha- m)+(\beta-n)-2}{2}-\bba\cdot\frac{d}{2}(1-\frac{1}{\bbp})}.
\end{align}
\end{enumerate}
\el
\begin{proof}
    We only need to show \eqref{CC03+}, the others are similar. It follows from the scaling \eqref{828:00} that
    \begin{align*}
        \partial_t(\Gamma_tg_t)(x,v)=&(-2d)t^{-2d-1}\Gamma_1g_1(t^{-3/2}x,t^{-1/2}v)\\
        &+(-3/2)t^{-2d-5/2}x\cdot \nabla_x \Gamma_1g_1(t^{-3/2}x,t^{-1/2}v)\\
        &+(-1/2)t^{-2d-3/2}v\cdot \nabla_v \Gamma_1g_1(t^{-3/2}x,t^{-1/2}v).
    \end{align*}
   Defining
   \begin{align*}
       G(x,v):=(-2d)\Gamma_1g_1(x,v)+(-3/2)x\cdot \nabla_x \Gamma_1g_1(x,v)+(-1/2)v\cdot \nabla_v \Gamma_1g_1(x,v),
   \end{align*}
   one sees that
   \begin{align*}
       \partial_t(\Gamma_tg_t)(x,v)=t^{-2d-1}G(t^{-3/2}x,t^{-1/2}v)
   \end{align*}
   and then
   \begin{align*}
&|x|^\alpha|v|^\beta\nabla_x^m\nabla_v^n\partial_t(\Gamma_tg_t)(x,v)\\&=t^{-2d-1+[3(\alpha-m)+(\beta-n)]/2}|t^{-3/2}x|^\alpha|t^{-1/2}v|^\beta \nabla_x^m\nabla_v^nG(t^{-3/2}x,t^{-1/2}v).
   \end{align*}
   Therefore, we have
   \begin{align*}
&\||x|^\alpha|v|^\beta|\nabla_x^m\nabla_v^n\partial_t(\Gamma_tg_t)(x,v)|\big\|_\bbp\\
&=t^{-2d-1+[3(\alpha-m)+(\beta-n)]/2} t^{\bba\cdot \frac{d}{\bbp}}\||x|^\alpha|v|^\beta|\nabla_x^m\nabla_v^nG(x,v)|\big\|_\bbp\lesssim t^{\frac{3(\alpha- m)+(\beta-n)-2}{2}-\bba\cdot\frac{d}{2}(1-\frac{1}{\bbp})},
   \end{align*}
   provided that $\Gamma_1g_1$ is a Schwartz function. This completes the proof.
\end{proof}

We note that for any $t\in\mR$ and $\bbp\in[1,\infty]^2$
\begin{align*}
    \|\Gamma_t f\|_{\bbp}\le \| f\|_{\bbp},
\end{align*}
which shows that $(\Gamma_t)_{t\geq0}$ is invariant on $\LL^\bbp$. However, when we consider it on the other Besov spaces,  the invariant property does no hold in general. We illustrate this point in \cref{lem:Gamma-Besov} below.\\

Let's first recall the following useful results in \cite[Lemma 6.7]{HWZ20}.
\bl\label{HWZ67}
 For $t\geq 0$ and $j\in\mN_0$, define
\begin{align*}
\Theta^{t}_j:=\Big\{\ell\geq -1: 2^{\ell}\leq 2^{4} (2^j+t2^{3j}),\ 2^{j}\leq 2^{4} (2^\ell+t2^{3\ell})\Big\}.
\end{align*}
Then for any $\ell\notin\Theta^{t}_j$, it holds that
\begin{align}\label{EM3}
\cR^\bba_j\Gamma_{t}\cR^\bba_\ell=0.
\end{align}
\el

\bl\label{lem:Gamma-Besov}
Let $\beta>0$ and $T>0$. Then for any $\bbp\in[1,\infty]^2$ and $q\in[1,\infty]$, there is a constant $C=C(d,\bbp,T,q)>0$ such that for all $t\in[0,T]$ and $f\in\bB^{3\beta,q}_{\bbp;\bba}$
\begin{align}\label{0222:00}
    \|\Gamma_tf\|_{\bB^{\beta,q}_{\bbp;\bba}}\le C(\|f\|_{\bB^{\beta,q}_{\bbp;\bba}}+t^\beta\|f\|_{\bB^{3\beta,q}_{\bbp;\bba}}).
\end{align}
\el
\begin{proof}
    For any fixed $t\in[0,T]$, we let
   \begin{align*}
       \Theta^{t,1}_j:=\Big\{\ell\geq -1: 2^{2\ell}\le t^{-1},\ \ 2^{j}\leq 2^{8} 2^\ell\Big\}
   \end{align*}
   and
   \begin{align*}
       \Theta^{t,2}_j:=\Big\{\ell\geq -1: 2^{2\ell}\ge t^{-1},\ \ 2^{j}\leq 2^{8} t2^{3\ell}\Big\}.
   \end{align*}
   It is easy to see that $\Theta^{t}_j\subset \Theta^{t,1}_j\cup \Theta^{t,2}_j$. Then by \eqref{SE1}
and \cref{HWZ67}, we have
   \begin{align*}
       \|\cR_j^\bba \Gamma_t f\|_{\bbp}&
       \le \sum_{\ell\in \Theta^t_j}\|\cR_j^\bba\Gamma_t \cR_\ell^\bba f\|_{\bbp}\lesssim \sum_{\ell\in \Theta^t_j}\| \cR_\ell^\bba f\|_{\bbp}\lesssim \sum_{\ell\in \Theta^{t,1}_j}\| \cR_\ell^\bba f\|_{\bbp}+\sum_{\ell\in \Theta^{t,2}_j}\| \cR_\ell^\bba f\|_{\bbp}\\
       &\lesssim \sum_{\ell\in \Theta^{t,1}_j} 2^{-\beta \ell}\|f\|_{\bB^{\beta}_{\bbp;\bba}}+\sum_{\ell\in \Theta^{t,2}_j} 2^{-3\beta \ell}\|f\|_{\bB^{3\beta}_{\bbp;\bba}}.
   \end{align*}
   We note that
   \begin{align*}
       \Theta^{t,1}_j\subset \{\ell\ge-1~:~ \ell\ge j-8\},\quad \Theta^{t,2}_j\subset \{\ell\ge-1~:~ \ell\ge (\ln(t^{-1})+j-8)/3\},
   \end{align*}
   which implies that
   \begin{align*}
       \|\cR_j^\bba \Gamma_t f\|_{\bbp}&\lesssim 2^{-\beta j}\|f\|_{\bB^{\beta}_{\bbp;\bba}}+t^\beta2^{-\beta j}\|f\|_{\bB^{3\beta}_{\bbp;\bba}}.
   \end{align*}
   Based on the definition of Besov spaces and its property \eqref{AB2}, we complete the proof.
\end{proof}

\begin{lemma}
    \label{lem:est-semi-sob}
Let $\alpha\in(0,1)$.    For $f\in\mB_{\bbp,\bba}^{\alpha}$ with  $\bbp\in[1,\infty]^2$ 
we have
     \begin{align}\label{est-semi-sob}
     \| \p_r\Gamma_rg_r*(\Gamma_tf)\|_{\bbp}\lesssim {\left(\|f\|_{\mB_{\bbp,\bba}^{\alpha}}+t^\alpha\|f\|_{\mB_{\bbp,\bba}^{3\alpha}} \right)}r^{\frac{\alpha}{2}-1}.
 \end{align}
 \end{lemma}
 \begin{proof}
 Denote $F_t:=\Gamma_tf$.
  We note that $\int_{\mR^{2d}}\p_r\Gamma_rg_r(z')\dif z'=0 $ and
  \begin{align*}
      \p_r\Gamma_rg_r*(\Gamma_tf)(z)&=\int_{\mR^{2d}} \p_r\Gamma_rg_r(z')\Gamma_tf(z-z')\dif z'\\
      &=\int_{\mR^{2d}} \p_r\Gamma_rg_r(z')(\Gamma_tf(z-z')-\Gamma_tf(z))\dif z'
  \end{align*}
Let $\bbp\prime$ be the dual of $\bbp$. By equivalent form \eqref{CH1}
     it reads
     \begin{align*}
        \| \p_r\Gamma_rg_r*F_t\|_{\bbp}&\lesssim \|F_t\|_{\mB_{\bbp,\bba}^{\alpha}}\int_{\mR^{2d}}|z'|_a^\beta|\p_r\Gamma_rg_r(z')|\dif z'\lesssim \|F_t\|_{\mB_{\bbp,\bba}^{\alpha}}r^{\frac{\alpha}{2}-1}\\
   &\lesssim\left(\|f\|_{\mB_{\bbp,\bba}^{\alpha}}+t^\alpha\|f\|_{\mB_{\bbp,\bba}^{3\alpha}} \right)r^{\frac{\alpha}{2}-1},
   \end{align*}
   where the second inequality is based on \eqref{CC03+} and the third inequality is from \eqref{0222:00}. This completes the proof.
     \end{proof}

\begin{lemma}
    \label{lem:est-pro-sob}
   Let $r\in[0,1]$, $0<s\leq t\leq 1$, $\alpha\in (0,1)$ and $p\in[1,\infty)$. Denote $Z_t:=(\int_0^tW_s\dif s,W_t)$. For  $f\in\mB_{\bbp,\bba}^{\alpha}$ with $\bbp=(p_x,p_v)\in[1,\infty]^2$, $p=:p_x\wedge p_v$, there exists constant $C=C(\alpha,p,d, \bbp)$ so that
   \begin{align}
       &\label{est:E-f-Z}
       \| f(Z_t)\|_{L^p(\Omega)}\leq Ct^{-\bba\cdot\frac{d}{2\bbp}} \|f\|_{\bbp},
       \\
         &\| \Gamma_{r}f(Z_t)-{\Gamma_{t-s+r}}f(Z_s)\|_{L^p(\Omega)}\leq Cs^{-\bba\cdot\frac{d}{2\bbp}}|t-s|^{\frac{\alpha}{6}}(|t-s|^{\frac{\alpha}{3}}+r^{\frac{\alpha}{3}}) \|f\|_{\mB_{\bbp,\bba}^{\alpha}}.\label{est:E-dif-f-Z}
   \end{align}
\end{lemma}
\begin{proof}
    According to the fact that $g_t$ from \eqref{def:g} is the  the density of  $Z_t$,  
    it writes by H\"older’s inequality for $\bbq:=(\frac{p_x}{p_x-p},\frac{p_v}{p_v-p})$
    \begin{align}\label{est:mom-p}
    \| f(Z_t)\|_{L^p(\Omega)}^p=\big(\int_{\mR^{2d}}|f(z)|^pg_t(z)\dif z\big) \lesssim    \|f\|_{\bbp}^p\|g_t\|_{\bbq}. 
    \end{align}
 By \eqref{CC03}   we know that
 \begin{align*}
    \|g_t\|_{\bbq}\lesssim   t^{-\bba\cdot\frac{d}{2}(1-\frac{1}{\bbq})}.
 \end{align*}
 Therefore we plug the above into \eqref{est:mom-p} then get \eqref{est:E-f-Z}.

    Moreover, we note that
    \begin{align}
        Z_t-\Gamma_{t-s}Z_s&=\left(\int_0^tW_r dr,W_t\right)-\left(\int_0^sW_r dr+(t-s)W_s,W_s\right)\nonumber\\
        &=\left(\int_s^t (W_r-W_s)d r,W_t-W_s\right)\overset{(d)}{=}\left(\int_0^{t-s} W_rd r,W_{t-s}\right)=Z_{t-s}\label{eq:in-in}
    \end{align}
    is independent of $\sF_s$, which implies that
    \begin{align*}
   &  \| \Gamma_{r}f(Z_t)-\Gamma_{t-s+r}f(Z_s)\|_{L^p(\Omega)}^p=\| \Gamma_{r}f(Z_t)-\Gamma_{r}f(\Gamma_{t-s}Z_s)\|_{L^p(\Omega)}^p\\
   = &\int_{\mR^{2d}}\int_{\mR^{2d}}   |\Gamma_{r}f(z+\Gamma_{t-s}z')-\Gamma_{r}f(\Gamma_{t-s}z')|^pg_{t-s}(z)g_{s}(z')\dif z'\dif z
     \\
     =&\int_{\mR^{2d}}\int_{\mR^{2d}} |f(z+z')-f(z')|^pg_{t-s}(\Gamma_{-r}z)g_{s}(\Gamma_{s-t}z')\dif z'\dif z.
     \end{align*}
It follows from \eqref{CH1} and \eqref{def:Bes-pq} that
\begin{align*}
    \|f(z+\cdot)-f(\cdot)\|_{\bbp}\lesssim |z|_\bba^\alpha \|f\|_{\bB^\alpha_{\bbp;\bba}}
\end{align*}
     which together with H\"older's inequality and \eqref{CC02+} implies that
     \begin{align*}
     &\| \Gamma_{r}f(Z_t)-\Gamma_{t-s+r}f(Z_s)\|_{L^p(\Omega)}^p\lesssim 
     \|f\|_{\bB^{\alpha}_{\bbp;\bba}}^p \|\Gamma_{s-t}g_{s}\|_{\bbq}\int_{\mR^{2d}}|z|_\bba^{\alpha p}g_{t-s}(\Gamma_{-r}z)\dif z
     \\
     &\lesssim s^{-\bba\cdot\frac{d}{2\bbp}}\|f\|_{\bB^{\alpha}_{\bbp;\bba}}^p\int_{\mR^{2d}}|(x+rv,v)|_\bba^{\alpha p} g_{t-s}(z)\dif z\\
     &\lesssim s^{-\bba\cdot\frac{d}{2\bbp}} \|f\|_{\bB^{\alpha}_{\bbp;\bba}}^p \int_{\mR^{2d}}(|(x,v)|_\bba^{\alpha p}+|rv|^{\frac{{\alpha p}}{3}})g_{t-s}(z)\dif z\\
     &\lesssim s^{-\bba\cdot\frac{d}{2\bbp}}\|f\|_{\bB^{\alpha}_{\bbp;\bba}}^p\left((t-s)^{\frac{\alpha p}{2}}+r^{\frac{\alpha p}{3}}(t-s)^{\frac{\alpha p}{6}}\right).
     \end{align*}
This completes the proof.
\end{proof}
As a corollary, we have the following results.
\begin{corollary}\label{lem002}
\begin{enumerate}[(i)]
    \item For $f\in \mL^{\bbp\prime}$
  with $\bbp\prime=(p_x\prime,p_v\prime),$  $p_x\prime,p_v\prime\in[1,\infty)$,
  for any $0\leq s\leq t\leq 1$, and $\bbp=(p_x,p_v)\in[p_x\prime,\infty)\times[p_v\prime,\infty)$,  for any  $\delta\in(0,1]$, for any $k\in\mN$, the following holds 
\begin{align}
    \label{est:semi-P-L}
   \|{\nabla^k_v}(P_tf-P_s\Gamma_{t-s}f)\|_{\bbp}\le C\|f\|_{{\bbp}\prime}|t-s|^\delta s^{d(\frac{3}{2}(\frac{1}{p_x}-\frac{1}{p_x\prime})+\frac{1}{2}(\frac{1}{p_v}-\frac{1}{p_v\prime})){-\frac{k}{2}}-\delta}.
 \end{align}
\item For  any $\beta\in(0,2]$ and  $\delta\in[\beta/2,1]$, for $f\in\mB_{\bbp,\bba}^{\beta}$ with $\bbp=(p_x,p_v)$ and $p_x,p_v\in[1,\infty)$ 
we have
\begin{align}
    \label{est:semi-P-So}
   \|P_tf-P_s\Gamma_{t-s}f\|_{\bbp} \le C(t-s)^{\delta}s^{-\frac{2\delta-\beta}{2}}\left(\|f\|_{\mB_{\bbp,\bba}^{\beta}}+(t-s)^\beta\|f\|_{\mB_{\bbp,\bba}^{3\beta}}\right).
\end{align}
Particularly for $f\in \mL^\infty$,
\begin{align}\label{est:infty}
\|P_tf-P_s\Gamma_{t-s}f\|_{\infty}\le C\Big([(t-s)s^{-1}]\wedge1\Big)\|f\|_{\infty}.
\end{align}
\end{enumerate}
\end{corollary}
\begin{remark}\label{rmk314}
Based on \cite[Lemma 3.4]{HZZZ22}, we have  for any $\alpha\in[0,2]$ and $\beta\in(-\infty,\alpha]$ there is a constant $C=C(d,T,\alpha,\beta)>0$ such that for all $0< s<t\le T$
\begin{align*}
    \|P_t f-\Gamma_{t-s}P_sf\|_\infty\le C(t-s)^{\frac{\alpha}{2}}s^{-\frac{\alpha-\beta}{2}}\|f\|_{\bC^\beta_a}.
\end{align*}
Therefore, we have the follow observation
\begin{align*}
    \|P_t f-\Gamma_{t-s}P_sf\|_\infty+\|P_t f-P_s\Gamma_{t-s}f\|_\infty\lesssim (t-s)s^{-1}\|f\|_\infty,
\end{align*}
in spite of $\Gamma_{t-s}P_sf\ne P_s\Gamma_{t-s}f$.
\end{remark}
 \begin{proof}[Proof of \cref{lem002}]
By \eqref{CC01}, one sees that
 \begin{align*}
P_tf-P_s\Gamma_{t-s}f=(\Gamma_tg_t)*(\Gamma_tf)-(\Gamma_sg_s)*(\Gamma_tf)=\int_s^t\p_r\Gamma_rg_r\dif r\ast F_t,   \quad F_t:=\Gamma_tf,
 \end{align*}
moreover following from \eqref{CC03+}
 together with Young’s inequality for convolution with $\frac{1}{\bbp}+1=\frac{1}{\bbp\prime}+\frac{1}{{\bbq}}$ 
 we get that 
 \begin{align*}
  \|  {\nabla^k_v}( P_tf-P_s\Gamma_{t-s}f)\|_{\bbp}=& \| {\nabla^k_v}\int_s^t\p_r\Gamma_rg_r\ast F_t\dif r\|_{\bbp}
  \\\lesssim& \int_s^t \| {\nabla^k_v}\p_r\Gamma_rg_r\|_{\bbq}\dif r  \|F_t\|_{{\bbp}\prime}
        \\ \lesssim&\|f\|_{{\bbp}\prime}\int_s^t\big( r^{d(\frac{3}{2q_x}+\frac{1}{2q_v}-2){-\frac{k}{2}}-1}\big)\dif r
\\=&\|f\|_{{\bbp}\prime}\int_s^t\big( r^{d(\frac{3}{2}(\frac{1}{p_x}-\frac{1}{p_x\prime})+\frac{1}{2}(\frac{1}{q_x}-\frac{1}{q_x\prime}))-1{-\frac{k}{2}}}\big)\dif r
\\\lesssim&\|f\|_{{\bbp}\prime}|t-s|^\delta s^{d(\frac{3}{2}(\frac{1}{p_x}-\frac{1}{p_x\prime})+\frac{1}{2}(\frac{1}{q_x}-\frac{1}{q_x\prime})){-\frac{k}{2}}-\delta}
 \end{align*}
 for any $\delta\in(0,1]$
 Therefore \eqref{est:semi-P-L} holds. \\

 For (ii) by \eqref{est-semi-sob} we have the following
  \begin{align*}
  \|   P_tf-P_s\Gamma_{t-s}f\|_{\bbp}=& \| \int_s^t\p_r\Gamma_rg_rF_t\dif r\|_{\bbp}
  \\\lesssim  &\int_s^t{\left(\|f\|_{\mB_{\bbp,\bba}^{\beta}}+|t-s|^\beta\|f\|_{\mB_{\bbp,\bba}^{3\beta}} \right)}r^{\frac{\beta}{2}-1}\dif r.
  \\\lesssim  & {\left(\|f\|_{\mB_{\bbp,\bba}^{\beta}}+|t-s|^\beta\|f\|_{\mB_{\bbp,\bba}^{3\beta}} \right)}\int_s^t r^{\frac{\beta}{2}-1}\dif r
  \\\lesssim  &(t-s)^{\delta}s^{-\frac{2\delta-\beta}{2}}\left(\|f\|_{\mB_{\bbp,\bba}^{\beta}}+(t-s)^\beta\|f\|_{\mB_{\bbp,\bba}^{3\beta}}\right)
 \end{align*}
 where $\delta\in[\frac{\beta}{2},1)$.
In particular we have for $f\in \mL^\infty$
 \begin{align*}
\|P_tf-P_s\Gamma_{t-s}f\|_{\infty}=\|(\Gamma_tg_t)*(\Gamma_tf)-(\Gamma_sg_s)*(\Gamma_tf)\|_{\infty}\le \|\Gamma_tg_t-\Gamma_sg_s\|_1\|f\|_{\infty}.
\end{align*}
Moreover, in the view of \eqref{CC03+} and Fubini's theorem,
 \begin{align*}
 \|\Gamma_tg_t-\Gamma_sg_s\|_1&\le \int_s^t\|\p_r\Gamma_rg_r\|_1\dif r\\
 &\le 
 \int_s^t r^2r^{-3}+rr^{-(1/2+3/2)}+r^{-1}+r^{\frac{1-3}{2}}\dif r\\&\lesssim \int_s^t r^{-1}\dif r\lesssim (t-s)s^{-1}
\end{align*}
 and it shows the estimate for \eqref{est:infty}.
The proof completes.
 \end{proof}

 In the end of this session, we also include a result from \cite{HRZ} on the estimates of the kinetic semigroup  $(P_t)_{t\geq0}$ that heavily used in the proof of weak convergence part.
 \begin{lemma}\cite[Lemma 2.16, (2.35)]{HRZ}
     \label{lem:est-Pt-itself}
     Let $\bbp_1,\bbp_2,\bbp_3\in[1,\infty]^2$ with $\bbp_1,\bbp_2\leq \bbp_3$. Denote $\kappa_i:=\bba\cdot(\frac{d}{\bbp_i}-\frac{d}{\bbp_3}), i=1,2$. For any $\beta_1,\beta_2\in\mR$, there are constants $C_i=C_i(d,\beta_1,\beta_2,\bbp_i,\bbp_3),i=1,2$, such for all $t>0$
     \begin{align}\label{est:Pt-itself-1}
      \mathbf{1}_{\beta_2\neq\beta_1-\kappa_1}\|P_tf\|_{\bB_{\bbp_3;\bba}^{\beta_2,1}}+ \|P_tf\|_{\bB_{\bbp_1;\bba}^{\beta_2,1}}\leq C_1(1\wedge t)^{-\frac{(\beta_2-\beta_1+\kappa_1)}{2}}\|f\|_{\bB_{\bbp_3;\bba}^{\beta_1}},
     \end{align}
      particularly, for any $q\in[1,\infty]$
     \begin{align}
\label{est:Pt-itself-2}
        \|P_tf\|_{\bB_{\bbp_3;\bba}^{\beta_1+\beta_2,q}}\lesssim(1\wedge t)^{-\frac{\beta_2\vee0}{2}}\|f\|_{\bB_{\bbp_3;\bba}^{\beta_1,q}}.
     \end{align}
 \end{lemma}
\section{Strong and weak well-posedness and stability of SDEs}\label{sec:Strong-Weak-Stable}
In this section, we first use the method in \cite[Section 4]{HRZ} to study the weak/strong well-posedness of the following kinetic SDE:
\begin{align}\label{KSDE}
\left\{
\begin{aligned}
&X_t=X_0+\int_0^t V_s\dif s,\\
&V_t=V_0+\int^t_0\DD(s,Z_s)\dif s+W_t,
\end{aligned}
\right.
\end{align}
where $Z_t:=(X_t,V_t)$, $\DD: \mR_+\times\mR^{2d}\to\mR^d$ is a measurable vector field, and  $(W_t)_{t\geq 0}$ is a standard $d$-dimensional Brownian motion as in the introduction.

We have the following well-posedness result for SDE \eqref{KSDE}.
\bt\label{thmW}
Let $T\in[0,\infty)$, $\beta>0$, $\bbp\in(1,\infty)^2$  and $\bba\cdot \frac{d}{\bbp} <1$.
\begin{enumerate}[(i)]
\item {\bf (Weak well-posedness)} Suppose $\DD\in L^\infty_T(\bB^\beta_{\bbp;\bba})$. For any initial distribution $\mu_0$,
there is a unique weak solution $Z$ to kinetic SDE \eqref{KSDE} satisfying the following Krylov estimate:  there is a constant $C>0$ such that for all $f\in L^q_T(\mL^{\bbp}_z)$ with $2-\frac{2}{q}-\bba\cdot\frac{d}{\bbp} >0$,and $0\le s<t \le T$
\begin{align}\label{0320:03}
    \mE\left|\int_s^t f(r,Z_r)\dif r\right|\lesssim_C \|f\|_{L^q_T(\mL_z^{\bbp})}.
\end{align}
\item {\bf (Strong well-posedness)} Suppose $\DD\in L^\infty_T(\bB^{\frac23,\beta}_{\bbp;x,\bba})$. For any initial value $Z_0$,
there is a unique strong solution $Z$ to kinetic SDE \eqref{KSDE}
\item  {\bf (Weak stability)}
For any $\delta\in(-1+\bba\cdot \frac{d}{\bbp},0)$, there is a positive constant $C$ which depends on $d,\bbp,\delta,\beta,\|b^1\|_{L_T^\infty(\bB^{-\gamma}_{\bbp;\bba})}, \|b^2\|_{L_T^\infty(\bB^{-\gamma}_{\bbp;\bba})}$ such that for any two weak solutions $X^1$ and $X^2$ corresponding to $b^1$ and $ b^2$, for any $t\in[0,T]$
\begin{align*}
       \|\mP\circ(X^1_t)^{-1}-\mP\circ(X^2_t)^{-1}\|_{\rm var}\le \|\mP\circ(X^1_0)^{-1}-\mP\circ(X^2_0)^{-1}\|_{\rm var}+C\|b^1-b^2\|_{L_T^\infty(\bB^{-\gamma}_{\bbp;\bba})}.
   \end{align*}
\item {\bf (Strong stability)} Let $\wt\DD\in L^\infty_T(\bB^{\frac23,\beta}_{\bbp;x,\bba})$ be another vector field and $\wt Z_0$ another initial random variable
on the same probability space. Let $\wt Z$ be the strong solution of \eqref{KSDE} corresponding to $\wt b$ and $\wt Z_0$. Then for any $\delta\in(-1+\bba\cdot \frac{d}{\bbp} ,0)$, there is a constant $C=C(d,\beta,\delta,T,\|b\|_{L^\infty(\bB^{\beta}_{\bbp;\bba})})>0$
\begin{align}\label{SW3}
\mE\left(\sup_{t\in[0,T]}|Z_t-\wt Z_t|^2\right)\lesssim_C \mE|Z_0-\wt Z_0|^2
+\|b-\wt b\|^2_{L^\infty_T(\bB^{\delta}_{\bbp;\bba})}.
\end{align}
\end{enumerate}
\et
The weak well-posedness is obtained in \cite{RZ24}, but for the readers' convenience, we give the entire proof to this theorem.
First, we can give
\begin{proof}[Proof of the weak existence]\label{proof:existence}
Let $(\Omega,\sF,\mP,(\sF_s)_{s\ge0})$ be a probability space and $(W_t)_{t\ge0}$ be a standard $d$-dimensional Brownian motion thereon. Based on \eqref{est:E-f-Z}, we have
\begin{align*}
  \Big(  \mE_s\int_s^t |b(r,\int_0^r W_u\dif u,W_r)|^2\dif r\Big)^\frac{1}{2}\lesssim (t-s)^{\frac{1-\bba\cdot d/\bbp}2}\|b\|_{L^\infty_T(\mL^{\bbp})},
\end{align*}
which by standard argument (John-Nirenberg inequality for instance \cite{Le2022}) implies that for any $m\geq1$ 
\begin{align*}
    \mE_s\left|\int_s^t |b(r,\int_0^r W_u\dif u,W_r)|^2\dif r\right|^m\lesssim m![(t-s)^{1-\bba\cdot d/\bbp}\|b\|_{L^\infty_T(\mL^{\bbp})}]^{2m},
\end{align*}
and thus
\begin{align*}
    \mE\exp\(\int_0^T\big|b\big(r,\int_0^r W_u\dif u,W_r\big)\big|^2\dif r\)<\infty.
\end{align*}
By Girsanov's theorem, under the new probability
\begin{align*}
    \dif \mQ(\omega):=\exp\(-\int_0^Tb(r,\int_0^r W_u\dif u,W_r)\dif W_r-\frac12\int_0^T|b(r,\int_0^r W_u\dif u,W_r)|^2\dif r\)\dif\mP,
\end{align*}
the process $V_t$ defined by
\begin{align*}
    V_t=v+\int_0^t b(r,\int_0^r W_u\dif u,W_r)\dif r+W_t
\end{align*}
is a $\mQ$-Brownian motion. Then $M_t:=(x+tv+\int_0^t W_r\dif r, v+W_t)$ is a weak solution to \eqref{KSDE} on $(\Omega,\sF,\mQ)$ with initial data $Z_0=(x,v)$.

Moreover,  by Girsanov theorem, H\"older's inequality and \eqref{est:E-f-Z}, for any $f\in L^q_T(\mL^{\bbp})$  we get
\begin{align}\label{0320:01}
    \mE\left|\int_s^t f(r,Z_r)\dif r\right| =\mE^\mQ\left|\int_s^t f(r,M_r)\dif r\right|\lesssim (t-s)^{\frac{2-\alpha}{2}}\|f\|_{L^q_T(\mL^{\bbp})},
\end{align}
where $\alpha:=2/q+\bba\cdot d/\bbp<2$.
\end{proof}

To show the uniqueness in law and pathwise uniqueness, we need to study the following linear kinetic equation with transport drift:
\begin{align}\label{LKE}
\p_t u=\Delta_v u+v\cdot\nabla_x u-\lambda u+\DD\cdot\nabla_v u+f,\quad u_0=0,
\end{align}
where $\lambda\ge0$, $\DD: \mR_+\times\mR^{2d}\to\mR^d$ and $f: \mR_+\times\mR^{2d}\to\mR$ are two Borel measurable functions.

\bt\label{thm41}
Let $T>0$  $\beta,\gamma>0$, $\bbp\in(1,\infty)^2$ with $\bba\cdot d/\bbp<1$ and  $\delta\in(-1+\bba\cdot d/\bbp,\beta]$.
\begin{enumerate}[(i)]
\item For any $\DD\in L^\infty_T(\bB^{\beta}_{\bbp;\bba})$  and $f\in L^\infty_T(\bB^{\delta}_{\bbp;\bba})$,
there is a unique weak solution $u$ to equation \eqref{LKE} with the regularity that for any $\theta\in[0,2]$
and all $\lambda\geq 1$,
\begin{align}\label{FA00}
\|u\|_{L^{\infty}_{T}(\bB^{\theta+\delta}_{\bbp;\bba})}\lesssim_C \lambda^{\frac{\theta-2}{2}}\|f\|_{L^\infty_T(\bB^\delta_{\bbp;\bba})},
\end{align}
where $C>0$ only depends on $T,d,\beta,\delta,q,\theta$ and $\|\DD\|_{L^\infty_T(\bB^\beta_{\bbp;\bba})}$.
\item Let $\DD\in L^\infty_T(\bB^{\gamma,\beta}_{\bbp;x,\bba})$, $f\in L^\infty_T(\bB^{\gamma,\delta}_{\bbp;x,\bba})$ and $u$ be the solution of \eqref{LKE}.
For any $\theta\in[0,2]$,
there is a constant $C>0$ depending on $T,d,\delta,\gamma,\beta,\theta$ and $ \|\DD\|_{L^\infty_T(\bB^{\gamma,\beta}_{\bbp;x,\bba})}$ such that for all $\lambda\geq 1$,
\begin{align}\label{FA001}
\|u\|_{\mL^{\infty}_{T}(\bB^{\gamma,\theta+\delta}_{\bbp;x,\bba})}\lesssim_C \lambda^{\frac{\theta-2}{2}}\|f\|_{L^\infty_T(\bB^{\gamma,\delta}_{\bbp;x,\bba})}.
\end{align}
\end{enumerate}
\et
\begin{proof}
(i) By the standard continuity method (e.g. \cite{Zhang2018, ZZ21}), we only need to prove the a priori estimate \eqref{FA00}. By Duhamel's formula, we have
\begin{align*}
u_t&=\int_0^tP^\lambda_{t-s}(\DD_s\cdot\nabla_v u_s+f_s)\dif s,
\end{align*}
where
$$
P^\lambda_tf(z)=\e^{-\lambda t} P_tf(z),\ \ z\in\mR^{2d},\ t>0.
$$
By \cite[Lemma 3.3]{ZZ21}, 
for any $\theta\in[0,2)$, we have
\begin{align}
\|u_t\|_{\bB^{\theta+\delta}_{\bbp;\bba}}&\lesssim\int_0^t\e^{-\lambda(t-s)}(t-s)^{-\frac{\theta}2}
\Big(\|\DD_s\cdot\nabla_v u_s\|_{\bB^{\delta}_{\bbp;\bba}}+\|f_s\|_{\bB^{\delta}_{\bbp;\bba}}\Big)\dif s\no\\
&\lesssim\int_0^t\e^{-\lambda(t-s)}(t-s)^{-\frac{\theta}2}\Big(\|\DD_s\|_{\bB^{\beta}_{\bbp;\bba}}\|u_s\|_{\bC^{1+\delta\vee0}_{\bba}}+\|f_s\|_{\bB^{\delta}_{\bbp;\bba}}\Big)\dif s.\label{Aw1}
\end{align}
In particular, taking $\theta=1-\delta\wedge 0+\bba\cdot d/\bbp$, we have
\begin{align}\label{est:AW1}
\|u_t\|_{\bC^{1+\delta\vee0}_\bba}\lesssim\|u_t\|_{\bB^{1+\bba\cdot d/\bbp+\delta}_{\bbp;\bba}}\lesssim&\|\DD\|_{L^\infty_T(\bB^{\beta}_{\bbp;\bba})}\int_0^t\e^{-\lambda(t-s)}(t-s)^{-\frac{1-\delta\wedge 0+\bba\cdot d/\bbp}2}\|u_s\|_{\bC^{1+\delta\vee0}_\bba}\dif s\nonumber\\
&+\|f\|_{L^\infty_T(\bB^{\delta}_{\bbp;\bba})}\int_0^t\e^{-\lambda s}s^{-\frac{1-\delta\wedge 0+\bba\cdot d/\bbp}2}\dif s.
\end{align}
Since $\frac {1-\delta\wedge 0+\bba\cdot d/\bbp}2<1$, by Gr\"onwall's inequality of Volterra type (\cite{Zhang10}), we get
$$
\|u\|_{L^\infty_T(\bC^{1+\delta\vee0}_\bba)}\lesssim \|f\|_{L^\infty_T(\bB^{\delta}_{\bbp;\bba})}.
$$
Then by \eqref{Aw1} and \eqref{est:AW1} we have for $\lambda\geq 1$,
\begin{align*}
\|u\|_{L^{\infty}_T(\bB^{\theta+\delta}_{\bbp;\bba})}
&\lesssim\left(\|u\|_{L^\infty_T(\bC^{1+\delta\vee0}_\bba)}\|b\|_{L^\infty_T(\bB^\beta_{\bbp;\bba})}
+\|f\|_{L^\infty_T(\bB^\delta_{\bbp;\bba})}\right)
\left(\int_0^T\e^{-\lambda s}s^{-\frac{\theta}2}\dif s\right)\\
&\lesssim\|f\|_{L^\infty_T(\bB^\delta_{\bbp;\bba})} \lambda^{\frac{\theta-2}2}.
\end{align*}
Thus we get \eqref{FA00}.

(ii) Noting that for any $j\geq 0$ (see \cite[Lemma 3.3]{ZZ21}),
$$
\|\cR^x_jP_tf\|_{\bB^{\theta+\beta}_{\bbp;\bba}}\lesssim t^{-\frac{\theta}2}\|\cR^x_jf\|_{\bB^{\beta}_{\bbp;\bba}}\lesssim t^{-\frac{\theta}2}2^{-\gamma k}\|f\|_{\bB^{\gamma,\beta}_{\bbp;x,\bba}},
$$
we have for any $\gamma\in\mR$,
$$
\|P_tf\|_{\bB^{\gamma,\theta+\beta}_{\bbp;x,\bba}}\lesssim t^{-\frac{\theta}2}\|f\|_{\bB^{\gamma,\beta}_{\bbp;x,\bba}}.
$$
By this estimate, the estimate \eqref{FA001} is obtained completely the same as getting \eqref{FA00}.
\end{proof}

The following is another auxiliary lemma that is also crucial for showing the weak uniqueness and pathwise uniqueness.
 \bl\label{lem44}
 Let $T>0$
 and $\beta>0$. Suppose that $\DD,f\in L^\infty_T(\bB^{\beta}_{\bbp;\bba})$ with  $\bbp\in(1,\infty)^2$  and $\bba\cdot d/\bbp<1$.
Let $u$ be the unique weak solution of the following backward kinetic PDE:
\begin{align}\label{Aw2}
\p_t u+\Delta_v u+v\cdot\nabla_x u+\DD\cdot\nabla_v u=f,\quad u(T)=0,
\end{align}
and $Z$ be a weak solution of kinetic SDE \eqref{KSDE} satisfying \eqref{0320:03}. Then for any $t\in[0,T]$,
\begin{align}\label{SDE304}
u(t,Z_t)=u(0,Z_0)+\int^t_0 f(s, Z_s)\dif s
+\int^t_0\nabla_v u(s,Z_s)\dif W_s.
\end{align}
\el
\begin{proof} Firstly observe that \eqref{Aw2}  can be derived from \eqref{LKE} simply by changing of time variable in the backward way. Hence \cref{thm41} yields the a priori estimates \eqref{FA00} and \eqref{FA001} of the solution to \eqref{Aw2}, therefore further implies the existence and uniqueness of the weak  solution to \eqref{Aw2}.

Let $\rho\in \bC^\infty_c$ be a smooth density function with support in $B^\bba_1$.
For each $n\in\mN$, define
\begin{align}\label{Mo5}
\rho_{n}(z)=\rho_{n}(x,v):=n^{4d}\rho(nx,n^{3}v)
\end{align}
and
$$
u^n(t,z):=(u(t,\cdot)\ast\rho_{n})(z).
$$
Taking convolutions for both sides of \eqref{Aw2} with $\rho_n$, we get
\begin{align*}
\p_t u_n+(\Delta_v+v\cdot\nabla_x)u_n+\DD\cdot\nabla_v u_n=f_n+g_n+h_n,
\end{align*}
where
\begin{align*}
g_n=v\cdot\nabla_x u_n-(v\cdot\nabla_x u)\ast\rho_n,\ \ h_n:=\DD\cdot\nabla_v u_n-(\DD\cdot\nabla_v u)\ast\rho_n.
\end{align*}
By It\^o's formula, we have
\begin{align*}
u_n(t,Z_t)&=u_n(0,Z_0)+\int^t_0 (f_n+h_n+g_n)(s, Z_s)\dif s+\int^t_0\nabla_v u_n(s,Z_s)\dif W_s.
\end{align*}
Noting that by \eqref{FA00} with $\theta=2$,
\begin{align}
\|u_n-u\|_{L^\infty_T(\bC_\bba^{2-\bba\cdot d/\bbp})}&\leq\sup_{t\in[0,T]}\int_{\mR^{2d}}\|u(t,\cdot-y,\cdot-w)-u(t,\cdot)\|_{\bC_\bba^{2-\bba\cdot d/\bbp}}\rho_n(y,w)\dif y\dif w\no\\
&\leq\|u\|_{L^\infty_T(\bC^{2-\bba\cdot d/\bbp+\beta}_\bba)}\int_{\mR^{2d}}(|y|^{\frac1{3}}+|w|)^{\beta}\rho_n(y,w)\dif y\dif w
\no\\
&\lesssim n^{-\frac{\beta}{3}}\|u\|_{L^\infty_T(\bB^{2+\beta}_{\bbp;\bba})},\label{Aw3}
\end{align}
we have
$$
 a.s.\quad \left(u_n(t,Z_t),u_n(0,Z_0)\right)\to\left(u(t,Z_t),u(0,Z_0)\right)\quad \text{as} \quad n\to\infty.
$$
Letting
$$
A_n:=\int^t_0\left(\nabla_v u_n(s,Z_{s})-\nabla_v u(s,Z_{s})\right)\dif W_s,
$$
by the isometry formula of stochastic integrals, we have
\begin{align*}
\mE|A_n|^2\le \int^t_0\|\nabla_v u_n(s)-\nabla_vu(s)\|_\infty\dif s,
\end{align*}
which by \eqref{Aw3} implies that
\begin{align*}
\lim_{n\to\infty}\mE|A_n|^2=0.
\end{align*}
Finally, noting that
\begin{align*}
g_n(t,z)=\int_{\mR^{2d}}u(t,x-y,v-w) w\cdot\nabla_y \rho_n(y,w)\dif y\dif w,
\end{align*}
since $\rho_{n}(y,w)=n^{4d}\rho(ny,n^{3}w)$, we have
\begin{align*}
\|g_n(t)\|_{\infty}\lesssim n^{-\alpha}\|u(t)\|_{\infty}
\end{align*}
and
\begin{align*}
\|h_n(t)\|_{\bbp}&\lesssim n^{-\beta/3}\|\DD(t)\|_{\bB^{\beta}_{\bbp,\bba}}\|\nabla_v u(t)\|_{\infty}.
\end{align*}
Hence, by \eqref{0320:03}
\begin{align*}
a.s.\ \ \int^t_0(|f_n-f|+|h_n|+|g_n|)(s,Z_s)\dif s\to 0,\ \ \text{as}\ \ n\to\infty.
\end{align*}
Thus we complete the proof.
\end{proof}
\br\rm
Notice that the stochastic integral in \eqref{SDE304} is martingale by the regularity estimates of $u$ in \eqref{FA00}.
 \er

To show the stability, we give the following estimate.
\begin{lemma}
  Assume $b\in L^{\infty}_T(\bB^{\beta}_{\bbp;\bba})$ with $\beta>0$,  $\bbp\in(1,\infty)^2$  and $\bba\cdot d/\bbp<1$, and let $Z$ be a weak solution to \eqref{KSDE}. Then for any $\delta<0$ with $-\delta+\bba\cdot d/\bbp<1$, there is a constant $C>0$ such that for all $f\in L^\infty_T(\mB^\beta_{\bbp;\bba})$, for $m\geq1$
  \begin{align}\label{0321:02}
      \mE\left[\sup_{t\in[0,T]}\left|\int_0^t f(r,Z_r)\dif r\right|^m\right]\lesssim_C \|f\|_{L^\infty_T(\bB^{\delta}_{\bbp;\bba})}^m.
  \end{align}
\end{lemma}
\begin{proof}
    We consider the PDE \eqref{Aw2}. Then by \eqref{SDE304} and BDG  inequality, we have
    \begin{align*}
        I(f)&:=\mE\left[\sup_{t\in[0,T]}\left|\int_0^t f(r,Z_r)\dif r\right|^m\right]\lesssim \|u\|_{\infty}^m+\mE\left|\int_0^T|\nabla_v u(r,Z_r)|^2\dif r\right|^{m/2}\\
        &\lesssim \|u\|^m_{L^{\infty}_T(\bB^{1,1}_{\infty;\bba})}\lesssim \|u\|^m_{L^{\infty}_T(\bB^{1+\bba\cdot d/\bbp,1}_{\bbp;\bba})}.
    \end{align*}
    Since $1+\bba\cdot d/\bbp<2+\delta$, based on \eqref{FA00} and \eqref{AB2}, we complete the proof.
\end{proof}

Now we are ready to give the whole proof of  \cref{thmW}.

\begin{proof}[Proof of  \cref{thmW}:]
({\bf Existence of weak solutions}) It has been shown in \hyperref[proof:existence]{proof of the weak existence}.

({\bf Weak uniqueness}) Let $Z^1$ and $Z^2$ be two weak solutions of SDE \eqref{KSDE} with the same initial distribution.
For any $f\in  \bC^{\infty}_0$, by \cref{lem44} with $T=t$, we have
$$
\mE \int_0^t f(Z^1_s)\dif s=-\mE u(0,Z^1_0)=-\mE u(0,Z^2_0)=\mE \int_0^t (f+\lambda u)(Z^2_s)\dif s.
$$
Hence,
$$
\mE f(Z^1_t)=\mE (f+\lambda u)(Z^2_t).
$$
In particular, $Z^1$ and $Z^2$ have the same one dimensional marginal distribution. Thus,
 as in \cite[Theorem 4.4.3]{EK86}, by a standard method,
we obtain the uniqueness.

({\bf Pathwise uniqueness})
We use the well-known Zvonkin transformation to prove the strong stability, which automatically implies the pathwise uniqueness.
The strong uniqueness then follows by Yamada-Watanabe's theorem.
Recall that
$$
\DD,\wt\DD\in L^\infty_T(\bB^{2/3,\beta}_{\bbp;x,\bba}).
$$
Fix $T>0$ and $\lambda\geq 1$. Let $\bu:=(u_1,\cdots, u_d)$ solve the following backward PDE:
$$
\p_t \bu+(\Delta_v+v\cdot\nabla_x-\lambda)\bu+\DD\cdot\nabla_v \bu+\DD=0,\quad \bu(T)=0.
$$
By \eqref{FA001}, we have for any $\theta\in[0,2]$,
\begin{align}\label{AM3}
\|\bu\|_{L^{\infty}_{T}(\bB^{2/3,\theta+\beta}_{\bbp;x,\bba})}\lesssim_C \lambda^{\frac{\theta-2}{2}}\|\DD\|_{L^\infty_T(\bB^{2/3,\beta}_{\bbp;x,\bba})}.
\end{align}
Similarly, we define $\wt\bu$ by $\wt\DD$. Then \eqref{AM3} also holds for $\wt b$ and $\wt u$.
In particular, letting $\theta=1+\bba\cdot d/\bbp$ in \eqref{AM3},
by \eqref{IA2}, we obtain
\begin{align}
    \label{EST:U-1}
\|\nabla \bu\|_{\mL^\infty_T}\vee \|\nabla \wt\bu\|_{\mL^\infty_T}\lesssim \|\bu\|_{L^\infty_{T}(\bB^{2/3,\theta+\beta}_{\bbp;x,\bba})}\vee \|\wt\bu\|_{L^\infty_{T}(\bB^{2/3,\theta+\beta}_{\bbp;x,\bba})}\lesssim\lambda^{\frac{\bba\cdot d/\bbp-1}{2}},
\end{align}
where the implicit constant is independent of $\lambda\geq 1$.
Since $\bba\cdot d/\bbp<1$, one can choose $\lambda$ large enough so that
\begin{align}
    \label{EST:U-2}
    \|\nabla \bu\|_{\mL^\infty_T}\vee \|\nabla \wt\bu\|_{\mL^\infty_T}\leq\tfrac12.
\end{align}
Moreover, by \eqref{IA3} and \eqref{AM3} with taking $\theta=\frac{4}{3}$, we have
\begin{align} \label{EST:U-3}
    \|\nabla\nabla_v \bu\|_{L^\infty(\mL^{\bbp})}\lesssim \|\bu\|_{L^{\infty}_{T}(\bB^{2+\beta}_{\bbp;\bba})}\lesssim \|\bu\|_{L^{\infty}_{T}(\bB^{2/3,4/3+\beta}_{\bbp;x,\bba})}<\infty.
\end{align}
Now if we define
$$
\Phi(t,z)=\Phi(t,x,v):=(x,v+\bu(t,x,v)),
$$
then for $z=(x,v)$ and $z'=(x',v')$,
$$
|z-z'|\leq |\Phi(t,z)-\Phi(t,z')|+|\bu(t,z)-\bu(t,z')|\leq
|\Phi(t,z)-\Phi(t,z')|+\tfrac12|z-z'|,
$$
which implies
\begin{align}\label{Am1}
|z-z'|\leq 2|\Phi(t,z)-\Phi(t,z')|.
\end{align}
By \eqref{SDE304} with $f=-\DD$ and $f=(\DD-\wt \DD)\nabla_v \Phi$, one finds that
\begin{align}\label{SDE324}
\Phi(t, Z_t)=\Phi(0, Z_0)+\int^t_0 (V_s,\lambda u(s,Z_s))\dif s
+\int^t_0\nabla_v\Phi(s,Z_s)\dif W_s,
\end{align}
and for  $\wt Z$ which is any solution of SDE \eqref{SDE304} corresponding to $\wt\DD$ and initial value $\wt Z_0$,
\begin{align}\label{0321:00}
\Phi(t, \wt Z_t)=\Phi(0, \wt Z_0)+\int^t_0 (\wt V_s,\lambda u(s,\wt Z_s))\dif s
+\int^t_0\nabla_v\Phi(s,\wt Z_s)\dif W_s+R_t,
\end{align}
where
\begin{align*}
    R_t:=\int_0^t (b-\wt b)\cdot\nabla_v\Phi(s,\wt Z_s)\dif s.
\end{align*}
Then by BDG  inequality, we have
\begin{align}\label{est:difference-strong}
\mE\left(\sup_{s\in[0,t]}|Z_s-\wt Z_s|^2\right)
\lesssim& \mE|Z_0-\wt Z_0|^2+(1+\lambda^2\|\nabla u\|_{\mL_T^\infty}^2)\int^t_0\mE|Z_s-\wt Z_s|^2\dif s\nonumber\\
&+\mE\int_0^t|\nabla_v\Phi(s,Z_s)-\nabla_v\Phi(s,\wt Z_s)|^2\dif s+\mE[\sup_{s\in[0,t]}|R_s|^2].
\end{align}
Let $\mathcal{M}$ be the Hardy--Littlewood maximal operator defined as
\begin{align}\label{def:M}
	\cmm f(x):=\sup_{0<r<\infty}\frac{1}{|\cB_r|}\int_{\cB_r}f(x+y)dy, \quad \cB_r:=\{x\in\mathbb{R}^{2d}:|x|<r\},\quad r>0.
\end{align}
It is well-known that $\cmm$ is bounded on $\mL_\bbp(\R^{2d})$ (\cite[Appendix Lemma 5.4(i)]{Zhang2011}). Thus we have
\begin{align}
\label{max}
\Vert\mathcal{M} f\Vert_{ \mL^\bbp}\lesssim \Vert f\Vert_{ \mL^\bbp}.
\end{align}
Based on \cite[Appendix Lemma 5.4(ii)]{Zhang2011}, we have
\begin{align*}
   |\nabla_v\Phi(s,Z_s)-\nabla_v\Phi(s,\wt Z_s)|\le |Z_s-\wt Z_s|A_s,
\end{align*}
where
\begin{align*}
    A_s:=\cM(|\nabla\nabla_v\Phi|)(s,Z_s)+\cM(|\nabla\nabla_v\Phi|)(s,\wt Z_s).
\end{align*}
Since
\begin{align*}
\|\cM(|\nabla\nabla_v\Phi|)\|_{L^\infty_T(\mL^{\bbp})}\lesssim \|\nabla\nabla_v\Phi\|_{L^\infty_T(\mL^{\bbp})}<\infty,
\end{align*}
Just as standard procedure, say for instance, by \cite{Le2022} and \cite[Proof of Lemma 2.3]{GL}, it follows  that
\begin{align*}
    \mE e^{\kappa \int_0^T|A_s|^2\dif s}<\infty,\quad \forall \kappa>0.
\end{align*}
Then, by \eqref{est:difference-strong} and stochastic Gr\"onwall's inequality, we have
\begin{align*}
   \mE\left(\sup_{s\in[0,t]}|Z_s-\wt Z_s|^2\right)
\lesssim& \mE|Z_0-\wt Z_0|^2+\mE[\sup_{s\in[0,t]}|R_s|^2].
\end{align*}
When $b=\wt b$, it is the pathwise uniqueness. Now we show the stability. In view of \eqref{0321:02}, by applying \cref{lem:A2} and \eqref{Sob1} we have for any $\delta\in(-1+\bba\cdot d/\bbp,0)$,
\begin{align*}
   \mE[\sup_{s\in[0,t]}|R_s|^2]&\lesssim \|(b-\wt b)\cdot\nabla_v\Phi\|_{L^\infty_T(\bB^{\delta}_{\bbp;\bba})}\lesssim \|b-\wt b\|_{L_T^\infty(\bB^{\delta}_{\bbp;\bba})}(1+\|\nabla_v u\|_{L^\infty_T(\bB^{-\delta}_{\infty;\bba})})\\
   &\lesssim\|b-\wt b\|_{L_T^\infty(\bB^{\delta}_{\bbp;\bba})}\left(1+\| u\|_{L^\infty_T(\bB^{1-\delta+\bba\cdot d/\bbp}_{\bbp;\bba})}\right)\\
   &\lesssim\|b-\wt b\|_{L_T^\infty(\bB^{\delta}_{\bbp;\bba})}\left(1+\| u\|_{L^\infty_T(\bB^{2}_{\bbp;\bba})}\right)\\
   &\lesssim \|b-\wt b\|_{L_T^\infty(\bB^{\delta}_{\bbp;\bba})}
\end{align*}
and complete the proof.
\end{proof}

 \section{Quantitative Regularization of the singular functionals} \label{sec:quantitative-est}
In this section we study the necessary estimates on the functional that appears in EM-scheme. We start with no drift case. By scaling we can take $T=1$ without losing generality.
 \begin{lemma}\label{lem:est-nob}
   Let $W$ be the standard $d$-dimensional Brownian motion and $f:\mR_+\times\mR^{2d}\to\mR$ be measurable. Let  $M_t(z):=(x+tv+\int_0^tW_s\dif s,v+W_t)$, $t\in[0,1]$, $z=(x,v)\in\mR^{2d}$. Let $\epsilon\in(0,\frac{1}{2})$.
   Then for any $s,t\in[S,1]_\leq^2$ and $m\in\mN$, 
for $f\in{L^\infty_T(\mB_{\bbp,\bba}^{3\beta})}$ and $g\in L^\infty_T(\bC^{\beta'}_a)$ with $\beta\in(0,\frac{1}{3})$, $\beta'\in(1-2\epsilon,1)$, $\bbp=(p_x,p_v)$ and $p_x,p_v\in(4d,\infty)$, for any $m<p:=\min(p_x,p_v)$
 \begin{align}\label{ieq:est-nob-sob}
\Big\|\int_s^tg(r,M_r(z))\(\Gamma_{r-k_n(r)}&f(r,M_{k_n(r)}(z))-f(r,M_r(z))\)\dif r\Big\|_{L^m(\Omega)}
\nonumber\\&\le N \|g\|_{L_T^\infty(\bC^{\beta'}_a)}\|f\|_{L^\infty_T(\mB_{\bbp,\bba}^{3\beta})}|t-s|^{\frac{1}{2}+\epsilon}n^{-\frac{1+\beta}{2}+\epsilon}S^{-\bba\cdot\frac{d}{2\bbp}}
\end{align}
where $N=N(m,\bbp,d,\beta,\beta',\epsilon)$.
 \end{lemma}
 \begin{proof}
 Define $k$ by $\frac{k}{n}=k_n(s)$ by $s\in[0,1]$. Let
     \begin{align}
         \label{def:A}
         A_{s,t}:=\mE_s[\mathcal{A}_{s,t}]:=\mE_s \big[\int_s^tg(r,M_s(z))\big(\Gamma_{r-k_n(r)}f(r,M_{k_n(r)}(z))-f(r,M_r(z))\big)\dif r\big].
     \end{align}
     In order to apply \cref{lem:SSL-shift} we first verify (i) of \cref{lem:SSL-shift}. It reads
     \begin{align}
         \label{def:hatA}
          \|A_{s,t}\|_{L^m(\Omega)}\leq \|g\|_\infty\hat{A}_{s,t}:=\|g\|_\infty\int_s^t \big\|\mE_s[\Gamma_{r-k_n(r)}f(r,M_{k_n(r)}(z))-f(r,M_r(z))]\big\|_{L^m(\Omega)}\dif r.
     \end{align}
     Depending on the relations among $s$, $t$ and $k$, we divide the  analysis into the following three cases:

         \begin{itemize}
         \item[\bf Case I.] $t\in(\frac{k+4}{n},1]$.
     \end{itemize}
     We write
     \begin{align}
         \label{est:hatA-12}
         \hat{A}_{s,t}&\leq \big(\int_s^\frac{k+4}{n}+\int_\frac{k+4}{n}^t\big)\big\|\mE_s[\Gamma_{r-k_n(r)}f(r,M_{k_n(r)}(z))-f(r,M_r(z))]\big\|_{L^m(\Omega)}\dif r
     =:I_1+I_2.
     \end{align}
     First it reads
     \begin{align*}
         I_1=\Big(\int_s^\frac{k+1}{n}+\int^\frac{k+4}{n}_\frac{k+1}{n}\Big)\big\|\mE_s[\Gamma_{r-k_n(r)}f(r,M_{k_n(r)}(z))-f(r,M_r(z))]\big\|_{L^m(\Omega)}\dif r
     =:I_{11}+I_{12}.
     \end{align*}
     For $I_{11}$ we have
     \begin{align}\label{est:def-I11}
         I_{11}=&\int_s^\frac{k+1}{n}\big\|\mE_s[\Gamma_{r-k_n(r)}f(r,M_{k_n(r)}(z))-f(r,M_r(z))]\big\|_{L^m(\Omega)}\dif r
        \nonumber \\=&\int_s^\frac{k+1}{n}\big\|\mE_s[\Gamma_{r-\frac{k}{n}}f(r,M_{\frac{k}{n}}(z))-f(r,M_r(z))]\big\|_{L^m(\Omega)}\dif r
          \nonumber    \\=&\int_s^\frac{k+1}{n}\big\|\Gamma_{r-\frac{k}{n}}f(r,M_{\frac{k}{n}}(z))-P_{r-s}f(r,M_s(z))\big\|_{L^m(\Omega)}\dif r
           \nonumber    \\\lesssim & \int_s^\frac{k+1}{n}\big\|\Gamma_{r-s}f(r,M_s(z))-P_{r-s}f(r,M_s(z))\big\|_{L^m(\Omega)}\dif r
              \nonumber \\&+\int_s^\frac{k+1}{n}\big\|{\Gamma_{r-\frac{k}{n}}}f(r,M_{\frac{k}{n}}(z))-\Gamma_{r-s}f(r,M_s(z))\big\|_{L^m(\Omega)}\dif r
       =: I_{111}+ I_{112}.
     \end{align}
     By \eqref{est:E-f-Z} and  \eqref{est:semi-P-So} with $\delta=\frac{\beta}{2}$
     we get
     \begin{align}\label{est:I111}
        I_{111}=  &\int_s^\frac{k+1}{n}\big\|\big(\Gamma_{r-s}f(r,M_s(z))-P_{r-s}f(r,M_s(z))\big)\big\|_{L^m(\Omega)}\dif r
         \nonumber   \\\lesssim&
        s^{-\bba\cdot\frac{d}{2\bbp}} \int_s^\frac{k+1}{n}\big\|\big(\Gamma_{r-s}f(r)-P_{r-s}f(r)\big)\big\|_{\bbp}\dif r
           \nonumber \\\lesssim&  S^{-\bba\cdot\frac{d}{2\bbp}} \|f\|_{L^\infty_T(\mB_{\bbp,\bba}^{3\beta})} \int_s^\frac{k+1}{n}(r-s)^{\frac{\beta}{2}}\dif r
              \nonumber \\\lesssim&  S^{-\bba\cdot\frac{d}{2\bbp}} \|f\|_{L_T^\infty(\mB_{\bbp,\bba}^{3\beta})}n^{-\frac{\beta}{2}}(\frac{k+1}{n}-s)
             \lesssim S^{-\bba\cdot\frac{d}{2\bbp}}n^{-\frac{\beta+1}{2}+\epsilon} \|f\|_{L^\infty_t\mB_{\bbp,\bba}^{3\beta,}}|t-s|^{\frac{1}{2}+\epsilon}.
     \end{align}
    Besides, for $I_{112}$ we have by \eqref{est:E-dif-f-Z} (accordingly with taking $s$ (here below) as $t$, $\frac{k}{n}$ as $s$ and $r=0$ therein)
     \begin{align}\label{est:I112}
       I_{112}=   &\int_s^\frac{k+1}{n}\big\|\big({ \Gamma_{s-\frac{k}{n}}\Gamma_{r-s}}f(r,M_{\frac{k}{n}}(z))-\Gamma_{r-s}f(r,M_s(z))\big)\big\|_{L^m(\Omega)}\dif r
          \nonumber  \\\lesssim & \int_s^\frac{k+1}{n} (\frac{k}{n})^{-\bba\cdot\frac{d}{2\bbp}}|s-\frac{k}{n}|^{\frac{\beta}{6}}|s-\frac{k}{n}|^{\frac{\beta}{3}}\|\Gamma_{r-s}f\|_{L^\infty_T(\mB_{\bbp,\bba}^{\beta})} \dif r
        \nonumber     \\\lesssim & (\frac{k+1}{n} -s)(\frac{k}{n})^{-\bba\cdot\frac{d}{2\bbp}}|s-\frac{k}{n}|^\frac{\beta}{2}\|\Gamma_{r-s} f\|_{L^\infty_T(\mB_{\bbp,\bba}^{\beta})}
          \nonumber     \\\lesssim & S^{-\bba\cdot\frac{d}{2\bbp}}n^{-\frac{\beta+1}{2}+\epsilon} \|f\|_{L^\infty_T(\mB_{\bbp,\bba}^{3\beta})}|t-s|^{\frac{1}{2}+\epsilon}.
     \end{align}
   In the last inequality above we applied \eqref{0222:00}. Then  \eqref{est:def-I11} together with \eqref{est:I111} and \eqref{est:I112}  yield
     \begin{align}
         \label{est:I11} I_{11}\lesssim S^{-\bba\cdot\frac{d}{2\bbp}}n^{-\frac{\beta+1}{2}+\epsilon} \|f\|_{L^\infty_t\mB_{\bbp,\bba}^{3\beta}}|t-s|^{\frac{1}{2}+\epsilon}.
     \end{align}
  Moreover,
  by applying \eqref{est:E-f-Z} and \eqref{est:semi-P-So} with $\delta=\frac{\beta+1}{2}$ we get
     \begin{align} \label{est:I12}
        I_{12}\lesssim &\int^\frac{k+4}{n}_\frac{k+1}{n}r^{-\bba\cdot\frac{d}{2\bbp}}\|P_{r-s}f-P_{k_n(r)-s}\Gamma_{r-k_n(r)}f\|_{\bbp}\dif r
       \nonumber      \\\lesssim &\int_\frac{k+1}{n}^\frac{k+4}{n} r^{-\bba\cdot\frac{d}{2\bbp}} n^{-\frac{\beta+1}{2}}(k_n(r)-s)^{-\frac{1}{2}}\|f\|_{L^\infty_T(\mB_{\bbp,\bba}^{3\beta})}\dif r
        \nonumber    \\\lesssim & s^{-\bba\cdot\frac{d}{2\bbp}}n^{-\frac{\beta+1}{2}} \|f\|_{L^\infty_T(\mB_{\bbp,\bba}^{3\beta})}|t-s|^{\frac{1}{2}}\lesssim S^{-\bba\cdot\frac{d}{2\bbp}}n^{-\frac{3\beta+1}{2}+\epsilon} \|f\|_{L^\infty_T(\mB_{\bbp,\bba}^{3\beta})}|t-s|^{\frac{1}{2}+\epsilon}
     \end{align}
     since $n^{-1}\lesssim|t-s|$.  Similarly for $I_2$ we again by  applying \eqref{est:E-f-Z} and \eqref{est:semi-P-So} with  $\delta=\frac{\beta+1}{2}$ then get
     \begin{align}\label{est:I2}
         I_{2} \lesssim &\int_\frac{k+4}{n}^t\|P_{r-s}f-P_{k_n(r)-s}\Gamma_{r-k_n(r)}f\|_{\bbp}\dif r
      \nonumber       \\\lesssim &\int_\frac{k+4}{n}^t r^{-\bba\cdot\frac{d}{2\bbp}} n^{-\frac{\beta+1}{2}}(k_n(r)-s)^{-\frac{1}{2}}\|f\|_{L^\infty_T(\mB_{\bbp,\bba}^{3\beta})}\dif r
         \nonumber  \\\lesssim & s^{-\bba\cdot\frac{d}{2\bbp}}n^{-\frac{\beta+1}{2}} \|f\|_{L^\infty_T(\mB_{\bbp,\bba}^{3\beta})}|t-s|^{\frac{1}{2}}\lesssim S^{-\bba\cdot\frac{d}{2\bbp}}n^{-\frac{\beta+1}{2}+\epsilon} \|f\|_{L^\infty_T(\mB_{\bbp,\bba}^{3\beta})}|t-s|^{\frac{1}{2}+\epsilon}.
     \end{align}
     In the end \eqref{est:I11}, \eqref{est:I12} and  \eqref{est:I2} together with notation \eqref{est:hatA-12} show that in this case
     \begin{align}
         \label{est:hatA-So}
            \hat{A}_{s,t}\lesssim S^{-\bba\cdot\frac{d}{2\bbp}}n^{-\frac{\beta+1}{2}+\epsilon} \|f\|_{L^\infty_T(\mB_{\bbp,\bba}^{3\beta})}|t-s|^{\frac{1}{2}+\epsilon}.
     \end{align}
 \begin{itemize}
         \item[\bf Case II.] $t\in[s,\frac{k+4}{n}]$, $k\geq1$.
     \end{itemize}
     In this case
     \begin{align}\label{est:hat-At}
         \hat{A}_{s,t}=&\Big(\int_s^{t\wedge\frac{k+1}{n}}+\int^t_{t\wedge\frac{k+1}{n}}  \Big)\big\|\mE_s[\Gamma_{r-k_n(r)}f(r,M_{k_n(r)}(z))-f(r,M_r(z))]\big\|_{L^m(\Omega)}\dif r
       \nonumber  \\=&\int_s^{t\wedge\frac{k+1}{n}} \big\|\Gamma_{k_n(r)-s}f(r,M_{k_n(r)}(z))-P_{r-s}f(r,M_s(z))\big\|_{L^m(\Omega)}\dif r
        \nonumber  \\&+\int^t_{t\wedge\frac{k+1}{n}}\big\|\mE_s[\Gamma_{r-k_n(r)}f(r,M_{k_n(r)}(z))-f(r,M_r(z))]\big\|_{L^m(\Omega)}\dif r
          \nonumber  \\\lesssim&
          \int_s^{t\wedge\frac{k+1}{n}} \big\|\Gamma_{r-s}f(r,M_s(z))-P_{r-s}f(r,M_s(z))\big\|_{L^m(\Omega)}\dif r
           \nonumber  \\&+\int_s^{t\wedge\frac{k+1}{n}} \big\|\Gamma_{r-s}f(r,M_s(z))-\Gamma_{k_n(r)-s}f(r,M_{k_n(r)}(z))\big\|_{L^m(\Omega)}\dif r
        \nonumber  \\&+\int^t_{t\wedge\frac{k+1}{n}}\big\|\Gamma_{r-k_n(r)}f(r,M_{k_n(r)}(z))-f(r,M_r(z))\big\|_{L^m(\Omega)}\dif r=:S_1+S_2+S_3.
     \end{align}
    We follow the ideas of getting \eqref{est:I111} and \eqref{est:I112} then obtain
    \begin{align}\label{est:S1+S2}
      &S_1+S_2\nonumber  \\\lesssim &  S^{-\bba\cdot\frac{d}{2\bbp}} \|f\|_{L^\infty_T(\mB_{\bbp,\bba}^{3\beta})} \int_s^{t\wedge\frac{k+1}{n}}(r-s)^{\frac{\beta}{2}}\dif r+({t\wedge\frac{k+1}{n}} -s)(\frac{k}{n})^{-\frac{2d}{p}}|s-\frac{k}{n}|^\frac{3\beta}{2} \|f\|_{L^\infty_T(\mB_{\bbp,\bba}^{3\beta})}
       \nonumber  \\=& S^{-\bba\cdot\frac{d}{2\bbp}} \|f\|_{L^\infty_T(\mB_{\bbp,\bba}^{3\beta})} ({t\wedge\frac{k+1}{n}}-s)^{\frac{\beta}{2}+1}+({t\wedge\frac{k+1}{n}} -s)(\frac{k}{n})^{-\bba\cdot\frac{d}{2p}}|s-\frac{k}{n}|^\frac{\beta}{2}\|f\|_{L^\infty_T(\mB_{\bbp,\bba}^{3\beta})}
       \nonumber  \\\leq &S^{-\bba\cdot\frac{d}{2\bbp}} \|f\|_{L^\infty_T(\mB_{\bbp,\bba}^{3\beta})}(t-s)^{\frac{1}{2}+\epsilon}({\frac{k+1}{n}}-s)^{\frac{\beta+1}{2}+\epsilon}\nonumber\\&\qquad+(t-s)^{\frac{1}{2}+\epsilon}({\frac{k+1}{n}} -s)^{\frac{1}{2}-\epsilon}S^{-\bba\cdot\frac{d}{2p}}n^{-\frac{\beta}{2}}\|f\|_{L^\infty_T(\mB_{\bbp,\bba}^{3\beta})}
        \nonumber   \\\lesssim& S^{-\bba\cdot\frac{d}{2\bbp}}n^{-\frac{\beta+1}{2}+\epsilon} \|f\|_{L^\infty_T(\mB_{\bbp,\bba}^{3\beta})}|t-s|^{\frac{1}{2}+\epsilon}.
    \end{align}
    Now we turn to $S_3$. By the idea from \eqref{est:I12} it follows
    \begin{align}\label{est:S3}
    S_3\lesssim &     \int^t_{t\wedge\frac{k+1}{n}}r^{-\bba\cdot\frac{d}{2p}}\|P_{r-s}f-P_{k_n(r)-s}\Gamma_{r-k_n(r)}f\|_{\bbp}\dif r
       \nonumber      \\\lesssim &\int^t_{t\wedge\frac{k+1}{n}}r^{-\bba\cdot\frac{d}{2\bbp}} n^{-\frac{\beta+1}{2}+\epsilon}(k_n(r)-s)^{-\frac{1}{2}+\epsilon}\|f\|_{L^\infty_T(\mB_{\bbp,\bba}^{\beta})}\dif r
        \nonumber    \\\lesssim & s^{-\bba\cdot\frac{d}{2\bbp}}n^{-\frac{\beta+1}{2}+\epsilon} \|f\|_{L^\infty_T(\mB_{\bbp,\bba}^{3\beta})}|t-(t\wedge\frac{k+1}{n})|^{\frac{1}{2}+\epsilon} \nonumber    \\\lesssim & S^{-\bba\cdot\frac{d}{2\bbp}}n^{-\frac{\beta+1}{2}+\epsilon} \|f\|_{L^\infty_T(\mB_{\bbp,\bba}^{3\beta})}|t-s|^{\frac{1}{2}+\epsilon}.
    \end{align}
    We collect \eqref{est:S1+S2} and \eqref{est:S3} together with \eqref{est:hat-At} then obtain
    \begin{align}
        \label{est:hatA-SS}
           \hat{A}_{s,t}\lesssim S^{-\bba\cdot\frac{d}{2\bbp}}n^{-\frac{\beta+1}{2}+\epsilon} \|f\|_{L^\infty_T(\mB_{\bbp,\bba}^{3\beta})}|t-s|^{\frac{1}{2}+\epsilon}.
    \end{align}

         \begin{itemize}
         \item[\bf Case III.] $t\in[s,\frac{k+4}{n}]$, $k=0$.
     \end{itemize}
      We estimate $\hat{A}$ by different ways according to whether $\beta>\bba\cdot\frac{d}{\bbp}$ or not. If $\beta>\bba\cdot\frac{d}{\bbp}$ then, we use the embedding from \eqref{Sob1} then get
      \begin{align*}
     \|f\|_{L^\infty_T(\bC^{\beta-\bba\cdot\frac{d}{\bbp}}_a)}\lesssim   \|f\|_{L^\infty_T(\mB_{\bbp,\bba}^{\beta})}.
      \end{align*}
      Notice that
      \begin{align*}
      \big \|\sup_{r\in[0,1]}|(r-k_n(r))W_{k_n(r)}-\int_{k_n(r)}^rW_u\dif u|^{\frac{\beta}{3}}\big\|_{L^m(\Omega)}
    &\lesssim_mn^{-\frac{\beta}{2}},
      \nonumber    \\\big \|\sup_{r\in[0,1]}|W_r-W_{k_n(r)}|^{\beta})\big\|_{L^m(\Omega)}
        &\lesssim_m n^{-\frac{\beta}{2}},
     \end{align*}
    it combines with the fact that $|t-s|\leq S\leq s\leq n^{-1}$ gives us
      \begin{align}\label{est:hatAsma1}
       \hat{A}_{s,t}\lesssim     &\|f\|_{L^\infty_T(\bC^{\beta-\bba\cdot\frac{d}{\bbp}}_a)}   \int_s^{t}n^{-\frac{\beta-\bba\cdot\frac{d}{\bbp}}{2}}\dif r=\|f\|_{L^\infty_T(\bC^{\beta-\frac{4d}{p}}_a)}   n^{-\frac{\beta-\bba\cdot\frac{d}{\bbp}}{2}}|t-s|
      \nonumber \\\lesssim & S^{-\frac{d}{2p}}n^{-\frac{\beta+1}{2}+\epsilon} \|f\|_{L^\infty_T(\mB_{\bbp,\bba}^{\beta})}|t-s|^{\frac{1}{2}+\epsilon}.
      \end{align}
      If  $\beta\leq\bba\cdot\frac{d}{\bbp}$, then again due to $|t-s|\leq S\leq s\leq n^{-1}$ we have
       \begin{align}\label{est:hatAsm2}
       \hat{A}_{s,t}\leq     &2\|f\|_\infty|t-s| \leq\|f\|_\infty n^{-\frac{1}{2}+\epsilon}|t-s|^{\frac{1}{2}+\epsilon}
    \nonumber \\\lesssim &\|f\|_\infty n^{-\frac{1+\beta}{2}+\epsilon}|t-s|^{\frac{1}{2}+\epsilon}S^{-\frac{\beta}{2}}  \lesssim  S^{-\bba\cdot\frac{d}{2\bbp}}n^{-\frac{\beta+1}{2}+\epsilon} \|f\|_{L^\infty_T(\mB_{\bbp,\bba}^{\beta})}|t-s|^{\frac{1}{2}+\epsilon}.
      \end{align}
      Hence \eqref{est:hatA-SS} and \eqref{est:hatAsma1} show that in this case
      \begin{align}
          \label{est:hatA3}
          \hat{A}_{s,t}\lesssim  S^{-\bba\cdot\frac{d}{2\bbp}}n^{-\frac{\beta+1}{2}+\epsilon} \|f\|_{L^\infty_T(\mB_{\bbp,\bba}^{\beta})}|t-s|^{\frac{1}{2}+\epsilon}.
      \end{align}

    So far we have verified the condition (1) from \cref{lem:SSL-shift} by \eqref{est:hatA-So}, \eqref{est:hatA-SS} and \eqref{est:hatA3}. For condition (2) we have   for any $(s,u,t)\in[0,1]_\leq^3$
     \begin{align*}
         \delta A_{s,u,t}=\big(\mE_s&\int_u^tg(r,M_{s}(z))-\mE_u\int_u^tg(r,M_{u}(z))\big)\\&\mE_u[\big(f(r,M_r(z))-\Gamma_{r-k_n(r)}f(r,M_{k_n(r)}(z))]\dif r,
     \end{align*}
so
\begin{align*}
    \mE_s \delta A_{s,u,t}=\big(\int_u^t&\mE_s[g(r,M_{s}(z))-g(r,M_{u}(z))]\big)\\&\mE_u[\big(f(r,M_r(z))-\Gamma_{r-k_n(r)}f(r,M_{k_n(r)}(z))]\dif r,
\end{align*}
by \eqref{est:hatA-So}, \eqref{est:hatA-SS} and \eqref{est:hatA3}  we have
\begin{align}
    \label{est:Asut-So}
  &  \| \mE_s \delta A_{s,u,t}\|_{L^m(\Omega)}
    \nonumber\\&\lesssim \int_u^t\|g(r,M_s(z))-g(r,M_u(z)\|_{L^{\frac{pm}{p-m}}(\Omega)}  \nonumber\\&\qquad\qquad\qquad\qquad\qquad\big\|\mE_u[\big(\Gamma_{r-k_n(r)}f(r,M_{k_n(r)}(z))-f(r,M_r(z))\big)]\big\|_{L^{p}(\Omega)}\dif r
    \nonumber\\&\lesssim \|g\|_{L_T^\infty(\bC^{\beta'}_a})|s-u|^{\frac{\beta'}{2}}S^{-\bba\cdot\frac{d}{\bbp}}n^{-\frac{\beta+1}{2}+\epsilon} \|f\|_{L_T^\infty(\mB_{\bbp,\bba}^{3\beta})}|t-s|^{\frac{1}{2}+\epsilon}
      \nonumber\\&\lesssim \|g\|_{L_T^\infty(\bC^{\beta'}_a)}\|f\|_{L_T^\infty(\mB_{\bbp,\bba}^{3\beta})}S^{-\bba\cdot\frac{d}{2\bbp}}|t-s|^{1+\epsilon'} n^{-\frac{\beta+1}{2}+\epsilon}
\end{align}
with $\epsilon'=\epsilon+\frac{\beta'-1}{2}>0$ since $\beta'>1-2\epsilon$ by assumption. 

    Applying shifted stochastic sewing we get \eqref{ieq:est-nob-sob} holds.  The proof completes.
      \end{proof}
  \begin{corollary}
   \label{cor:est-descr}
     Let $W$ be a standard $d$-dimensional Brownian motion and $f:\mR_+\times\mR^{2d}\to\mR$ be measurable. Let  $M_t(z):=(x+tv+\int_0^tW_s\dif s,v+W_t)$, $\forall z=(x,v)\in\mR^{2d}$, $t\in[0,1]$. Let $\epsilon\in(0,\frac{1}{2})$.
   Then 
   for any $s,t\in[S,1]_\leq^2$ and $m\in\mN$,
for $f\in L_T^\infty(\mB_{\bbp,\bba}^{3\beta})$  and $g\in L_T^\infty(\bC^{\beta'}_a)$ with $\beta\in(0,\frac{1}{3})$, $\beta'\in(1-2\epsilon,1)$, $\bbp=(p_x,p_v)$ and $p_x,p_v\in(4d,\infty)$, for any $m< p:=\min(p_x,p_v)$ we have
 \begin{align}\label{est:nob-sob}
\Big\|\sup_{t\in(0,1)}\big|\int_0^tg(r,M_r(z))\(\Gamma_{r-k_n(r)}f(r,M_{k_n(r)}(z))&-f(r,M_r(z))\)\dif r\big|\Big\|_{L^m(\Omega)}\nonumber\\&\le\|g\|_{L^\infty_T(\bC^{\beta'}_a)}\|f\|_{L_T^\infty(\mB_{\bbp,\bba}^{3\beta})}n^{-\frac{1+\beta}{2}+\epsilon}
\end{align}
with $N=N(m,\bbp,d,\beta,\epsilon)$;
  \end{corollary}
\begin{proof}
    We follow the idea from \cite[Lemma 3.4]{DGL}. Denote
    $K:=\|g\|_{L^\infty_T(\bC^{\beta'}_a)}\|f\|_{L_T^\infty(\mB_{\bbp,\bba}^{3\beta})}.$
    We first show that for any $q\in(m,p)$ we have for any $t\in[0,1]$
    \begin{align}
        \label{est:q-mom}
      \Big\|\int_0^tg(r,M_r(z))\(\Gamma_{r-k_n(r)}f(r,M_{k_n(r)}(z))&-f(r,M_r(z))\)\dif r\Big\|_{L^m(\Omega)} \lesssim  Kn^{-\frac{1+\beta}{2}+\epsilon}.
    \end{align}
    By \cref{lem:est-nob} \eqref{ieq:est-nob-sob} and the fact that $d<p$ we have for any integer $j$
    \begin{align*}
       \Big\|  \int_{t2^{-j}}^t&g(r,M_r(z))\(\Gamma_{r-k_n(r)}f(r,M_{k_n(r)}(z))-f(r,M_r(z))\)\dif r\Big\|_{L^m(\Omega)}  \\&\lesssim
       \sum_{k=0}^{j-1}\Big\| \int_{t2^{k-j}}^{t2^{k-j+1}}g(r,M_r(z))\(\Gamma_{r-k_n(r)}f(r,M_{k_n(r)}(z))-f(r,M_r(z))\)\dif r\Big\|_{L^m(\Omega)}
       \\&\lesssim  Kn^{-\frac{1+\beta}{2}+\epsilon}\sum_{k=0}^{j-1}(t2^{k-j})^{\frac{1}{2}+\epsilon}(t2^{k-l})^{-\bba\cdot\frac{d}{\bbp}}
        \\&\lesssim  Kn^{-\frac{1+\beta}{2}+\epsilon}\sum_{k=0}^{j-1} 2^{-j\epsilon} 2^{k\epsilon}\lesssim  Kn^{-\frac{1+\beta}{2}+\epsilon}2^{-j\epsilon} \frac{2^{j\epsilon}-1}{2^\epsilon-1}\lesssim  Kn^{-\frac{1+\beta}{2}+\epsilon}.
    \end{align*}
    It yields \eqref{est:q-mom} after letting $j\rightarrow\infty$ and applying Fatou’s lemma.  Then we get from \eqref{est:q-mom} that for any $t\in[0,1]$ we have
    \begin{align}\label{est:t-1}
 \Big\|\int_t^1g(r,M_r(z))\(\Gamma_{r-k_n(r)}f(r,M_{k_n(r)}(z))&-f(r,M_r(z))\)\dif r\Big\|_{L^m(\Omega)} \lesssim  Kn^{-\frac{1+\beta}{2}+\epsilon}.
    \end{align}
 Let $\tau\leq 1$ be a bounded stopping time which takes only finitely many values $t_1,\ldots,t_k$. Observe
 \begin{align}
     \label{eq:tau-1}
      \Big\|\int_\tau^1&g(r,M_r(z))\(\Gamma_{r-k_n(r)}f(r,M_{k_n(r)}(z))-f(r,M_r(z))\)\dif r\Big\|_{L^m(\Omega)}^m
   \nonumber   \\&=\sum_{i=1}^k \Big\|\mI_{{\tau=t_i}}\int_{t_i}^1\(\Gamma_{r-k_n(r)}f(r,M_{k_n(r)}(z))-f(r,M_r(z))\)\dif r\Big\|_{L^m(\Omega)}^m,
 \end{align}
 moreover
 \begin{align}
     \label{est:tau-1}
     \Big\|&\mI_{{\tau=t_i}}\int_{t_i}^1g(r,M_r(z))\(\Gamma_{r-k_n(r)}f(r,M_{k_n(r)}(z))-f(r,M_r(z))\)\dif r\Big\|_{L^m(\Omega)}^m
      \nonumber   \\\lesssim&\mE\Big(\mI_{\tau=t_i}\Big|\int_{k_n(t_i)+n^{-1}}^1g(r,M_r(z))\(\Gamma_{r-k_n(r)}f(r,M_{k_n(r)}(z))-f(r,M_r(z))\)\dif r\Big|^m\Big) \nonumber \\&\qquad\qquad\qquad+n^{-m}\mP(\tau=t_i).
 \end{align}
 Since $f$ is bounded, if $1-t_i\leq\frac{3}{n}$ we easily get
 \begin{align}
     \label{est:ti}
     \mE\Big(\mI_{\tau=t_i}\Big|\int_{k_n(t_i)+n^{-1}}^1g(r,M_r(z))\(\Gamma_{r-k_n(r)}f(r,M_{k_n(r)}(z))-&f(r,M_r(z))\)\dif r\Big|^m\Big)\nonumber \\&\lesssim n^{-m}\mP(\tau=t_i),
 \end{align}
 if $1-t_i\geq\frac{3}{n}$ then we have
 \begin{align}
     \label{est:ti+}
  \mE\Big(&\mI_{\tau=t_i}\Big|\int_{k_n(t_i)+n^{-1}}^1g(r,M_r(z))\(\Gamma_{r-k_n(r)}f(r,M_{k_n(r)}(z))-f(r,M_r(z))\)\dif r\Big|^m\Big)  \nonumber    \\=&\mE(\mI_{\tau=t_i}H(M_{k_n(t_i)+n^{-1}}(z)))
 \end{align}
 with
 \begin{align*}
     H(y):= &\mE\Big(\Big|\int_{k_n(t_i)+n^{-1}}^1g(r,M_r(z)-M_{k_n(t_i)+n^{-1}}(z)+y)\\&\qquad\qquad\(\Gamma_{r-k_n(r)}f(r,M_{k_n(r)}(z)-M_{k_n(t_i)+n^{-1}}(z)+y)\\&\qquad\qquad\qquad-f(r,M_r(z)-M_{k_n(t_i)+n^{-1}}(z)+y)\)\dif r\Big|^m\Big)
     \\=&
     \mE\Big(\Big|\int^{1-k_n(t_i)-n^{-1}}_0g(r+k_n(t_i)+n^{-1},M_{r}(z)+y)\\&\qquad\qquad\(\Gamma_{r-k_n(r)}f(r+k_n(t_i)+n^{-1},M_{k_n(r)}(z)+y)\\&\qquad\qquad\qquad-f(r+k_n(t_i)+n^{-1},M_{r}(z)+y)\)\dif r\Big|^m\Big) \\\lesssim&Kn^{-\frac{1+\beta}{2}+\epsilon}.
 \end{align*}
 In the last equality we applied \eqref{est:t-1}. Again plugging it into \eqref{est:ti+} we obtain
 \begin{align*}
      \mE\Big(&\mI_{\tau=t_i}\Big|\int_{k_n(t_i)+n^{-1}}^1g(r,M_r(z))\(\Gamma_{r-k_n(r)}f(r,M_{k_n(r)}(z))-f(r,M_r(z))\)\dif r\Big|^m\Big)\\\lesssim &(Kn^{-\frac{1+\beta}{2}+\epsilon})^m\mP(\tau=t_i),
 \end{align*}
 then the above together with \eqref{est:tau-1}, \eqref{est:t-1} and \eqref{eq:tau-1} give us
 \begin{align}
     \label{ieq:tau}
       \Big\|\int_\tau^1&g(r,M_r(z))\(\Gamma_{r-k_n(r)}f(r,M_{k_n(r)}(z))-f(r,M_r(z))\)\dif r\Big\|_{L^m(\Omega)}\lesssim Kn^{-\frac{1+\beta}{2}+\epsilon}.
 \end{align}
 We know that the stopping time $\tau\leq 1$ can be approximated by simple ones $\tau_j:=k_n(\tau)+n^{-1}$ as $j\rightarrow\infty$. Hence by a approximation argument \eqref{ieq:tau} holds for any stopping times which is bounded by $1$. Combining with \eqref{est:q-mom} shows for any stopping times which is bounded by $1$ we have
 \begin{align*}
     \Big\|\int^\tau_0&g(r,M_r(z))\(\Gamma_{r-k_n(r)}f(r,M_{k_n(r)}(z))-f(r,M_r(z))\)\dif r\Big\|_{L^m(\Omega)}\lesssim Kn^{-\frac{1+\beta}{2}+\epsilon}.
 \end{align*}
 Finally we get \eqref{est:nob-sob} by Lenglart’s inequality (see e.g. \cite[Theorem 2.2]{GS}).
\end{proof}

\begin{proposition}[Khasminskii’s Lemma]\label{lem:Kaha}
Let $W$ be the standard $d$-dimensional Brownian motion and $f:\mR_+\times\mR^{2d}\to\mR$ be measurable so that 
\begin{align}
    \label{con.tamin}
   \forall (s,t)\in[0,1]_\leq^2, 
  \qquad |t-s|
  \sup_{r\in[s,t]}\|f(r)\|_\infty \lesssim  \sup_{r\in[s,t]}\|f(r)\|_\bbp.
\end{align}  Let  $M_t(z):=(x+tv+\int_0^tW_s\dif s,v+W_t)$, $\forall z=(x,v)\in\mR^{2d}, t\in[0,1]$. Let $\epsilon\in(0,\frac{1}{2})$.
   Then
   for any $t\in[0,1]$,
for $f$ being measurable,  there exists a constant $\kappa$ depending on $c, d,\bbp$ and uniformly on $n$  so that
\begin{align}
    \label{est:novikow}
    \mE\exp\Big(c\int_0^1f(r,M_{k_n(r)}(z))\dif r\Big)\lesssim \exp(\kappa\|f\|_{L_T^\infty (\mL^\bbp)}).
\end{align}
\end{proposition}
\begin{proof}
   Without loss of generality we assume $f$ is non-negative.  First we write  $\overline{t}:=k_n(t)+\frac{1}{n}$. Observe that for any $s,t\in[0,1]_\leq^2$
   \begin{align*}
 I_{s,t}:=    \int_s^tf(r,M_{k_n(r)}(z))\dif r=\big(\int_s^{\overline{s}\wedge t}+\int_{\overline{s}\wedge t}^t\big)f(r,M_{k_n(r)}(z))\dif r=:I_1+I_2.
   \end{align*}
   Then by \eqref{con.tamin}
   \begin{align*}
       \mE_sI_1\lesssim &\|f\|_\infty(\overline{s}\wedge t-s)
       \lesssim |t-s|
       \sup_{r\in[0,1]}\|f(r)\|_{\bbp} .
   \end{align*}
   Moreover by \eqref{eq:in-in} and \eqref{est:E-f-Z}
   \begin{align*}
       \mE_sI_2=&  \int_{\overline{s}\wedge t}^t \mE_sf\big(r, M_{k_n(r)}(z)-\Gamma_{k_n(r)-s}M_s(z)+\Gamma_{k_n(r)-s}M_s(z)\big)\dif r
       \\\lesssim&\int_{\overline{s}\wedge t}^t |k_n(r)-s|^{-\bba\cdot\frac{d}{\bbp}} \|f(r)\|_{\bbp}\dif r
       \\\lesssim&\int_{\overline{s}\wedge t}^t |r-s|^{-\bba\cdot\frac{d}{\bbp}} \|f(r)\|_{\bbp}\dif r
       \\\lesssim &|t-s|^{1-\bba\cdot\frac{d}{\bbp}} \|f\|_{L_T^\infty(\mL^\bbp)}.
   \end{align*}
  In the end we get for any $s,t\in[0,1]_\leq^2$ 
  \begin{align*}
      \mE_s I_{s,t}\lesssim |t-s|^{1-\bba\cdot\frac{d}{\bbp}} \|f\|_{L_T^\infty(\mL^\bbp)}.
  \end{align*}
  After applying \cite[Lemma 3.5]{LL}
we in the end get that for any $c>0$  \eqref{est:novikow} holds.

\end{proof}

In the end, we have the following result for $Z^n$  from \eqref{eq:SDE-EM}.
\bc\label{cor:final}
  Let $W$ be the standard $d$-dimensional Brownian motion and $f:\mR_+\times\mR^{2d}\to\mR$ be measurable. Let  $Z^n:=(X^n,V^n)\in\mR^{2d}$ be the solution from \eqref{eq:SDE-EM}. Let $\vartheta\in(0,(2\bba\cdot\frac{d}{\bbp})^{-1}).$ Let $\epsilon\in(0,\frac{1}{2})$.
   Then
   for any $s,t\in[S,1]_\leq^2$ and $m\in\mN$,
for $f\in L_T^\infty(\mB_{\bbp,\bba}^{3\beta})$,  $g\in L_T^\infty(\bC^{\beta'}_a)$ with $\beta\in(0,\frac{1}{3})$, $\beta'\in(1-2\epsilon,1)$, $\bbp=(p_x,p_v)$ and $p_x,p_v\in(4d,\infty)$, for any $m< p:=\min(p_x,p_v)$
 \begin{align}\label{AA05}
\Big\|\sup_{t\in(0,1)}\big|\int_0^tg(r,Z_r^n(z))\(\Gamma_{r-k_n(r)}f(r,Z_{k_n(r)}^n(z))&-f(r,Z_r^n(z))\)\dif r\big|\Big\|_{L^m(\Omega)}\nonumber\\&\le\|g\|_{L_T^\infty(\bC^{\beta'}_a)}\|f\|_{L_T^\infty(\mB_{\bbp,\bba}^{3\beta})}n^{-\frac{1+\beta}{2}+\epsilon}
\end{align}
with $N=N(m,\bbp,d,\beta,\beta',\epsilon, \vartheta)$.
\ec
\begin{proof}
Denote by
\begin{align} \label{def:rho}
\rho(\gamma):&=\exp\Big(-\int_0^1\Gamma_{s-k_n(s)}b_n(s,M_{k_n(s)})\dif W_s-\frac{\gamma^2}{2}\int_0^1\big|\Gamma_{s-k_n(s)}b_n(s,M_{k_n(s)})\big|^2\dif s\Big), \gamma>0,\\\nonumber
\mathcal{B}(M):&=\sup_{t\in(0,1)}\Big|\int_0^tg(r,M_r(z))\(\Gamma_{r-k_n(r)}f(r,M_{k_n(r)}(z))-f(r,M_r(z))\)\dif r\Big|,
\end{align}
where $M_t(z):=(x+tv+\int_0^tW_s\dif s,v+W_t)$, $\forall z=(x,v)\in\mR^{2d}, t\in[0,1]$, the same as defined in \cref{lem:Kaha}.

First, for any  for any $\kappa>1$ so that $\kappa m<p$,
by \eqref{est:nob-sob}
\begin{align}\label{est:nob}
    \|\mathcal{B}(M)\|_{L^{\kappa m}(\Omega)} \lesssim \|g\|_{L_T^\infty(\bC^{\beta'}_a)}\|f\|_{L_T^\infty(\mB_{\bbp,\bba}^{3\beta})}n^{-\frac{1+\beta}{2}+\epsilon}.
\end{align}

Following from Cauchy-Schwarz, we know for $\kappa'>1$ (the dual of $\kappa$)
\begin{align}\label{est:rho1}
    \|\rho(1)\|_{L^{\kappa'}(\Omega)}\leq  &\big(\mE\rho(2\kappa')\big)^{\frac{1}{2}}
   \nonumber \\&\Big[\mE\exp\Big((2{\kappa'}^2-\kappa')\int_0^1\big|\Gamma_{s-k_n(s)}b_n(s,M_{k_n(s)})\big|^2\dif s\Big)\Big]^\frac{1}{2}.
\end{align}
By the definition of $b_n$ from \eqref{def:bn}, for $\vartheta\in(0,(2\bba\cdot\frac{d}{\bbp})^{-1})$, for $|t-s|\leq n^{-1}$, by Young's convolution inequality
\begin{align*}
  |t-s|
  \sup_{r\in[s,t]}\|\Gamma_{r-k_n(r)}b_n(r)\|_\infty^2
  &  \lesssim   n^{-1}
  \sup_{r\in[s,t]}\|b_n(r)\|_\infty^2
  \\&  \lesssim  n^{-1}
  \sup_{r\in[s,t]}\|b(r)*n^{4d\vartheta}\phi(n^{3\vartheta}x, n^{\vartheta}v)\|_\infty^2
   \\&\lesssim  n^{-1}
   \sup_{r\in[s,t]}\|b(r)\|_\bbp^2
  \|n^{4d\vartheta}\phi(n^{3\vartheta}x, n^{\vartheta}v)\|_{\bbp'}^2
  \\&\lesssim \sup_{r\in[s,t]}\|b(r)\|_\bbp^2 n^{2\vartheta\bba\cdot\frac{d}{\bbp}-1}\lesssim \sup_{r\in[s,t]}\|b(r)\|_\bbp^2.
\end{align*}
 Then from \cref{lem:Kaha} with taking $f(s,x):=|\Gamma_{s-k_n(s)}b_n(s,x)|^2$ and \eqref{est:rho1} 
\begin{align*}
  \Big[\mE\exp\Big((2{\kappa'}^2-\kappa')\int_0^1\big|\Gamma_{s-k_n(s)}b_n(s,M_{k_n(s)})\big|^2\dif s\Big)\Big]^\frac{1}{2}\lesssim \exp\big(K\sup_{r\in[s,t]}\|b(r)\|_\bbp^2\big).
\end{align*}
It yields that there exists constant $K$ so that
\begin{align}
    \label{est:exp-kapp}
     \|\rho(1)\|_{L^{\kappa'}(\Omega)}\lesssim   \exp\big(K\sup_{r\in[s,t]}\|b(r)\|_\bbp^2\big)
\end{align}
since  $\rho(\gamma)$ is an exponential-martingale so has finite any-order moment. It also implies that $\{V^n_t\}_{0\le t\le T}$ is a $\mQ$-Brownian motion, where $\dif \mQ:=\rho\dif \mP.$

Hence, by H\"older's inequality, we have for any $\kappa>1$ so that $\kappa m<p$
\begin{align}\label{est:apply-Gir}
\Big\|\sup_{t\in(0,1)}&\big|\int_0^tg(r,Z_r^n(z))\(\Gamma_{r-k_n(r)}f(r,Z_{k_n(r)}^n(z))-f(r,Z_r^n(z))\)\dif r\big|\Big\|_{L^m(\Omega)}\nonumber\\
&\lesssim   \|\rho(1)\|_{L^{\kappa'}(\Omega)}   \|\mathcal{B}(M)\|_{L^{\kappa m}(\Omega)}.
\end{align}
  Therefore 
\eqref{AA05} holds due to \eqref{est:exp-kapp}  and
\eqref{est:nob}.
\end{proof}
\section{Convergence rate of Euler scheme to second order SDE with bounded drift}\label{sec:weak-strong-con}


\subsection{Weak convergence}
In this section, we assume that $b\in L^\infty_T(\mL^{\bbp})$ with $\bbp\in(2,\infty]^2$ and $\bba\cdot d/\bbp<1$ which is slightly weaker than the condition of \cref{thmW}. Based on \cite{RZ24}, there is a unique weak solution to SDE \eqref{eq:SDE}. We recall \eqref{def:bn} for the definition of $b_n$ and $\vartheta$, and \eqref{eq:SDE-EM} for the Euler scheme $Z^n$.

Denote by $\rho_t$ and $\rho^n_t$ the distributional density of $Z_t$ and $Z^n_t$ respect to the Lebesgue measure respectively. The following is the main result in this section.
\bt\label{thm-weak} Let $\vartheta=\frac12<\frac12(\bba\cdot \frac{d}{\bbp})^{-1}$. For
  any $\bbq\ge\bbp$, there is a constant $C=C(d,\bbp,\bbq,\|b\|_{\bbp})>0$ such that for all $t\in(0,T]$ and $n>4$,
     \begin{align}
    \label{est:thm-weak-S-1}
 \|\rho_t-\rho_t^n\|_{\bbq'}\leq Ct^{-\frac12(\bba\cdot \frac{d}{\bbq}+\bba\cdot \frac{d}{\bbp})}n^{-\frac{1}{2}} 
\end{align}
where $\bbq'$ is the dual of $\bbq$.
In particular, for $\bbq=\infty$,
\begin{align*}
    \int_0^T \|\mP\circ(Z_t)-\mP\circ(Z_t^n)\|_{var}^2\dif t\le C n^{-\frac12}.
\end{align*}
\et
\begin{remark}
   When $\bbp=\bbq=(\infty,\infty)$, we get
    \begin{align}
 \sup_{t\in[0,T]}\|\mP\circ(Z_t)^{-1}-\mP\circ(Z^n_t)^{-1}\|_{var}\leq C n^{-\frac{1}{2}}. 
\end{align}
\end{remark}

In order to show it, we introduce the following lemmas.

\begin{lemma}
  For any $\bbp,\bbq\in[1,\infty]^2$ with $1/\bbp+1/\bbq\le \1$, there is a constant $C=C(b,\bbp,\bbq)>0$ such that for any $0<s<t\le 1$ and $f_1\in \mL^{\bbp}$, $f_2\in \mL^{\bbq}$,
\begin{align}\label{0418:01}
    \mE |f_1(Z^n_{s})f_2(Z^n_t)|\le C s^{-\frac{1}{2}(\bba\cdot\frac{d}{\bbp}+\bba\cdot\frac{d}{\bbq})}\|f_1\|_{\bbp}\|f_2\|_{\bbq};
\end{align}
in particular, there is a constant $C=C(b,\bbp)>0$ such that for any $0<t\le 1$
\begin{align}
\|\rho_t^n\|_{\bbp}\leq C t^{-\frac{1}{2}(\bba\cdot\frac{d}{\bbp})}\label{est:density-Zn}.
\end{align}
\end{lemma}
\begin{proof}
 Let  $M_t:=(\int_0^tW_s\dif s,W_t)$, $t\in[0,1]$.   By the definition, one sees that
\begin{align}
M_t&=M_{s}+\left(\int_{s}^t W_r\dif r,W_t-W_{s}\right)=M_{s}+\left((t-s)W_{s}+\int_{s}^t(W_r-W_{s})\dif r,W_t-W_{s}\right).\label{0418:00}
\end{align}
    Based on the Girsanov transform (the same as we did for obtaining \eqref{est:apply-Gir}) and \eqref{0418:00}, we have by H\"older's inequality
\begin{align*}
    \mE |f_1(Z^n_s)f_2(Z^n_t)|&\lesssim \|f_1(M_{s})f_2(M_t)\|_{L^2(\Omega)}\\
    &=\big(\int_{\mR^{4d}}|f_1|^2(z)|f_2|^2(\Gamma_{t-s}z+z')g_{t-s}(z')g_{s}(z)\dif z\dif z'\big)^\frac{1}{2},
\end{align*}
again H\"older's inequality implies that 
\begin{align*}
    \mE |f_1(Z^n_{s})f_2(Z^n_t)|\lesssim \|f_1\|_{\bbp}\|f_2\|_{\bbq}\|g_{t-s}\|_{\mathbf{1}}\|g_{s}\|_{\bbr}\lesssim  s^{-\bba\cdot \frac{d}{2\bbr}}
\|f_1\|_{\bbp}\|f_2\|_{\bbq}.
\end{align*}
where $1/\bbr=1/\bbp+1/\bbq$, which is \eqref{0418:01}.  \eqref{est:density-Zn} is simply obtained by duality and \eqref{0418:01} with taking $f_1$ therein to be some constant function.
\end{proof}
\begin{lemma}
Let $\bbp\in[1,\infty]^2$
    and $f\in \bC^{1}_{\bba}$, $h\in \mL^{\bbp}$. Then there is a constant $C=C(d,T,\bbp)>0$ such that for all $t\in(T/n,T]$ and $n\in\mN$,
\begin{align}\label{0426:00}
        \left|\mE h(Z^n_{k_n(t)})\Big(\Gamma_{t-k_n(t)}f(Z^n_{k_n(t)})-f(Z^n_t)\Big)\right|\le Cn^{-\frac12}\|h\|_{\bbp}(k_n(t))^{-\bba\cdot \frac{d}{2\bbp}}\|f\|_{\bC^1_\bba}.
    \end{align}
    Moreover, if $\bbp\ge(2,2)$ and $f\in \bB^{2,1}_{\bbq;\bba}\cap \bC^3_\bba$ with some $\bbq\in[2,\infty]^2$, then there is a constant $C=C(d,T,\bbp,\bbq)>0$ such that
    \begin{align}\label{0416:00}
    \begin{split}
            &\left|\mE h(Z^n_{k_n(t)})\Big(\Gamma_{t-k_n(t)}f(Z^n_{k_n(t)})-f(Z^n_t)\Big)\right|\\
        &\le C\|h\|_{\bbp}\left(n^{-2}(k_n(t))^{-\bba\cdot \frac{d}{\bbp}}\|f\|_{\bC^{3}_\bba}+n^{-1}(k_n(t))^{-\bba\cdot \frac{d}{\bbp}}\|f\|_{\bC^{1}_\bba}+n^{-1}(k_n(t))^{-\frac12(\bba\cdot \frac{d}{\bbp}+\bba\cdot \frac{d}{\bbq})}\|f\|_{\bB^{2}_{\bbq;\bba}}\right).
    \end{split}
    \end{align}
\end{lemma}
\begin{proof}
First  we define
 \begin{align*}
     A_n(t,z):=\int_{k_n(t)}^t\int_{k_n(t)}^s\Gamma_{r-k_n(r)}b_n(r,z)\dif r\dif s,\ \ \quad \ \ B_n(t,z):=\int_{k_n(t)}^t\Gamma_{s-k_n(s)}b_n(s,z)\dif s.
 \end{align*}
Note that
\begin{align*}
&Z^n_t-Z^n_{k_n(t)}=\left(\int_{k_n(t)}^t V^n_r\dif r,B_n(t,Z^n_{k_n(t)})+W_t-W_{k_n(t)}\right)\no\\
&=\Big((t-k_n(t))V^n_{k_n(t)}+A_n(t,Z^n_{k_n(t)})+\int_{k_n(t)}^t(W_s-W_{k_n(t)})\dif s, B_n(t,Z^n_{k_n(t)})+W_t-W_{k_n(t)}\Big).
\end{align*}
Since $\Big(W_t-W_{k_n(t)},\int_{k_n(t)}^t(W_s-W_{k_n(t)})\dif s\Big)\stackrel{(d)}{=}\Big(W_{t-k_n(t)},\int_0^{t-k_n(t)}W_s\dif s\Big)=:M_{t-k_n(t)}$ is independent of $Z^n_{k_n(t)}$, denoting the density of $M_t$ and $Z^n_t$ by $g_t$ and $p^n_t$, we have
\begin{align*}
    &\sI_n(t):=\left|\mE h(Z^n_{k_n(t)})\Big(\Gamma_{t-k_n(t)}f(Z^n_{k_n(t)})-f(Z^n_t)\Big)\right|\\
    &=\Big|\int_{\mR^{4d}}\Big(f(x+(t-k_n(t)v,v)-f(x+(t-k_n(t)v)+A_n(t,z)+x',v+B_n(t,z)+v')\Big)\\
     &\qquad\qquad\qquad\qquad\qquad\qquad h(z)g_{t-k_n(t)}(z')p^n_{k_n(t)}(z)\dif z\dif z'\Big|,
\end{align*}
where $z=(x,v)$ and $z'=(x',v')$ in the integral.
In view of the fact $\vartheta<\frac12(\bba\cdot \frac{d}{\bbp})^{-1}$
\begin{align*}
    \|A_n(t)\|_\infty\le n^{-2}\|b_n\|_\infty\lesssim n^{-\frac32}\quad \text{and}\quad \|B_n(t)\|_\infty\le n^{-1}\|b_n\|_\infty\lesssim n^{-\frac12}.
\end{align*}
By H\"older's inequality, one sees that
\begin{align*}
    \sI_n(t)&\le \|f\|_{\bC^1_\bba}\int_{\mR^{4d}}(\|A_n(t)\|_\infty^{\frac13}+\|B_n(t)\|_\infty+|z'|_a)|h|(z)g_{t-k_n(t)}(z')p^n_{k_n(t)}(z)\dif z\dif z'\\
    &\lesssim \|f\|_{\bC^1_\bba}\int_{\mR^{2d}}|h|(z)p^n_{k_n(t)}(z)\dif z\Big(n^{-\frac12}+\int_{\mR^{2d}}|z'|_ag_{t-k_n(t)}(z')\dif z'\Big)\\
    &\lesssim \|f\|_{\bC^1_\bba}\|h\|_\bbp\|p^n_{k_n(t)}\|_{\bbp'}n^{-\frac12}\lesssim \|f\|_{\bC^1_\bba}\|h\|_\bbp(k_n(t))^{-\bba\cdot\frac{d}{2\bbp}}n^{-\frac12},
\end{align*}
which is \eqref{0426:00}.

Next, we show \eqref{0416:00}. We note that
\begin{align*}
    &\quad f(x+(t-k_n(t)v,v)-f(x+(t-k_n(t)v)+A_n(t,z)+x',v+B_n(t,z)+v')\\
    &=-\delta_{(A_n(t,z),B_n(t,z))}f(x+(t-k_n(t)v)+x',v+v')-\delta_{(x',v')}f(x+(t-k_n(t))v,v),
\end{align*}
where $\delta(x',v')f(x,v):=f(x+x',v+v')-f(x,v)$,
and
\begin{align*}
    |\delta_{(A_n(t,z),B_n(t,z))}f(x+(t-k_n(t)v)+x',v+v')|\le |A_n(t,z)|\|f\|_{\bC^3_\bba}+|B_n(t,z)|\|f\|_{\bC^{1}_\bba},
\end{align*}
which implies that
\begin{align*}
    \sI_n(t)\le&\int_{\mR^{4d}}|h(z)|(|A_n(t,z)|\|f\|_{\bC^3_\bba}+|B_n(t,z)|\|f\|_{\bC^{1}_\bba})g_{t-k_n(t)}(z')p^n_{k_n(t)}(z)\dif z\dif z' \\
    &+\Big|\int_{\mR^{4d}}\Big(\delta_{(x',v')}f(x+(t-k_n(t))v,v)\Big)h(z)g_{t-k_n(t)}(z')p^n_{k_n(t)}(z)\dif z\dif z'\Big|\\
    =:&\sI^1_n(t)+\sI^2_n(t).
\end{align*}
For $\sI^1_n(t)$, noting that
\begin{align*}
    \|A_n(t)\|_{\bbp}\le\frac12\|b\|_{\bbp}n^{-2}\quad \text{and}\quad \|B_n(t)\|_{\bbp}\le \|b\|_{\bbp}n^{-1},
\end{align*}
by H\"older's inequality and heat kernel estimate of $p^n_t$, \eqref{est:density-Zn}, we have
\begin{align*}
    \sI^1_n(t)&\lesssim \|h\|_{\bbp}(\|A_n(t)\|_{\bbp}\|f\|_{\bC^3_\bba}+\|B_n(t)\|_{\bbp}\|f\|_{\bC^1_\bba})\|g_{t-k_n(t)}\|_1\|p^n_{k_n(t)}\|_{(\bbp/2)'}\\
    &\lesssim (k_n(t))^{-\bba\cdot\frac{d}{\bbp}}\|h\|_{\bbp}(n^{-2}\|f\|_{\bC^3_\bba}+n^{-1}\|f\|_{\bC^1_\bba}).
\end{align*}
For $\sI^2_n(t)$, by the symmetry, H\"older's inequality and \eqref{CH1}, we have
\begin{align*}
    \sI^2_n(t)&=\Big|\int_{\mR^{4d}}\Big(\delta_{(x',v')}f(x+(t-k_n(t))v,v)-v'\cdot\nabla_vf(x+(t-k_n(t))v,v)\Big)\\
    &\qquad\qquad\qquad\qquad\qquad\qquad h(z)g_{t-k_n(t)}(z')p^n_{k_n(t)}(z)\dif z\dif z'\Big|\\
    &=\Big|\int_{\mR^{4d}}\Big(\delta_{(x',v')}f(x,v)-v'\cdot\nabla_vf(x,v)\Big)h(z)g_{t-k_n(t)}(z')p^n_{k_n(t)}(\Gamma_{k_n(t)-t}z)\dif z\dif z'\Big| \\
    &\lesssim \|h\|_{\bbp}\|f\|_{\bB^{2}_{\bbq;\bba}}\|p^n_{k_n(t)}\|_{\bbr}\int_{\mR^d}|z'|_\bba^2g_{t-k_n(t)}(z')\dif z'\\
    &\lesssim (k_n(t))^{-\frac{1}{2}(\bba\cdot\frac{d}{\bbp}+\bba\cdot\frac{d}{\bbq})}\|h\|_{\bbp}n^{-1}\|f\|_{\bB^{2}_{\bbq;\bba}}
\end{align*}
where $1/\bbr=1-1/\bbp-1/\bbq$.
    This completes the proof.
\end{proof}

Now, it is the position to give
\begin{proof}[Proof of  \cref{thm-weak}]
First, we fix any $t\in(0,1]$ and $\varphi\in\bC^\infty_0$. By using It\^o's formula to $r\to P_{t-r}\varphi(Z_r)$ and $P_{t-r}\varphi(Z^n_r)$, one sees that
\begin{align*}
    \mE\varphi(Z_t)=\mE P_t\varphi(Z_0)+\mE\int_0^t b(r,Z_r)\cdot \nabla_vP_{t-r}\varphi(Z_r)\dif r,
\end{align*}
and
\begin{align*}
\mE\varphi(Z_t^n)=\mE P_t\varphi(Z_0)+\mE\int_0^t \Gamma_{r-k_n(r)}b_n(r,Z_{k_n(r)}^n)\cdot \nabla_vP_{t-r}\varphi(Z_r^n)\dif r,
\end{align*}
which implies that
\begin{align}
    |\<\varphi,\rho_t-\rho^n_t\>|\le& \Big|\mE\int_0^t (b-b_n)(r,Z_r)\cdot \nabla_vP_{t-r}\varphi(Z_r)\dif r\Big|\no\\
    &+\Big|\mE\int_0^t b_n(r)\cdot \nabla_vP_{t-r}\varphi(Z_r)-b_n(r)\cdot \nabla_vP_{t-r}\varphi(Z_r^n)\dif r\Big|\no\\
    &+\Big|\mE\int_0^t \big(b_n(r,Z^n_r)-\Gamma_{r-k_n(r)}b_n(r,Z_{k_n(r)}^n)\big)\cdot \nabla_vP_{t-r}\varphi(Z_r^n)\dif r\Big|\no\\
    =:&I_1^n(t)+I_2^n(t)+I_3^n(t).\label{0419:00}
\end{align}
Then we divide the proof into three steps: in step 1, we estimate $I^n_3(t)$ for any $\varphi\in \mL^{\bbq}$; in step 2, we estimate $I^n_1(t)$ for any $\varphi\in \mL^{\bbq}$; finally in step 3, we show \eqref{est:thm-weak-S-1}.

{\textbf{Step 1:}} Without loss of generality, we let $t>3/n$ and $n>4$. Now we assume $\varphi\in \mL^{\bbq}$ with any $\bbq\ge \bbp\ge (2,2)$
and make the following decomposition:
\begin{align}\label{eq:In3-S1S2}
    I^n_3(t)\le& \Big|\mE\int_0^t \Gamma_{s-k_n(s)}b_n(s,Z^n_{k_n(s)})\cdot\left(\Gamma_{s-k_n(s)}\nabla_v P_{t-s}\varphi(Z^n_{k_n(s)})-\nabla_v P_{t-s}\varphi(Z^n_s)\right)\dif s\Big|\nonumber\\
    &+\Big|\mE\int_0^t \Gamma_{s-k_n(s)}(b_n(s)\cdot \nabla_v P_{t-s}\varphi)(Z^n_{k_n(s)})-(b_n(s)\cdot \nabla_v P_{t-s}\varphi)(Z^n_s)\dif s\Big|\nonumber\\
    =:&S_{1}+S_{2}.
\end{align}
In the sequel, we define
\begin{align*}
    \frac{1}{\bbr}:=\frac{1}{\bbp}+\frac{1}{\bbq}.
\end{align*}
For $S_{1}$, we make the following decomposition:
\begin{align*}
    &S_{1}\lesssim\Big|\mE\int_0^{2t/n} \Gamma_{s-k_n(s)}b_n(s,Z^n_{0})\cdot\left(\Gamma_{s-k_n(s)}\nabla_v P_{t-s}\varphi(Z^n_{0})-\nabla_v P_{t-s}\varphi(Z^n_s)\right)\dif s\Big|\no\\
    &+ \Big|\mE\left(\int_{2t/n}^{t/2}+\int_{t/2}^t\right) \Gamma_{s-k_n(s)}b_n(s,Z^n_{k_n(s)})\cdot\left(\Gamma_{s-k_n(s)}\nabla_v P_{t-s}\varphi(Z^n_{k_n(s)})-\nabla_v P_{t-s}\varphi(Z^n_s)\right)\dif s\Big|\\
    =:&S_{11}+S_{12}+S_{13}.
\end{align*}
Following from \cref{lem:est-Pt-itself} \eqref{est:Pt-itself-2} firstly and \eqref{est:Pt-itself-1} (see also \cite[Lemma 2.16]{HRZ}) that for any $\alpha\ge0$ ($\bC^0_\bba:=\mL^\infty$)
\begin{align}\label{0426:05}
    \|\nabla_v P_{t-s}\varphi\|_{\bB^{2}_{\bbq;\bba}} \lesssim \|\varphi\|_\bbq (t-s)^{-\frac{3}{2}},\quad \|\nabla_v P_{t-s}\varphi\|_{\bC^\alpha_\bba}\lesssim  \|\varphi\|_\bbq (t-s)^{-\frac{\alpha+1}2-\bba\cdot\frac{d}{2\bbq}},
\end{align}
moreover, by considering the fact that $s<\frac{2t}{n}$ yielding $t-s\leq s$ for $n>4$ and
\begin{align}\label{0419:01}
\begin{split}
    n^{\vartheta(\bba\cdot \frac{d}{\bbp})}\int_0^{2t/n}(t-s)^{-\frac12-\bba\cdot \frac{d}{2\bbq}}\dif s&\lesssim (t-2t/n)^{-\bba\cdot\frac{d}{2\bbq}-\vartheta(\bba\cdot \frac{d}{\bbp})}n^{\vartheta(\bba\cdot \frac{d}{\bbp})}\int_0^{2t/n}s^{-\frac12+\vartheta(\bba\cdot \frac{d}{\bbp})}\dif s\\
    &\lesssim n^{-\frac12}t^{-\bba\cdot\frac{d}{2\bbq}-\vartheta(\bba\cdot \frac{d}{\bbp})},
    \end{split}
\end{align}
we have
\begin{align}\label{0426:02}
    S_{11}&\lesssim \|b_n\|_\infty\int_0^{2t/n}\|\nabla_v P_{t-s}\varphi\|_\infty\dif s\lesssim n^{\vartheta(\bba\cdot \frac{d}{\bbp})}\|\varphi\|_\bbq\int_0^{2t/n}(t-s)^{-\frac12-\bba\cdot\frac{d}{2\bbq}}\dif s\no\\
    &\lesssim n^{-\frac12}t^{-\bba\cdot\frac{d}{2\bbq}-\vartheta(\bba\cdot \frac{d}{\bbp})}\|\varphi\|_\bbq.
\end{align}
For $S_{12}$,  by \eqref{0426:00} with taking $f=\nabla_v P_{t-s}\varphi$ correspondingly, we have
\begin{align}
    S_{12}\lesssim&n^{-\frac12}\|b_n\|_\bbp\int_{2t/n}^{t/2}(k_n(s))^{-\bba\cdot\frac{d}{2\bbp}}\|\nabla_vP_{t-s}\varphi\|_{\bC^1_\bba}\dif s\no\\
   \lesssim& n^{-\frac12}\|b\|_\bbp\|\varphi\|_\bbq\int_{0}^{t/2}s^{-\bba\cdot\frac{d}{2\bbp}}(t-s)^{-1-\frac{d}{2\bbq}}\dif s\lesssim n^{-\frac12}\|b\|_\bbp t^{-1-\frac{d}{2\bbq}}\|\varphi\|_\bbq\int_{0}^{t/2}s^{-\bba\cdot\frac{d}{2\bbp}}\dif s\no\\
   \lesssim & n^{-\frac12}\|b\|_\bbp\|\varphi\|_\bbq t^{-\frac12\bba\cdot(\frac{d}{\bbp}+\frac{d}{\bbq})}.\label{0426:06}
\end{align}
For $S_{13}$, we define
\begin{align*}
    Q_n(s):=\mE \Gamma_{s-k_n(s)}b_n(s,Z^n_{k_n(s)})\cdot\left(\Gamma_{s-k_n(s)}\nabla_v P_{t-s}\varphi(Z^n_{k_n(s)})-\nabla_v P_{t-s}\varphi(Z^n_s)\right).
\end{align*}
On one hand, since $t/2>2t/n$ for $n>4$, by \eqref{0418:01} and from \cref{lem:est-Pt-itself} \eqref{est:Pt-itself-2} we have for $s>t/2$
\begin{align}\label{0426:03}
    |Q_n(s)|\lesssim (k_n(s))^{-\bba\cdot\frac{d}{2\bbr}}\|b_n\|_\bbp\|\nabla_vP_{t-s}\varphi\|_\bbq\lesssim s^{-\bba\cdot\frac{d}{2\bbr}}\|b\|_\bbp(t-s)^{-\frac12}\|\varphi\|_\bbq.
\end{align}
On the other hand, by \eqref{0416:00} and \eqref{0426:05} together with the fact $\bbq>\bbp$ we have for $s>t/2$
\begin{align*}
|Q_n(s)|\lesssim&\|b_n\|_{\bbp}\Big(n^{-2}(k_n(s))^{-\bba\cdot\frac{d}{\bbp}}\|\nabla_v P_{t-s}\varphi\|_{\bC^3_\bba}+n^{-1}(k_n(s))^{-\bba\cdot\frac{d}{\bbp}}\|\nabla_v P_{t-s}\varphi\|_{\bC^1_\bba}\\
&\qquad+n^{-1}(k_n(s))^{-\bba\cdot\frac{d}{2\bbr}}\|\nabla_v P_{t-s}\varphi\|_{\bB^2_{\bbq;\bba}}\Big)\\
\lesssim &\|b\|_{\bbp}\|\varphi\|_\bbq \Big(n^{-2}s^{-\bba\cdot\frac{d}{\bbp}}(t-s)^{-2-\frac{d}{2\bbq}}+n^{-1}s^{-\bba\cdot\frac{d}{\bbp}}(t-s)^{-1-\frac{d}{2\bbq}}+n^{-1}s^{-\bba\cdot\frac{d}{2\bbr}}(t-s)^{-\frac32}\Big)\\
\lesssim &\|\varphi\|_\bbq s^{-\bba\cdot\frac{d}{2\bbr}}\Big(n^{-2}(t-s)^{-2-\frac{d}{2\bbp}}+n^{-1}(t-s)^{-1-\frac{d}{2\bbp}}+n^{-1}(t-s)^{-\frac32}\Big)\\
\lesssim &\|\varphi\|_\bbq s^{-\bba\cdot\frac{d}{2\bbr}}\Big(n^{-2}(t-s)^{-2-\frac12}+n^{-1}(t-s)^{-1-\frac12}+n^{-1}(t-s)^{-\frac32}\Big),
\end{align*}
where we used the fact $s^{-\bba\cdot\frac{d}{\bbp}}\lesssim s^{-\bba\cdot\frac{d}{2\bbr}}(t-s)^{\bba\cdot(\frac{d}{2\bbq}-\frac{d}{2\bbp})}$ for $s>t/2$ and $\bbq\ge\bbp$, in the third inequality. Thus, combining it and \eqref{0426:03}, we have
\begin{align*}
   |Q_n(s)|&\lesssim \|\varphi\|_\bbq s^{-\bba\cdot\frac{d}{2\bbr}}(t-s)^{-\frac12}\left(\left[(n^{-2}(t-s)^{-2})\wedge 1\right]+\left[(n^{-1}(t-s)^{-1})\wedge1\right]\right)\\
   &\lesssim \|\varphi\|_\bbq s^{-\bba\cdot\frac{d}{2\bbr}}(t-s)^{-\frac12}\left[(n^{-1}(t-s)^{-1})\wedge 1\right],
\end{align*}
which by  \cref{lem:A1} implies that
\begin{align*}
    S_{13}&\le\int_{t/2}^t|Q_n(s)|\dif s\lesssim \|\varphi\|_\bbq \int_{t/2}^t s^{-\bba\cdot\frac{d}{2\bbr}}(t-s)^{-\frac12}\left[(n^{-1}(t-s)^{-1})\wedge 1\right]\dif s\\
    &\lesssim n^{-\frac12} t^{-\bba\cdot\frac{d}{2\bbr}}\|\varphi\|_\bbq.
\end{align*}
Therefore, combining with \eqref{0426:02} and \eqref{0426:06} implies
\begin{align}\label{est:In3-S1}
    S_1\lesssim n^{-\frac12}t^{-\bba\cdot\frac{d}{2\bbq}-(\vartheta\vee\frac12)(\bba\cdot \frac{d}{\bbp})}\|\varphi\|_\bbq.
\end{align}

Then we estimate $S_{2}$. Let $t\in(0,1]$. For any $f\in \mL^{\bbp}$, $r\to P_rf$ solves the following kinetic PDE
\begin{align}\label{eq:ke:S2}
    \p_r P_rf=(\Delta_v+v\cdot \nabla_x)P_rf,\quad P_0f=f,
\end{align}
which by It\^o's formula to $r\to P_{t-r}f(Z^n_r)$ implies that
\begin{align*}
    \mE f(Z^n_t)=\mE P_{t}f(Z_0)+\mE\int_0^t \Gamma_{r-k_n(r)}b_n(r,Z^n_{k_n(r)})\cdot \nabla_v P_{t-r}f(Z^n_r)\dif r.
\end{align*}
Then for any $0<s<t$, by replacing $f$ in \eqref{eq:ke:S2} with $\Gamma_{t-s}f$, we also have
\begin{align*}
    \mE \Gamma_{t-s}f(Z^n_{s})=\mE P_{s}\Gamma_{t-s}f(Z_0)+\mE\int_0^{s} \Gamma_{r-k_n(r)}b_n(r,Z^n_{k_n(r)})\cdot \nabla_v P_{s-r}\Gamma_{t-s}f(Z^n_r)\dif r.
\end{align*}
Hence, we have
\begin{align}\label{eq:I-123}
   |\mE & f(Z^n_t)-\mE \Gamma_{t-s}f(Z^n_{s})|\nonumber\\\le& \|P_{t}f-P_{s}\Gamma_{t-s}f\|_\infty+\mE\int_s^t |\Gamma_{r-k_n(r)}b_n(r,Z^n_{k_n(r)})\cdot \nabla_v P_{t-r}f(Z^n_r)|\dif r\nonumber\\
    &+\mE\int_0^{s}|\Gamma_{r-k_n(r)}b_n(r,Z^n_{k_n(r)})\nabla_v(P_{t-r}-P_{s-r}\Gamma_{t-s})f(Z^n_r)|\dif r\nonumber\\
    =:&I_1+I_2+I_3.
\end{align}
Now we assume $s>2/n$ below. Based on \eqref{0418:01} and from \cref{lem:est-Pt-itself} \eqref{est:Pt-itself-2}, it follows that
\begin{align}\label{est:I-I2}
    I_2&\lesssim \|b_n\|_{\bbp}\int_s^t (k_n(r))^{-\bba\cdot \frac{d}{\bbp}} \|\nabla_v P_{t-r}f\|_{\bbp}\dif r\nonumber\\
    &\lesssim (k_n(s))^{-\bba\cdot d/\bbp}\|f\|_{\bbp}\int_s^t (t-r)^{-\frac{1}{2}}\dif r\lesssim (k_n(s))^{-\bba\cdot \frac{d}{\bbp}}|t-s|^{\frac12}\|f\|_{\bbp}.
\end{align}
Similarly, for $I_3$, we have
\begin{align}\label{eq:I-I2}
  I_3\le& \|b_n\|_\infty\int_0^{1/n}\mE|\nabla_v(P_{t-r}-P_{s-r}\Gamma_{t-s})f|(Z^n_r)\dif r\nonumber\\
  &+\|b_n\|_{\bbp} \int_{1/n}^s (k_n(r))^{-\bba\cdot \frac{d}{\bbp}}\|\nabla_v(P_{t-r}-P_{s-r}\Gamma_{t-s})f\|_{\bbp}\dif r
 \nonumber\\
   \lesssim &n^{1/2}\int_0^{1/n}r^{-\bba\cdot\frac{d}{2\bbp}}\|\nabla_v(P_{t-r}-P_{s-r}\Gamma_{t-s})f\|_{\bbp}\dif r\nonumber\\ &+\|b_n\|_{\bbp} \int_{1/n}^s (k_n(r))^{-\bba\cdot \frac{d}{\bbp}}\|\nabla_v(P_{t-r}-P_{s-r}\Gamma_{t-s})f\|_{\bbp}\dif r,
\end{align}
which by the fact that $\|b\|_\infty\lesssim n^{-\bba\cdot\frac{d}{2\bbp}}$ and \eqref{est:semi-P-L} with correspondingly taking $k=1$, $\delta=\frac{1}{2}$,
\begin{align}\label{est:I-I3}
    I_3\lesssim&n^{1/2}\|f\|_{\bbp}|t-s|^{\frac12}\int_0^{1/n}(s-r)^{-1}r^{-\bba\cdot\frac{d}{2\bbp}}\dif r\nonumber\\
    &+\|f\|_{\bbp}\int_{0}^s r^{-\bba\cdot\frac{d}{\bbp}}(s-r)^{-\frac{1}{2}}[((t-s)(s-r)^{-1})\wedge 1]\dif r\nonumber\\
    \lesssim&n^{1/2}\|f\|_{\bbp}|t-s|^{\frac12} s^{-\frac12-\bba\cdot\frac{d}{2\bbp}}\int_0^{1/n}(\frac{1}{n}-r)^{-\frac12+\bba\cdot\frac{d}{2\bbp}}r^{-\bba\cdot\frac{d}{2\bbp}}\dif r\nonumber\\
    &+\|f\|_{\bbp}s^{-\bba\cdot\frac{d}{\bbp}}\int_{0}^1 r^{-\bba\cdot \frac{d}{\bbp}}(1-r)^{-\frac{1}{2}}[((t-s)(1-r)^{-1})\wedge 1]\dif r\nonumber\\
    \lesssim &\|f\|_{\bbp}|t-s|^{\frac12} s^{-\frac12-\bba\cdot\frac{d}{2\bbp}},
\end{align}
where we used a change of variable in the second and third inequality and the follow fact in the third one as well (see \cite[Lemma A.2]{HRZ}): for any $\alpha,\beta>0$ with $\alpha+\beta<1$,
\begin{align*}
    \int_0^1 r^{-\alpha}(1-r)^{-\beta}[(\lambda(1-r)^{-1})\wedge 1]\dif r\lesssim (\lambda\wedge 1)^{1-\beta}.
\end{align*}
Therefore, by the estimate \eqref{est:I-I2} of $I_2$, \eqref{est:I-I3} of $I_3$ and \eqref{est:semi-P-L} for $I_1$ with accordingly taking $k=0$ and $\delta=\frac{1}{2}$, from \eqref{eq:I-123} we have for any $2/n<s<t\le 1$
\begin{align}\label{est:I-123}
 |\mE f(Z^n_t)-\mE \Gamma_{t-s}f(Z^n_s)|\lesssim    \|f\|_{\bbp}|t-s|^{\frac12} s^{-\frac12-\bba\cdot\frac{d}{2\bbp}}.
\end{align}
Now we take $f=b_n(s)\cdot \nabla_v P_{t-s}\varphi$ in \eqref{est:I-123} then have
\begin{align}\label{est:In3-S2}
    S_{2}\lesssim& \Big|\mE\int_0^{2t/n} \Gamma_{s-k_n(s)}(b_n(s)\cdot \nabla_v P_{t-s}\varphi)(Z^n_{k_n(s)})-(b_n(s)\cdot \nabla_v P_{t-s}\varphi)(Z^n_s)\dif s\Big|\nonumber\\
    &+\int_{2t/n}^t |s-k_n(s)|^{\frac12} (k_n(s))^{-\frac12-\bba\cdot\frac{d}{2\bbp}}\|b_n(s)\cdot \nabla_v P_{t-s}\varphi\|_{\bbp}\dif s\nonumber\\
    \lesssim& \|b_n\|_\infty \int_0^{2t/n}\|\nabla_v P_{t-s}\varphi\|_\infty \dif s+\|b_n\|_{\bbp}n^{-\frac12}\int_{0}^1 s^{-\frac12-\bba\cdot\frac{d}{2\bbp}}\|\nabla_v P_{t-s}\varphi\|_{\infty}\dif s\nonumber\\
    \lesssim& \|\varphi\|_\bbq\left(n^{\vartheta(\bba\cdot \frac{d}{\bbp})}\int_0^{2t/n}(t-s)^{-\frac12-\bba\cdot\frac{d}{2\bbq}}\dif s+n^{-\frac12}\int_{0}^1 s^{-\frac12-\bba\cdot\frac{d}{2\bbp}}(t-s)^{-\frac12-\bba\cdot\frac{d}{2\bbq}}\dif s\right)\nonumber\\
    \lesssim& n^{-\frac12}t^{-\bba\cdot\frac{d}{2\bbq}-(\vartheta\vee \frac12)(\bba\cdot \frac{d}{\bbp})}\|\varphi\|_\bbq,
\end{align}
where we used \eqref{est:semi-P-L} in the second to last  inequality above and \eqref{0419:01} in the last inequality. Thus, by \eqref{eq:In3-S1S2}, \eqref{est:In3-S1} and \eqref{est:In3-S2} we have
\begin{align}\label{0419:02}
    I^n_3(t)\lesssim n^{-\frac12}t^{-\bba\cdot\frac{d}{2\bbq}-(\vartheta\vee \frac12)(\bba\cdot \frac{d}{\bbp})}\|\varphi\|_\bbq.
\end{align}

    {\textbf{(Step 2):}} For $I^n_1(t)$, we use  paraproduct. For any $s\in(0,T]$, we let $\varphi_b:=(b_n(s)-b(s))\cdot\nabla_v P_{t-s}\varphi$. Then by It\^o's formula to $r\to P_{s-r}\varphi_b(Z^n_r)$, we have
\begin{align*}
    \mE\varphi_b(Z^n_s)=\mE P_s\varphi_b(Z_0)+\mE\int_0^s \Gamma_{r-k_n(r)}b(r,Z^n_{k_n(r)})\cdot \nabla_v P_{s-r}\varphi_b(Z^n_r)\dif r,
\end{align*}
which by \eqref{0418:01} implies that
\begin{align*}
    |\mE\varphi_b(Z^n_s)|\lesssim \|P_s\varphi_b\|_\infty+ \int_0^s r^{-\bba\cdot \frac{d}{\bbp}}\|\nabla_v P_{s-r}\varphi_b\|_{\bbp}\dif r.
\end{align*}
That is to say,
\begin{align}
\label{est:Ephib} I_1^n(t)\leq \int_0^t\|P_s\varphi_b\|_\infty\dif s+ \int_0^t\int_0^s r^{-\bba\cdot \frac{d}{\bbp}}\|\nabla_v P_{s-r}\varphi_b\|_{\bbp}\dif r\dif s.
\end{align}
Recalling the Bony decomposition (\cite{HZZZ22}, see also \eqref{def:Bony})
\begin{align*}
    fg=f\prec g+f\circ g+f\succ g=:f\prec g+f\succcurlyeq g,
\end{align*}
based on \cref{lem:A2}, we have for any $\alpha>0$, from \eqref{0426:09}
\begin{align*}
    \|f\prec g\|_{\bB^{-\alpha}_{\bbr;\bba}}\lesssim \|g\|_{\bB^{-\alpha}_{\bbp;\bba}}\|f\|_{\bbq}
\end{align*}
and from \eqref{0426:08} together with \eqref{0426:07}
\begin{align*}
    \|f\succcurlyeq g\|_{\bB^{0}_{\bbr;\bba}}\lesssim \|f\|_{\bB^{\alpha,1}_{\bbp;\bba}}\|g\|_{\bB^{-\alpha}_{\bbq;\bba}}.
\end{align*}
Thus, due to $\varphi_b:=(b_n(s)-b(s))\cdot\nabla_v P_{t-s}\varphi$, by \cref{lem:f-fn} \eqref{est:f-fn-app} 
\begin{align*}
    &\|P_s\varphi_b\|_\infty\le\|P_s(\nabla_v P_{t-s}\varphi\prec(b_n-b))(s)\|_\infty+\|P_s(\nabla_v P_{t-s}\varphi\succcurlyeq(b_n-b))(s)\|_\infty\nonumber\\
    \lesssim& s^{-\frac12-\bba\cdot\frac{d}{2\bbp}}\|(\nabla_v P_{t-s}\varphi\prec(b_n-b))(s)\|_{\bB^{-1}_{\bbp;\bba}}+s^{-\bba\cdot\frac{d}{2\bbr}}\|(\nabla_v P_{t-s}\varphi\succcurlyeq(b_n-b))(s)\|_{\bB^0_{\bbr;\bba}}\nonumber\\
    \lesssim& s^{-\frac12-\bba\cdot\frac{d}{2\bbp}}\|b_n-b\|_{\bB^{-1}_{\bbp;\bba}}\|\nabla_v P_{t-s}\varphi\|_\infty\no\\
    &+s^{-\bba\cdot\frac{d}{2\bbr}}\left([\|\nabla_v P_{t-s}\varphi\|_{\bB^{2,1}_{\bbq;\bba}}\|b_n-b\|_{\bB^{-2}_{\bbp;\bba}}]\wedge[\|\nabla_v P_{t-s}\varphi\|_{\bbq}\|b_n-b\|_{\bbp}] \right)\nonumber\\
    \lesssim& \|\varphi\|_{\bbq}\Big(n^{-\vartheta}s^{-\frac12-\bba\cdot\frac{d}{2\bbp}}(t-s)^{-\frac12-\bba\cdot\frac{d}{2q}}+s^{-\bba\cdot\frac{d}{2\bbr}}[((t-s)^{-\frac{3}{2}}n^{-2\vartheta})\wedge (t-s)^{-\frac12}]\Big),\no
\end{align*}
then \cref{lem:A1} implies that
\begin{align}\label{est:Ephib1}
    \int_0^t \|P_s\varphi_b\|_\infty\dif s\lesssim& \|\varphi\|_{\bbq}\Big(n^{-\vartheta}\int_0^ts^{-\frac12-\bba\cdot\frac{d}{2\bbp}}(t-s)^{-\frac12-\bba\cdot\frac{d}{2q}}\dif s\no\\
    &\qquad\qquad+\int_0^t s^{-\bba\cdot\frac{d}{2\bbr}}(t-s)^{-\frac12}[((t-s)^{-1}n^{-2\vartheta})\wedge1]\dif s \Big)\no\\
    \lesssim &\|\varphi\|_{\bbq}n^{-\vartheta}t^{-\bba\cdot\frac{d}{2\bbr}}.
\end{align}
Moreover, by from \cref{lem:est-Pt-itself} \eqref{est:Pt-itself-2} and \cref{lem:f-fn} \eqref{est:f-fn-app} 
\begin{align*}
    &\|\nabla_v P_{s-r}\varphi_b\|_{\bbp}\\
    \lesssim& \left[(s-r)^{-\frac{3}{2}}\|(\nabla_v P_{t-s}\varphi\prec(b_n-b))(s)\|_{\bB^{-2}_{\bbp;\bba}}\right]\wedge \left[(s-r)^{\frac12}\|(\nabla_v P_{t-s}\varphi\prec(b_n-b))(s)\|_{\bbp}\right]\nonumber\\
    &+(s-r)^{-\frac12-\bba\cdot\frac{d}{2\bbq}}\|(\nabla_v P_{t-s}\varphi\succcurlyeq(b_n-b))(s)\|_{\bB^0_{\bbr;\bba}}\nonumber\\
    \lesssim& (s-r)^{-\frac12}\left(\left[(s-r)^{-1}\|\nabla_v P_{t-s}\varphi\|_{\infty}\|b_n-b\|_{\bB^{-2}_{\bbp;\bba}}\right]\wedge \left[\|\nabla_v P_{t-s}\varphi\|_{\infty}\|b_n-b\|_\bbp\right]\right)\\
    &+(s-r)^{-\frac12-\bba\cdot\frac{d}{2\bbq}}\left(\left[\|\nabla_v P_{t-s}\varphi\|_{\bB^{2,1}_{\bbq;\bba}}\|b_n-b\|_{\bB^{-2}_{\bbp;\bba}}\right]\wedge \left[\|\nabla_v P_{t-s}\varphi\|_{\bbq}\|b_n-b\|_{\bbp}\right]\right)\nonumber\\
    \lesssim&(s-r)^{-\frac12}(t-s)^{-\frac12-\bba\cdot\frac{d}{2\bbq}}\|\varphi\|_\bbq\left[((s-r)^{-1}n^{-2\vartheta})\wedge1\right]\\
    &+(s-r)^{-\frac12-\bba\cdot\frac{d}{2\bbq}}(t-s)^{-\frac12}\|\varphi\|_\bbq\left[((t-s)^{-1}n^{-2\vartheta})\wedge1\right],
\end{align*}
which by \cref{lem:A1} implies that
\begin{align}\label{est:Ephib2}
    &\int_0^t\int_0^sr^{-\bba\cdot \frac{d}{\bbp}}\|\nabla_v P_{s-r}\varphi_b\|_{\bbp}\dif r\dif s\no\\
    \lesssim& \|\varphi\|_\bbq\Big(\int_0^t\int_0^sr^{-\bba\cdot \frac{d}{\bbp}}(s-r)^{-\frac12}(t-s)^{-\frac12-\bba\cdot\frac{d}{2\bbq}}\left[((s-r)^{-1}n^{-2\vartheta})\wedge1\right] \dif r\dif s\no\\
    &+\int_0^t\int_0^sr^{-\bba\cdot \frac{d}{\bbp}}(s-r)^{-\frac12-\bba\cdot\frac{d}{2\bbq}}(t-s)^{-\frac12}\left[((t-s)^{-1}n^{-2\vartheta})\wedge1\right]\dif r\dif s\Big)\no\\
    \lesssim& \|\varphi\|_\bbq\Big(n^{-\vartheta}\int_0^ts^{-\bba\cdot\frac{d}{\bbp}}(t-s)^{-\frac12-\bba\cdot\frac{d}{2\bbq}}\dif s\no\\
    &+\int_0^ts^{\frac12-\bba\cdot \frac{d}{\bbp}-\bba\cdot \frac{d}{2\bbq}}(t-s)^{-\frac12}\left[((t-s)^{-1}n^{-2\vartheta})\wedge1\right]\dif s\Big)\no\\\
    \lesssim& \|\varphi\|_\bbq n^{-\vartheta} t^{\frac12-\bba\cdot \frac{d}{\bbp}-\bba\cdot \frac{d}{2\bbq}}.
\end{align}
Combining all the calculations above of \eqref{est:Ephib} with \eqref{est:Ephib1} and \eqref{est:Ephib2}, we have
    \begin{align}\label{0419:04}
    |I_1^n(t)|\lesssim n^{-\vartheta}t^{-\frac12(\bba\cdot\frac{d}{\bbp}+\bba\cdot\frac{d}{\bbq})}\|\varphi\|_\bbq.
\end{align}

{\textbf{(Step 3):}} By definition, it is easy to see that by  \cref{lem:est-Pt-itself} \eqref{est:Pt-itself-2}
\begin{align}
    I^n_2(t)&=\Big|\int_0^t \<b_n(r)\cdot\nabla_vP_{t-r}\varphi,\rho_r-\rho^n_r\>\dif r\Big|\le \|b_n\|_{\bbp}\int_0^t \|\nabla_vP_{t-r}\varphi\|_\infty\|\rho_r-\rho^n_r\|_{\bbp'}\dif r\no\\
    &\lesssim \|\varphi\|_{\bbq}\int_0^t (t-r)^{-\frac{1}{2}-\bba\cdot\frac{d}{2\bbq}}\|\rho_r-\rho^n_r\|_{\bbp'}\dif r.\label{0419:05}
\end{align}
Now, taking $\bbq=\bbp$ in \eqref{0419:02} and \eqref{0419:04}, based on \eqref{0419:00} together with \eqref{0419:05}, by taking supremum of $\varphi$ with $\|\varphi\|_\bbp=1$, we have
\begin{align*}
    \|\rho_t-\rho_t^n\|_{\bbp'}\lesssim n^{-\vartheta}t^{-\bba\cdot\frac{d}{\bbp}}+\int_0^t (t-r)^{-\frac{1}{2}-\bba\cdot\frac{d}{2\bbp}}\|\rho_r-\rho^n_r\|_{\bbp'}\dif r+ n^{-\frac12}t^{-[(\vartheta\vee \frac12)+\frac12](\bba\cdot \frac{d}{\bbp})}.
\end{align*}
Since $\vartheta(\bba\cdot d/\bbp)\le1/2$, we have
\begin{align*}
    \int_0^t (t-r)^{-\frac{1}{2}-\bba\cdot\frac{d}{2\bbp}}r^{-[(\vartheta\vee \frac12)+\frac12](\bba\cdot \frac{d}{\bbp})}\dif r\lesssim \int_0^t (t-r)^{-\frac{1}{2}-\bba\cdot\frac{d}{2\bbp}}r^{-\frac12-\bba\cdot \frac{d}{2\bbp}}\dif r\lesssim t^{-\bba\cdot \frac{d}{\bbp}},
\end{align*}
which by Gr\"onwall's inequality of Volterra's type (see \cite[Lemma 2.2]{Zhang10} and \cite[Lemma A.4]{Ha23}) implies that
\begin{align}\label{0419:06}
    \|\rho_t-\rho_t^n\|_{\bbp'}\lesssim n^{-(\vartheta\wedge \frac{1}{2})}t^{-\bba\cdot\frac{d}{\bbp}}+n^{-\frac12}t^{-\bba\cdot\frac{d}{2\bbp}-(\vartheta\vee \frac12)(\bba\cdot \frac{d}{\bbp})}.
\end{align}
Moreover, in view of \eqref{0419:00}, \eqref{0419:02}, \eqref{0419:04} and \eqref{0419:05}, by \eqref{0419:06}, for any $\bbq\ge\bbp$ and $\vartheta=1/2<1/(2\bba\cdot d/\bbp)$, we have
\begin{align*}
    \|\rho_t-\rho^n_t\|_{\bbq'}\lesssim& n^{-\frac12}t^{-\frac12(\bba\cdot\frac{d}{\bbp}+\bba\cdot\frac{d}{\bbq})}+\int_0^t (t-r)^{-\frac{1}{2}-\bba\cdot\frac{d}{2\bbq}}\|\rho_r-\rho^n_r\|_{\bbp'}\dif r\\
    \lesssim& n^{-\frac12}t^{-\frac12(\bba\cdot\frac{d}{\bbp}+\bba\cdot\frac{d}{\bbq})}+n^{-\frac12}\int_0^t (t-r)^{-\frac{1}{2}-\bba\cdot\frac{d}{2\bbq}}r^{-\bba\cdot \frac{d}{\bbp}}\dif r\\
    \lesssim& n^{-\frac12}t^{-\frac12(\bba\cdot\frac{d}{\bbp}+\bba\cdot\frac{d}{\bbq})}
\end{align*}
and complete the proof.
\end{proof}


\subsection{Strong convergence}
Here is the main result on the strong convergence of the solution $Z$ to kinetic equation \eqref{eq:SDE} and its EM-scheme  $Z^n$ from \eqref{eq:SDE-EM}.
    \bt\label{thm-strong} Assume \cref{ass:main1}
     holds. For $\epsilon\in(0,\frac{1}{2})$ being small and $\vartheta<\frac{1}{2}(\bba\cdot\frac{d}{\bbp})^{-1}$, there exists positive constant  $C=C(\bbp,d,\beta,\epsilon,m, \vartheta, \|b\|_{_{L^\infty_T(\bB^{2/3,\beta}_{\bbp;x,\bba})}})$ so that
 for any $m<p:=\min(p_x,p_v)$
     \begin{align}
    \label{est:thm-strong-S}
   \big\|\sup_{t\in[0,1]}|Z_t-Z_t^n|\big\|_{L^m(\Omega)}\leq C (n^{-\frac{1+\beta/3}{2}+\epsilon}+n^{-\vartheta(\beta+1-\bba\cdot\frac{d}{\bbp})+\epsilon}+\|Z_0-Z_0^n\|_{L^m_\omega}).
\end{align}

\et
Before giving the proof of the theorem above, we recall and introduce several auxiliary functions and estimates that used inside the proof. From now on we always assume \cref{ass:main1}
     holds.

 Let $U^i, i=1,\ldots, d$ be the solution to \eqref{LKE}  with taking $b=b_n, f=b_n^i$ and  sufficiently large $\lambda$, i.e.
 \begin{align*}
\p_t u=\Delta_v u+v\cdot\nabla_x u-\lambda u+b_n\cdot\nabla_v u+b_n^i,\quad u_0=0.
\end{align*}
Recall from \cref{thm41} (ii) \eqref{FA001} (precisely \eqref{EST:U-1},\eqref{EST:U-2} and \eqref{EST:U-3}) that for $U=(u^i)_{1\leq i\leq d},$  there exists $\lambda_0>0$ such that for every $\lambda\ge \lambda_0$, $\nabla_v U$ is bounded and H\"older continuous on $[0,1]\times\Rd$, for any $\theta\in[0,2]$
\begin{align}\label{tmp.2709regu}
	\sup_n\|U\|_{L_T^{\infty}(\bB^{2/3,\beta+\theta}_{\bbp;x,\bba})}<\infty
	\tand
	\sup_n\sup_{(t,x)\in[0,1]\times\Rd}|\nabla_v U(t,x)|=o_\lambda(1),
\end{align}
where $o_\lambda(1)$ denotes any constant such that $\lim_{\lambda\to\infty}o_\lambda(1)=0$; besides we have
\begin{align}
    \label{est:naxnavU}
   \nabla\nabla_vU \text{ is continuous on }[0,1]\times\mR^d,\quad \sup_n\|\nabla\nabla_vU\|_{L_T^{\infty}(\mL^\bbp)}<\infty.
\end{align}

Let $\mathcal{M}$ be   the Hardy-Littlewood maximal function from \eqref{def:M}. Define
\begin{align}\label{def.Ant}
	A_t^n:=& \int_0^t \left[ \cmm|\nabla\nabla_v U|(s,Z_s)+\cmm|\nabla\nabla_v U|(s,Z^n_s)\right]^2 \dif s,\\
 H^n(m):=&\Big\|\sup_{t\in[0,1]}\big|  \int_0^tb(r,Z_r)\dif r-  \int_0^tb^n(r,Z_r)\dif r\big|\Big\|_{L^{m}(\Omega)}
 \nonumber\\&+ \Big\|\sup_{t\in[0,1]}\big|  \int_0^t\big(\nabla_vU(r,Z_r^n)+\mathbf{I}\big)\big(b(r,Z_r^n)-\Gamma_{r-k_n(r)}b(r,Z_{k_n(r)}^n)\big)\dif r\big|\Big\|_{L^{m}(\Omega)},\label{def.Hnt}
\end{align}
for $m\in\mN$.

We first give the following rather general stability estimates for  equation \eqref{eq:SDE}
 and its approximation \eqref{eq:SDE-EM}.
 \begin{proposition}
    For any $m\geq1$, for any $\gamma>1$ so that $\gamma m<p:=\min(p_x,p_v)$, there exists a constant $C=C(\bbp,d,\beta,\epsilon,m, \vartheta, \|b\|_{L^\infty_T(\bB^{2/3,\beta}_{\bbp;x,\bba})})$ so that
    \begin{align}
        \label{est:Z-Zn-imp}
     \big \|\sup_{t\in[0,1]}|Z_t-Z_t^n|\big\|_{L^m(\Omega)}\leq C\big(\|Z_0-Z_0^n\|_{L^m(\Omega)}   +H^n(\gamma m)\big)
    \end{align}
    where $H^n$ is from \eqref{def.Hnt}.
    \end{proposition}
\begin{proof}
    By It\^o's formula we have
\begin{align}\label{eq:b}
    \int_s^tb^n(r,Z_r)\dif r=&U(s,Z_s)-U(t,Z_t)+\int_s^t\nabla_vU(r,Z_r)\dif W_r+\lambda\int_s^t U(r,Z_r)\dif r
    \nonumber\\&+\int_s^t\nabla_vU(r,Z_r)\big(b(r,Z_r)-b_n(r,Z_r)\big)\dif r,
\end{align}
and
\begin{align}\label{eq:bn}
    \int_s^tb^n(r,Z_r^n)=&U(s,Z_s^n)-U(t,Z_t^n)+\int_s^t\nabla_vU(r,Z_r^n)\dif W_r
    \nonumber\\&+
    \int_s^t\nabla_vU(r,Z_r^n)\big(\Gamma_{r-k_n(r)}b^n(r,Z_{k_n(r)}^n)-b^n(r,Z_r^n)\big)\dif r+\lambda\int_s^t U(r,Z_r^n)\dif r.
\end{align}
Plug \eqref{eq:b} and \eqref{eq:bn} into the difference of $Z-Z^n$, for any choice of $s,h\in[0,1]$ so that $(s,s+h)\in[0,1]_\leq^2$,  we have for any $m\geq2$
\begin{align}\label{est:Z-Zn}
    & \sup_{t\in [s,s+h]}|Z_t-Z_t^n|^m
    \nonumber\\\lesssim& \sup_{t\in [s,s+h]}\big|\int_s^tb(r,Z_r)\dif r-  \int_s^t\Gamma_{r-k_n(r)}b^n(r,Z_{k_n(r)}^n)\dif r\big|^m+ \sup_{t\in [s,s+h]}\big|\int_s^tV_r-V_r^n\dif r\big|^m \nonumber\\\lesssim &  \sup_{t\in [s,s+h]}\big|\int_s^tb(r,Z_r)\dif r-  \int_s^tb^n(r,Z_r)\dif r\big|^m+ \sup_{t\in [s,s+h]}\big|\int_s^tb^n(r,Z_r)\dif r-  \int_s^tb^n(r,Z_r^n)\dif r\big|^m
    \nonumber \\& + \sup_{t\in [s,s+h]}\big|\int_s^tb^n(r,Z_r^n)\dif r-  \int_s^t\Gamma_{r-k_n(r)}b^n(r,Z_{k_n(r)}^n)\dif r\big|^m+ \sup_{t\in [s,s+h]}\big|\int_s^tZ_r-Z_r^n\dif r\big|^m
     \nonumber\\\lesssim &  \sup_{t\in [s,s+h]}\big|U(t,Z_t^n)-U(t,Z_t)\big|^m+\big|U(s,Z_s^n)-U(s,Z_s)\big|^m+ \sup_{t\in [s,s+h]}\big|\int_s^tZ_r-Z_r^n\dif r\big|^m
     \nonumber\\&+ \sup_{t\in [s,s+h]}\lambda^m\Big|\int_s^t U(r,Z_r)\dif r-\int_s^t U(r,Z_r^n)\dif r\Big|^m
   \nonumber  \\&+ \sup_{t\in [s,s+h]} \big|\int_s^tb(r,Z_r)\dif r-  \int_0^tb^n(r,Z_r)\dif r\big|^m
 \nonumber\\&+ \sup_{t\in [s,s+h]}\big|\int_s^t\big(\nabla_vU(r,Z_r^n)+\mathbf{I}\big)\big(b(r,Z_r^n)-\Gamma_{r-k_n(r)}b(r,Z_{k_n(r)}^n)\big)\dif r\big|^m \nonumber\\&+ \sup_{t\in [s,s+h]}\big|\int_s^t\big(\nabla_vU(r,Z_r)-\nabla_vU(r,Z_r^n)\big)\dif W_r\big|^m.
\end{align}
Recall $A^n$ from \eqref{def.Ant}, for a constant $\tilde K>0$ to be chosen later, we  set an increasing sequence of stopping times by $\tau^0=0$ and
\begin{align}\label{def:stop}
\tau^{n,l+1}=\inf\big\{t\geq \tau^{n,l}: t\leq 1,\ A_t^n-A_{\tau^{n,l}}^n+\lambda^2\|\nabla U\|_\infty^2+\|\nabla U\|_\infty^2+t- \tau^{n,l}\geq  (4\tilde K)^{-\frac{2}{m}}\big\},
\end{align}
with the convention that $\tau^{n,l+1}=1$ if there is no $t\in [\tau^{n,l},1]$ such that $A_t^n-A_{\tau^ {n,l}}^n=(4\tilde K)^{-2/m}$.
Following from  Doob's optimal stopping theorem, we have
\begin{align*}
 \sM^{n,l}_t:= \int_{\tau^{n,l}}^{\tau^{n,l}+t}\big(\nabla_vU(r,Z_r)-\nabla_vU(r,Z_r^n)\big)\dif W_r
\end{align*}
is a continuous martingale w.r.t. to the filtration $\mathcal{G}_t=\mathcal{F}_{\tau^{n,l}+t}$. So by pathwise BDG’s inequality there exists a martingale $\hat\sM^{n,l}$ so that $a.s.$
\begin{align*}
    \sup_{t\in [0,h]}  \big|  \sM^{n,l}_t\big|^m\leq \hat\sM^{n,l}_h+[\sM^{n,l}]_h^{m/2},
\end{align*}
where $[\sM^{n,l}]$ denotes the quadratic variation of $\sM^{n,l}$;
further by elementary calculus  and the continuity of $\nabla \nabla_vu$ from \eqref{est:naxnavU} we have for any $z_1,z_2\in\mR^{2d}$
\begin{align*}
 \big| \nabla_vU(r,z_1)-\nabla_vU(r,z_2)\big|\lesssim \big(\big|\mathcal{M}(\nabla \nabla_vU)(r,z_1)\big|+\big|\mathcal{M}(\nabla \nabla_vU)(r,z_1)\big|\big)|z_1-z_2|
\end{align*}
where $\mathcal{M}$ denotes the Hardy-Littlewood maximal function from \eqref{def:M}. It yields
$a.s.$
\begin{align*}
[\sM^{n,l}]_h^{m/2}\leq& \big(\int_{\tau^{n,l}}^{\tau^{n,l}+h}\big|\big(\mathcal{M}(\nabla \nabla_vU)\big(Z_r)+\mathcal{M}(\nabla \nabla_vU)(Z_r^n)\big)\cdot(Z_r-Z_r^n)\big|^2\dif r\big)^{m/2}
\\=&\big(\int_{\tau^{n,l}}^{\tau^{n,l}+h}\big|Z_r-Z_r^n\big|^2\dif A_r^n\big)^{m/2},
\end{align*}
therefore the following pathwise inequality holds for the choice of $s:=\tau^{n,l}, h:=\tau^{n,l+1}-\tau^{n,l}$
\begin{align}\label{est:dif}
  \sup_{t\in [\tau^{n,l},\tau^{n,l+1}]}& |Z_t-Z_t^n|^m\nonumber\\\lesssim   &  
  \big( \sup_{t\in [\tau^{n,l},\tau^{n,l+1}]} |Z_t-Z_t^n|^m+|Z_{\tau^{n,l}}-Z_{\tau^{n,l}}^n|^m\big)\|\nabla U\|_\infty^m
    \nonumber\\&+\lambda^m \sup_{t\in [\tau^{n,l},\tau^{n,l+1}]} |Z_t-Z_t^n|^m\|\nabla U\|_\infty^m|\tau^{n,l+1}-\tau^{n,l}|^{\frac{m}{2}}
   \nonumber  \\&+  \sup_{t\in [\tau^{n,l},\tau^{n,l+1}]}\big|\int_0^tb(r,Z_r)\dif r-  \int_0^tb^n(r,Z_r)\dif r\big|^m
 \nonumber\\&+ \sup_{t\in [\tau^{n,l},\tau^{n,l+1}]}\big|\int_0^t\big(\nabla_vU(r,Z_r^n)+\mathbf{I}\big)\big(b(r,Z_r^n)-\Gamma_{r-k_n(r)}b(r,Z_{k_n(r)}^n)\big)\dif r\big|^m
\nonumber\\&+\hat\sM^{n,l}_{\tau^{n,l+1}}+\big(\int_{\tau^{n,l}}^{\tau^{n,l+1}}\big|Z_r-Z_r^n\big|^2\dif A_r^n\big)^{m/2}
\nonumber\\\leq   & 
|Z_{\tau^{n,l}}-Z_{\tau^{n,l}}^n|^m\|\nabla U\|_\infty^m+\hat\sM^{n,l}_{\tau^{n,l+1}}
   \nonumber  \\&+  \sup_{t\in [\tau^{n,l},\tau^{n,l+1}]}\big|\int_0^tb(r,Z_r)\dif r-  \int_0^tb^n(r,Z_r)\dif r\big|^m
 \nonumber\\&+ \sup_{t\in [\tau^{n,l},\tau^{n,l+1}]}\big|\int_0^t\big(\nabla_vU(r,Z_r^n)+\mathbf{I}\big)\big(b(r,Z_r^n)-\Gamma_{r-k_n(r)}b(r,Z_{k_n(r)}^n)\big)\dif r\big|^m
\nonumber\\&+\sup_{t\in [\tau^{n,l},\tau^{n,l+1}]} |Z_t-Z_t^n|^mC\Big(\big(A_{\tau^{n,l+1}}^n-A_{\tau^{n,l}}^n\big)^{m/2}+\|\nabla U\|_\infty^m+\lambda\|\nabla U\|_\infty^m|\tau^{n,l+1}-\tau^{n,l}|^{\frac{m}{2}}\big).
\end{align}
Taking $C=\tilde K$ as the desired constant in the definition of $\tau^n$ \eqref{def:stop}, we find
\begin{align*}
    \sup_{t\in [\tau^{n,l},\tau^{n,l+1}]}& |Z_t-Z_t^n|^m\nonumber\\\leq &\hat\sM^{n,l}_{\tau^{n,l+1}}+2  \tilde K \big(|Z_{\tau^{n,l}}-Z_{\tau^{n,l}}^n|^m+ \nonumber  \\&+  \sup_{t\in [\tau^{n,l},\tau^{n,l+1}]}\big|\int_0^tb(r,Z_r)\dif r-  \int_0^tb^n(r,Z_r)\dif r\big|^m
 \nonumber\\&+ \sup_{t\in [\tau^{n,l},\tau^{n,l+1}]}\big|\int_0^t\big(\nabla_vU(r,Z_r^n)+\mathbf{I}\big)\big(b(r,Z_r^n)-\Gamma_{r-k_n(r)}b(r,Z_{k_n(r)}^n)\big)\dif r\big|^m\big)
 \\=&:\hat\sM^{n,l}_{\tau^{n,l+1}}+2  \tilde K|Z_{\tau^{n,l}}-Z_{\tau^{n,l}}^n|^m+ 2  \tilde K \sV.
\end{align*}
Since $\hat\sM^{n,l}_t$ is a $\mathcal{G}_t$-martingale, taking expectation to both side of above gives us
\begin{align*}
    \mE  \sup_{t\in [\tau^{n,l},\tau^{n,l+1}]}& |Z_t-Z_t^n|^m\leq 2  \tilde K\mE|Z_{\tau^{n,l}}-Z_{\tau^{n,l}}^n|^m+ 2  \tilde K \mE\sV.
\end{align*}
Iteratively for an appropriately chosen constant $\kappa>0$ in function of $\tilde K$, we get
\begin{align*}
    \mE \big(e^{-l\kappa}  \sup_{t\in [\tau^{n,l},\tau^{n,l+1}]}  |Z_t-Z_t^n|^m\big)\leq  2  \tilde K\mE|Z_{0}-Z_{0}^n|^m+ 2  \tilde K \mE\sV.
\end{align*}
Let \begin{align*}
    B_t:=A_t^n+\lambda^2\|\nabla U\|_\infty^2+\|\nabla U\|_\infty^2t+t,\quad \nu:=(4 \tilde K)^{-2/m}, C:= 2\kappa/ \nu.
\end{align*}
Then we have
\begin{align*}
    \mE\Big[ e^{-C B_1} \sup_{t\in [0,1]} |Z_t-Z_t^n|^m \Big]
& = \sum_{l\in \mathbb{N}} \mE\Big[ e^{-C B_1} \sup_{t\in [0,1]} |Z_t-Z_t^n|^m \mathbbm{1}_{B_1\in [\nu l, \nu(l+1)]} \Big]\\
& \leq \sum_{l\in \mathbb{N}} \E \Big[ e^{-2\kappa l} \sup_{j=1,\ldots,l+1} \sup_{t\in [\tau^{n,l-1},\tau^{n,l}]} |Z_t-Z_t^n|^m \mathbbm{1}_{B_1\in [\nu l, \nu(l+1)]}\Big]\\
& \leq \sum_{l\in\mathbb{N}} e^{-\kappa l} \sum_{j=1}^{l+1} \E\Big[ e^{-\kappa l} \sup_{t\in [\tau^{n,l-1},\tau^{n,l}]} |Z_t-Z_t^n|^m\Big]\\
& \lesssim  |Z_0-Z_0^n|^m+\mE\sV.
\end{align*}
Notice that the process $B$ satisfies $\E[\exp(\lambda B)]<\infty$ for all $\lambda\in \R$; therefore for any $\delta\in (0,1)$ it holds
\begin{align*}
\mE\Big[\sup_{t\in [0,1]} |Z_t-Z_t^n|^{\delta m} \Big]
& \leq \mE\Big[e^{-C B_1} \sup_{t\in [0,1]} |Z_t-Z_t^n|^{m} \Big]^{\delta}\, \mE\Big[e^{C \frac{\delta}{1-\delta} B_1}\Big]^{1-\delta}\\
& \lesssim_\delta |Z_0-Z_0^n|^m +\E[\sV]^\delta.
\end{align*}
Recalling \eqref{def.Hnt} and taking the $(\delta m)^{-1}$-power on both sides and relabelling $\tilde{m}=m\delta$, $\gamma \tilde{m} =m$, overall yields the desired result.
\end{proof}
Finally we are ready to show the main result of this session.
\begin{proof}
    [Proof of \cref{thm-strong}]  First by \eqref{0321:02} and \cref{lem:f-fn} \eqref{est:f-fn-app},  we have 
    and for $\delta\in(-1+\bba\cdot(d/\bbp),0)$ so that $0<\beta-\delta\leq 1$
    \begin{align}
        \label{est:Kry-Z-Zn}
        \Big\|\sup_{t\in[0,1]}\big|  \int_0^tb(r,Z_r)\dif r-  \int_0^tb^n(r,Z_r)\dif r\big|\Big\|_{L^{m}(\Omega)} &\lesssim \|b-b_n\|_{L^\infty_T(\bB^{\delta}_{\bbp;\bba})}
       \nonumber\\&\lesssim\|b\|_{L^\infty_T(\bB^{\beta}_{\bbp;\bba})}n^{-\vartheta(\beta-\delta)}
\nonumber\\&\lesssim \|b\|_{\bB^{2/3,\beta}_{\bbp;x,\bba}}n^{-\vartheta(\beta-\delta)}.
    \end{align}
Secondly following from \eqref{est:naxnavU} we know that for $\epsilon\in(0,\frac{1}{2})$, there exists $\beta’\in(1-2\epsilon,1)$ so that $\nabla_vU\in \mathbb{C}_\bba^{\beta’}$. Now we apply \cref{cor:final} to estimate \eqref{def.Hnt}. We then together with \eqref{est:Kry-Z-Zn} and \eqref{AA05}  obtain that for $H^n$ defined from \eqref{def.Hnt}, for any $\gamma m< p:=\min(p_x,p_v)$ with $\gamma>1$
\begin{align*}
 H^n(\gamma m) \lesssim \|b-b_n\|_{L^\infty_T(\bB^{\delta}_{\bbp;\bba})}+n^{-\frac{1+\beta/3}{2}+\epsilon} \lesssim n^{-\vartheta(\beta+1-\bba\cdot\frac{d}{\bbp})+\epsilon}+n^{-\frac{1+\beta/3}{2}+\epsilon}.
\end{align*}
The proof completes after plugging the above estimate into \eqref{est:Z-Zn-imp}. 
\end{proof}

\appendix
\section{Several technical lemmas}\label{app}
First of all, we give the following elementary lemma similar as \cite[Lemma A.2]{HRZ}.
\begin{lemma}\label{lem:A1}
    Let $\alpha,\beta\in(0,1)$ and $\gamma\ge1$. There there is a constant $C=C(\alpha,\beta,\gamma)>0$ such that for all $t,\lambda>0$,
    \begin{align*}
        \int_0^t s^{-\alpha}(t-s)^{-\beta}\left([(\lambda(t-s))^{-\gamma}]\wedge 1\right)\le C \lambda^{-1+\beta}t^{-\alpha}.
    \end{align*}
\end{lemma}
\begin{proof}
    Based on a change of variable, we have
    \begin{align*}
        I(t,\lambda):=\int_0^t s^{-\alpha}(t-s)^{-\beta}\left([(\lambda(t-s))^{-\gamma}]\wedge 1\right)\dif s=t^{1-\alpha-\beta}I(1,t\lambda).
    \end{align*}
    Thus, we only estimate $I(1,\lambda)$. We have
    \begin{align*}
        I(1,\lambda)=&\left(\int_0^{1/2}+\int_{1/2}^1\right)s^{-\alpha}(1-s)^{-\beta}\left([(\lambda(1-s))^{-\gamma}]\wedge 1\right)\dif s\\
        \lesssim& [\lambda^{-\gamma}\wedge 1]\int_0^{1/2}s^{-\alpha}\dif s+\int_{1/2}^1(1-s)^{-\beta}\left([(\lambda(1-s))^{-\gamma}]\wedge 1\right)\dif s\\
        \lesssim& [\lambda^{-\gamma}\wedge 1]+\int_{0}^1s^{-\beta}\left([(\lambda s)^{-\gamma}]\wedge 1\right)\dif s\lesssim [\lambda^{-\gamma}\wedge 1]+\lambda^{-1+\beta}\lesssim\lambda^{-1+\beta}
    \end{align*}
    and complete the proof.
\end{proof}

Recall the Bony decomposition (\cite{HZZZ22})
\begin{align}\label{def:Bony}
    fg=f\prec g+f\circ g+f\succ g=:f\prec g+f\succcurlyeq g,
\end{align}
where
\begin{align*}
   f\prec g=g\succ f:=\sum_{k=1}^\infty S_{k-1}^\bba f\cR_j^\bba g,\quad S_k^\bba f:=\sum_{j=0}^k\cR_j^\bba f,
\end{align*}
and
\begin{align*}
   f\circ g:=\sum_{|i-j|\le2} \cR_i^\bba f\cR_j^\bba g.
\end{align*}
\begin{lemma}[Paraproduct estimates]\label{lem:A2}
    Let $\bbp,\bbq,\bbr\in[1,\infty]^2$ with $1/\bbp+1/\bbq=1/\bbr$ and $\alpha>0>\beta$. Then
\begin{align}\label{0426:09}
    \|f\prec g\|_{\bB^{-\alpha}_{\bbr;\bba}}\lesssim \|g\|_{\bB^{-\alpha}_{\bbp;\bba}}\|f\|_{\bbq}
\end{align}
and
\begin{align}\label{0426:08}
    \|f\succ g\|_{\bB^{\alpha+\beta}_{\bbr;\bba}}\lesssim \|f\|_{\bB^{\alpha}_{\bbp;\bba}}\|g\|_{\bB^{\beta}_{\bbq;\bba}}.
\end{align}
When $\alpha+\beta>0$,
\begin{align}\label{0426:10}
    \|f\circ g\|_{\bB^{\alpha+\beta}_{\bbr;\bba}}\lesssim \|f\|_{\bB^{\alpha}_{\bbp;\bba}}\|g\|_{\bB^{-\alpha}_{\bbq;\bba}}.
\end{align}
Moreover,
\begin{align}\label{0426:07}
    \|f\circ g\|_{\bbr}\lesssim \|f\|_{\bB^{\alpha,1}_{\bbp;\bba}}\|g\|_{\bB^{-\alpha}_{\bbq;\bba}}.
\end{align}
\end{lemma}
\begin{proof}
    The estimates for \eqref{0426:09}, \eqref{0426:08} and \eqref{0426:10} are standard (see \cite[Lemma 2.11]{HZZZ22} for instance). We only show \eqref{0426:07}. In fact, by definition and H\"older's inequality, we have
    \begin{align*}
        \|f\circ g\|_{\bbr}\lesssim \sum_{i=0}^\infty\sum_{\ell=-2}^2\|\cR_{i}^\bba f\|_{\bbp}\|\cR_{i+\ell}^\bba g\|_{\bbq}\lesssim \sum_{i=0}^\infty2^{\alpha i}\|\cR_{i}^\bba f\|_{\bbp}\|g\|_{\bB^{-\alpha}_{\bbq;\bba}}\lesssim \|f\|_{\bB^{\alpha,1}_{\bbp;\bba}}\|g\|_{\bB^{-\alpha}_{\bbq;\bba}}.
    \end{align*}
    \end{proof}
    \begin{lemma}
        \label{lem:f-fn}
        Let $f$ be a measurable function defined on $\mR^{2d}$, define $f_n:=f\ast\varphi_n$ where $\varphi_n(x,v):=n^{4d\vartheta}\varphi(n^{3\vartheta}x, n^{\vartheta}v),$ $\vartheta\geq0$ and $\varphi$ is a probability density function on $\mR^{2d}$. Then for any $s\in\mR$, $\beta\in(0,1]$, $\bbp\in[1,\infty)^2$,
         we have
         \begin{align}\label{est:f-fn-app}
             \|f-f_n\|_{\bB^{s}_{\bbp;\bba}}\lesssim n^{-\vartheta \beta} \|f\|_{\bB^{s+\beta}_{\bbp;\bba}}.
         \end{align}
    \end{lemma}
\begin{proof}
    By \cref{bs}, it reads
    \begin{align*}
      \|f-f_n\|_{\bB^{s}_{\bbp;\bba}}=\sup_j2^{sj}   \|\cR_j^\bba(f-f_n)\|_{\bbp}.
    \end{align*}
    Notice that, since $\varphi_n$ is again a probability density function, for any $(x,v)\in\mR^{2d}$
    \begin{align*}
        \cR_j^\bba(f-f_n)(x,v)=\int_{\mR^{2d}}\varphi_n(y,w)\big( \cR_j^\bba f(x,v)-\cR_j^\bba f(x-y,v-w)\big)\dif y\dif w.
    \end{align*}
    It implies
    \begin{align*}
        \| \cR_j^\bba(f-f_n)\|_\bbp\lesssim& \int_{\mR^{2d}}|\varphi_n(y,w)|(|y|^\frac{\beta}{3}+|w|^\beta)\| \cR_j^\bba f\|_{_{\bB^{\beta}_{\bbp;\bba}}}\dif y\dif w
        \\\lesssim&\| \cR_j^\bba f\|_{\bB^{\beta}_{\bbp;\bba}} n^{-\vartheta \beta}
         =\sup_k2^{k\beta}   \|\cR_j^\bba f\|_{\bbp} n^{-\vartheta \beta}.
    \end{align*}
    In then end we get
     \begin{align*}
      \|f-f_n\|_{\bB^{s}_{\bbp;\bba}}=\sup_j2^{sj}   \|\cR_j^\bba(f-f_n)\|_{\bbp}\lesssim \sup_j2^{(s+\beta)j}   \|\cR_j^\bba f\|_{\bbp}n^{-\vartheta \beta}.
    \end{align*}
\end{proof}
\section*{Acknowledgment}
KL has been funded by the Engineering \& Physical Sciences Research Council (EPSRC) Grant EP/W523860/1.
We are thankful to Prof. Fengyu Wang (Tianjin University, China) who suggested this question to us.

\end{document}